\documentclass[reqno,11pt]{article}
\usepackage{a4wide}
\usepackage{xcolor,eucal,enumerate,mathrsfs}
\usepackage{pdfsync}
\usepackage{amsmath,amssymb,epsfig,bbm,amsthm}
\usepackage[normalem]{ulem}
\usepackage[pdfborder={0 0 0}]{hyperref}  
\usepackage[latin1]{inputenc}

\numberwithin{equation}{section}

\newcommand{\C}{\mathbb{C}}
\newcommand{\D}{\mathbb{D}}

\newcommand{\G}{\mathbb{G}}

\newcommand{\M}{\mathbb{M}}
\newcommand{\N}{\mathbb{N}}

\newcommand{\R}{\mathbb{R}}
\newcommand{\Z}{\mathbb{Z}}

\newcommand{\BB}{\mathscr{B}}

\newcommand{\cE}{{\ensuremath{\mathcal E}}}

\newcommand{\cX}{{\ensuremath{\mathcal X}}}
\newcommand{\cY}{{\ensuremath{\mathcal Y}}}
\newcommand{\cV}{{\ensuremath{\mathcal V}}}

\renewcommand{\gg}{{\mbox{\boldmath$g$}}}

\newcommand{\ii}{{\mbox{\boldmath$i$}}}

\newcommand{\mm}{{\mbox{\boldmath$m$}}}
\newcommand{\nn}{{\mbox{\boldmath$n$}}}

\renewcommand{\tt}{{\mbox{\boldmath$t$}}}

\newcommand{\xx}{{\mbox{\boldmath$x$}}}

\newcommand{\aalpha}{{\mbox{\boldmath$\alpha$}}}

\newcommand{\ggamma}{{\mbox{\boldmath$\gamma$}}}

\newcommand{\mmu}{{\mbox{\boldmath$\mu$}}}

\newcommand{\ppi}{{\mbox{\boldmath$\pi$}}}

\newcommand{\saalpha}{{\mbox{\scriptsize\boldmath$\alpha$}}}

\newcommand{\sggamma}{{\mbox{\scriptsize\boldmath$\gamma$}}}

\newcommand{\sfc}{{\sf c}}
\newcommand{\sfd}{{\sf d}}

\newcommand{\rme}{{\mathrm e}}

\newcommand{\rmp}{{\mathrm p}}

\newcommand{\rmC}{{\mathrm C}}
\newcommand{\rmD}{{\mathrm D}}

\newcommand{\Kliminf}{K\kern-3pt-\kern-2pt\mathop{\rm lim\,inf}\limits}  
\newcommand{\supp}{\mathop{\rm supp}\nolimits}   
\newcommand{\argmin}{\mathop{\rm argmin}\limits}   
\newcommand{\Lip}{\mathop{\rm Lip}\nolimits}          
\renewcommand{\d}{{\mathrm d}}
\newcommand{\dt}{{\d t}}

\newcommand{\restr}[1]{\lower3pt\hbox{$|_{#1}$}}

\newcommand{\Leb}[1]{{\mathscr L}^{#1}}      

\newcommand{\down}{\downarrow}              
\newcommand{\up}{\uparrow}
\newcommand{\eps}{\varepsilon}  
\newcommand{\nchi}{{\raise.3ex\hbox{$\chi$}}}

\newcommand{\forevery}{\text{for every }}

\newcommand{\Pc}[2]{\overline{#1}\kern-2pt^{\vphantom 0}_{#2}}

\newcommand{\CD}{\ClosureDomain}

\newcommand{\BorelSets}[1]{\BB(#1)}

\newcommand{\Probabilities}[1]{\mathscr P(#1)}          
\newcommand{\ProbabilitiesTwo}[1]{\mathscr P_2(#1)}     

\newtheorem{theorem}{Theorem}[section]

\newtheorem{corollary}[theorem]{Corollary}
\newtheorem{lemma}[theorem]{Lemma}
\newtheorem{proposition}[theorem]{Proposition}

\newtheorem{definition}[theorem]{Definition}

\newtheorem{example}[theorem]{Example}

\newtheorem{remark}[theorem]{Remark}
\theoremstyle{remark}
\newtheorem*{acknowledgement}{Acknowledgement}

\renewcommand{\mm}{\mathfrak m}

\newcommand{\tmm}{\tilde{\mathfrak m}}
\newcommand{\tmmn}{\tilde{\mathfrak m}_n}
\renewcommand{\nn}{\mathfrak n}

\newcommand{\Ent}{{\rm Ent}}
\newcommand{\Enttn}{{\rm Ent}_{\tilde{\mm}_n}}
\newcommand{\Entt}{{\rm Ent}_{\tilde{\mm}}}
\newcommand{\Entn}{{\rm Ent}_{{\mm}_n}}
\newcommand{\Enti}{{\rm Ent}_{{\mm}_\infty}}

\newcommand{\Entti}{{\rm Ent}_{{\tilde\mm}_\infty}}
\newcommand{\entv}{{\rm Ent}_{\mm}}                    
\newcommand{\prob}{\Probabilities}
\newcommand{\probt}{\ProbabilitiesTwo}

\newcommand{\geo}{{\rm{Geo}}}                       
\newcommand{\e}{{\rm{e}}}                           
\newcommand{\Opt}{{\rm{Opt}}}
\newcommand{\Adm}{{\rm{Adm}}}

\newcommand{\fr}{\hfill$\blacksquare$}                      

\newcommand{\Dm}{{\rm D}}

\newcommand{\weakgrad}[1]{|\rmD #1|_w}                

\renewcommand{\C}{{\sf Ch}}
\newcommand{\s}{{\rm S}}
\newcommand{\z}{{\rm z}}

\newcommand{\EEE}{\normalcolor}

\newcommand{\sfH}{\mathsf H}

\newcommand{\AC}[3]{\mathrm{AC}^{#1}(#2;#3)}

\newcommand{\limi}{\varliminf}
\newcommand{\lims}{\varlimsup}
\newcommand{\X}{\mathbb X}
\newcommand{\Co}{{\sf C}}
\newcommand{\narrowly}{weakly}
\newcommand{\narrow}{weak}
\newcommand{\dualspace}[2]{#1[#2]}

\DeclareMathOperator*{\Glimi}{\Gamma-\limi}
\DeclareMathOperator*{\Glims}{\Gamma-\lims}

\newcommand{\heatw}{{\mathscr H}}
\newcommand{\h}{\heatw}

\newcommand{\Mloc}[1]{\mathscr M_{loc}(#1)}

\renewcommand{\M}[1]{\mathscr M(#1)}

\newcommand{\Cc}[1]{\rmC_{bs}(#1)}
\newcommand{\Cb}[1]{\rmC_b(#1)}
\newcommand{\mui}{\mu_\infty}
\newcommand{\pmm}{p.m.m.}
\newcommand{\pmmaxioms}{{\rm (\ref{pointed}a,b)}}

\newcommand{\Dmeas}[1]{\D^{#1}}

\newcommand{\pGw}{\rm pmG}

\newcommand{\pmmX}{(X,\sfd,\mm,\bar x)}
\newcommand{\pmmXclass}{[X,\sfd,\mm,\bar x]}
\newcommand{\pmmXcseq}[1]{[X_{#1},\sfd_{#1},\mm_{#1},\bar x_{#1}]}

\newcommand{\pmmXa}[1]{(X_{#1},\sfd_{#1},\mm_{#1},\bar x_{#1})}
\newcommand{\pmmXclassa}[1]{[X_{#1},\sfd_{#1},\mm_{#1},\bar x_{#1}]}

\newcommand{\pWl}{\rmp\widetilde W_{\kern-2pt\sfc}}
\newcommand{\tWc}{\widetilde W_{\kern-2pt\sfc}}
\newcommand{\Wc}{W_{\kern-2pt\sfc}}
\newcommand{\Wcs}{W_{\kern-2pt\sfc}^\star}
\newcommand{\pWc}{\widetilde W_{\kern-2pt \sfc}}
\newcommand{\Ecurve}[1]{\cE_2[#1]}

\renewcommand{\CD}{\mathrm{CD}}
\newcommand{\RCD}{\mathrm{RCD}}
\newcommand{\Y}{X}
\newcommand{\YY}{Y}
\newcommand{\GW}{\mathbb{G}_{\rm W}}
\newcommand{\pGW}{\rmp\mathbb{G}_{\rm W}}

\renewcommand{\G}{\mathbb{G}}
\newcommand{\jey}{j}

\title{Convergence of pointed non-compact metric measure spaces and
  stability of Ricci curvature bounds and
  heat flows}

\begin{document}

\author{Nicola Gigli\
   \thanks{\textsf{giglin@math.jussieu.fr}}
 \and
 Andrea Mondino\
 \thanks{\textsf{andrea.mondino@math.ethz.ch }}
   \and
   Giuseppe Savar\'e\
   \thanks{\textsf{giuseppe.savare@unipv.it}}
   }

\maketitle

\begin{abstract}
Aim of this paper is to discuss convergence of 
 pointed metric measure spaces  in absence of any compactness condition. We propose various definitions, show that all of them are equivalent and that for doubling spaces these are also equivalent to the well known measured-Gromov-Hausdorff convergence.

Then we show that the curvature conditions $\CD(K,\infty)$ 
and  $\RCD(K,\infty)$ are stable under this notion of 
convergence and that the heat flow passes to the limit as well, both in the Wasserstein and in the $L^2$-framework.
We also prove the variational convergence of Cheeger energies
in the naturally  adapted $\Gamma$-Mosco sense and
the convergence of the spectra of the 
Laplacian in the case of spaces either uniformly bounded or satisfying the $\RCD(K,\infty)$ condition with $K>0$. When applied to Riemannian manifolds, our results allow for sequences with diverging dimensions.
\end{abstract}

\tableofcontents

\section{Introduction}

The notion of convergence of metric measure spaces has been introduced by Fukaya in \cite{Fukaya87} as the natural variant of Gromov-Hausdorff convergence for metric structures endowed with a reference measure. Fukaya's interest was to study the behavior of eigenvalues of the Laplacian on smooth Riemannian manifolds with some uniform curvature bound under a zeroth-order convergence: it turned out that the purely metric notion of Gromov-Hausdorff convergence was not enough to obtain reasonable results, and the additional requirement of convergence of volume measures was needed.

Since then the notion has been widely used, in particular in connection with lower Ricci curvature bounds, and proved to be useful to handle also non smooth limits of Riemannian manifolds (see for instance the series of papers by Cheeger-Colding \cite{Cheeger-Colding97I},\cite{Cheeger-Colding97II},\cite{Cheeger-Colding97III}).

Almost ten years ago, Lott-Villani on one side \cite{Lott-Villani09} and Sturm on the other \cite{Sturm06I}, \cite{Sturm06II} introduced an abstract notion of lower Ricci curvature bound on metric measure setting. The notion introduced is that of $\CD(K,N)$ space, meaning of a space with Ricci Curvature bounded from below by $K\in\R$ and Dimension bounded above by $N\in[1,\infty]$ (in \cite{Lott-Villani09}, for $N<\infty$ only the case $K=0$ was considered). A key feature of their definition is that it is stable w.r.t.~convergence of the spaces. From the technical point of view, there are some differences between the  presentations of such stability result in the two approaches, although on compact and doubling metric measure spaces they produce the same convergence.
\begin{itemize}
\item Lott and Villani worked with proper (i.e. bounded closed sets are compact) pointed metric measure spaces, and showed that the $\CD(K,N)$ condition is stable under pointed measured Gromov-Hausdorff convergence. This means, roughly said, that for any $R>0$ there is measured Gromov-Hausdorff convergence of balls of radius $R$ around the given points of the spaces (see Definition \ref{def:pmGH} for the precise statement).
\item Sturm worked with normalized metric measure spaces with finite variance, i.e. spaces such that the reference measure is a probability measure with finite second moment. He then defined a distance $\D$ on this class of spaces putting 
\[
\D\Big((X_1,\sfd_1,\mm_1),(X_2,\sfd_2,\mm_2)\Big):=\inf W_2\big((\iota_1)_\sharp\mm_1,(\iota_2)_\sharp\mm_2\big),
\]
the infimum being taken among all metric spaces $(X ,\sfd)$ and all isometric embeddings $\iota_1:(\supp(\mm_1),\sfd_1)\to (X ,\sfd)$ and $\iota_2:(\supp(\mm_2),\sfd_2)\to (X ,\sfd )$. He then proved that Curvature-Dimension bounds are stable w.r.t.~$\D$-convergence.  
\end{itemize}
Given that $\CD(K,N)$ spaces with $N<\infty$ are always proper, the approach of Lott-Villani fully covers this situation. On the other hand, the $\CD(K,\infty)$ condition does not imply any sort of compactness, not even local, and hence to work with (pointed) measured Gromov-Hausdorff convergence is quite unnatural in this case. At least for spaces with finite variance, Sturm's $\D$ distance is suitable to discuss stability of these curvature bounds, as no compactness is required in order to work with it.

\bigskip

Aim of this paper is dual: on one side we propose a notion of
convergence (that we call \emph{pointed  measured Gromov  convergence},
$\pGw$-convergence in short) 
of metric measure structures which works without any
compactness assumptions on the metric structure and for possibly
$\sigma$-finite measures, on the other we prove that lower Ricci bounds and related constructions are stable w.r.t.\ this convergence. 

The structures we shall work with are then isomorphism classes of \emph{pointed metric measure} (\pmm)
spaces
$\pmmX$, where
\[
  \begin{gathered}
    \text{$(X,\sfd)$ is a complete and separable metric space, $\mm$ is a nonnegative and nonzero
   }\\ 
  \text{Borel measure
    which is finite on bounded sets, and $\bar x$ is
    a point in $\supp(\mm)$,}
  \end{gathered}
\]
and $(X_1,\sfd_1,\mm_1,\bar x_1)$ is said isomorphic to $(X_2,\sfd_2,\mm_2,\bar x_2)$ provided there exists 
\[
\text{an isometric embedding\quad $\iota:\supp(\mm_1)\to X_2$ such that $\iota_\sharp\mm_1=\mm_2$ and $\iota(\bar x_1)=\bar x_2$.}
\]
Our main results are
\begin{itemize}
\item[i)] to propose various a priori different notions of convergence and to show that they produce the same converging sequences, thereby called pmG-convergence (Section \ref{se:nozioni} and Theorem \ref{thm:main_convergence}),
\item[ii)] to prove that on doubling spaces, $\pGw$-convergence is the same as pointed measured Gromov-Hausdorff convergence (Propositions \ref{prop:pmghD} and \ref{prop:Dpmgh}), see also below for a brief description of these notions,
\item[iii)] to show that the curvature condition $\CD(K,\infty)$ is stable w.r.t.~$\pGw$-convergence (Theorem \ref{thm:stabilityCD}). 
\end{itemize}  
On top of this, we also prove that
\begin{itemize}
\item[iv)]     the heat flow, defined as the Wasserstein gradient flow of the relative
    entropy or as the $L^2$-gradient flow of the Cheeger energy, 
  is stable w.r.t.~$\pGw$-convergence (Theorems
  \ref{thm:stabgf1} and \ref{thm:L2conv}).
\item[v)]  The Cheeger energies of $\pGw$-converging \pmm~spaces
  are stable with respect to  (a suitably adapted notion of) Mosco-convergence (Theorem \ref{thm:Mosco}). This also provides convergence of the spectrum of the Laplacian, under further conditions that ensure its discreteness.
\end{itemize}  
Due to the discussion above, these results are easier under additional compactness assumptions (see e.g.\ \cite{Gigli10} for the
proof of the stability of the gradient flow of the entropy in the
compact case), 
but part of  the added value of this paper is precisely to show that it is 
possible to work out the severe technical obstructions one encounters
in working  without any a priori compactness and to extend the stability issues from the Wasserstein to the $L^2$-setting.

We underline that iv) above is particularly important when
considering the class of spaces with Riemannian Ricci curvature
bounded from below by $K$, $\RCD(K,\infty)$ spaces for short,
introduced in \cite{Ambrosio-Gigli-Savare-preprint11b} and further
analyzed in \cite{AGMRS12}. The class of $\RCD(K,\infty)$ spaces is
the subclass of $\CD(K,\infty)$ spaces where the gradient flow of the
entropy linearly depends on the initial datum or, equivalently, where the Cheeger energy is a quadratic Dirichlet form.
It has been introduced to somehow isolate the `Riemannian-like'
structures in the $\CD(K,\infty)$ class. 

It turns out that 
\begin{itemize}
\item[vi)] the $\RCD(K,\infty)$ condition is also stable
  w.r.t.~$\pGw$-convergence 
  (Theorem \ref{thm:stabrcd}),
\item[vii)] in this case finer convergence properties of the gradient
  flow of the entropy can be proved (Theorem \ref{thm:stabrgf}),
\item[viii)]   whenever 
  spaces 
 satisfying a uniform weak 
  Logarithmic-Sobolev-Talagrand inequality are
  considered
  (e.g.~when $K>0$ or 
  the diameters are  uniformly bounded), the spectrum of the (linear) Laplace operator
  is also stable (Theorem \ref{thm:Lap-stab}).
\end{itemize}  
We conclude this introduction by informally describing the 
notions of convergence we introduce, referring to the body of the paper for the more technical statements about lower Ricci bounds and heat flows.

\smallskip

\noindent{\bf (A)} \underline{Intrinsic approach: use of Gromov reconstruction theorem.} Given an integer $N\in\N$ and a continuous function with bounded support $\varphi:\R^{N\times N}\to\R$ we can consider the ``cylindrical'' function  on the class of isomorphism classes $\cX$ of p.m.m.\ spaces $(X,\sfd,\mm,\bar x)$ defined as
\[
  \dualspace\varphi\cX:=\int_{X^{N}}
  \varphi\big((\sfd(x_i,x_j))_{ij=1}^N\big)
  \,\d\delta_{\bar x}(x_1)\,\d\mm^{\otimes (N-1)}(x_2,\cdots,x_N).
  \]
Adapting a result due to Gromov for the case of spaces with finite mass, we shall see that $\cX_1=[X_1,\sfd_1,\mm_1,\bar x_1]$ and  $\cX_2=[X_2,\sfd_2,\mm_2,\bar x_2]$ are isomorphic if  and only if
\[
\varphi[\cX_1]=\varphi[\cX_2],
\]
for any $N\in\N$ and $\varphi$ as above, see Proposition \ref{thm:Grecon}.

This result suggests to declare that the sequence $(\cX_n)$ intrinsically converges to $\cX_\infty$ provided 
\[
\varphi[\cX_n]\longrightarrow\varphi[\cX_\infty],
\]
as $n\to\infty$ for every $N\in\N$ and $\varphi$ as above (Definition \ref{def:intrconv}).

\smallskip

\noindent{\bf (B)} \underline{Extrinsic approach: embedding everything in a common space.} A different, and to some extent more direct, way to speak about convergence of $\cX_n=[X_n,\sfd_n,\mm_n,\bar x_n]$ to $\cX_\infty=[X_\infty,\sfd_\infty,\mm_\infty,\bar x_\infty]$ is to ask for the existence of a complete and separable space $(X,\sfd)$ and isometric embeddings $\iota_n:X_n\to X$, $n\in\bar\N$, in such a way that
\[
\int \varphi\,\d(\iota_n)_\sharp\mm_n\longrightarrow\int\varphi\,\d(\iota_\infty)_\sharp\mm_\infty\quad \text{ for every $\varphi:X\to\R$ continuous with bounded support.}
\]
We shall call this notion extrinsic convergence, as it relies on the additional data $(X,\sfd)$ and $(\iota_n)$ (Definition \ref{def:Dconv}). Notice that a priori it is not clear that this definition is well-posed at the level of isomorphism classes, given that the $\iota_n$'s are required to be defined on the whole $X_n$'s rather than on $\supp(\mm_n)$. Yet,  in practice this is not an issue: in the simple Proposition \ref{le:extension} we shall see that the domain of an isometry  can always be extended provided we take the freedom of extending also the target space.

\smallskip

\noindent{\bf (C)} \underline{A variant of Sturm's distance $\D$.} As previously mentioned, in \cite{Sturm06I} Sturm worked with metric measure spaces $(X,\sfd,\mm)$ such that $\mm\in\probt X$ and introduced a distance $\D$ between their isomorphism classes whose relation to the quadratic transportation distance $W_2$ is essentially the same that the Gromov-Hausdorff distance has to the Hausdorff distance: he posed
\[
\D\big((X_1,\sfd_1,\mm_1),(X_2,\sfd_2,\mm_2)\big):=\inf W_2((\iota_1)_\sharp\mm_1,(\iota_2)_\sharp\mm_2),
\]
the infimum being taken among all spaces $(X,\sfd)$ and isometric embeddings $\iota_i:X_i\to X$, $i=1,2$.

If one is working with p.m.m.\ spaces with  controlled volume growth, it is possible to adapt the definition of $\D$ in the following way. Say that $\psi:[0,\infty)\to(0,\infty)$ is a given continuous non-increasing map and denote by $\X^\psi$ the set of equivalence classes of spaces $(X,\sfd,\mm,\bar x)$ such that
\[
\int (1+\sfd^2(\cdot,\bar x))\psi(\sfd(\cdot,\bar x))\,\d\mm<\infty.
\]
Then we can define the distance $\D^\psi$ between $\cX_1,\cX_2\in \X^\psi$ as 
\[
\D^\psi\big(\cX_1,\cX_2\big):= \big|\log\frac{\z_{1,\psi}}{\z_{2,\psi}}\big|+\inf W_2((\iota_1)_\sharp\mm_{1,\psi},(\iota_2)_\sharp\mm_{2,\psi})+\sfd(\iota_1(\bar x_1),\iota_2(\bar x_2)),
\]
where $\cX_i=[X_i,\sfd_i,\mm_i,\bar x_i]$ and 
\[
\z_{i,\psi}:=\int\psi(\sfd_i(\cdot,\bar x_i))\,\d\mm_i,\qquad\text{ and }\qquad\mm_{i,\psi}:=\frac1{\z_{i,\psi}}\psi(\sfd_i(\cdot,\bar x_i))\,\d\mm_i,
\]
for $i=1,2$ and the inf is taken among all spaces $(X,\sfd)$ and isometric embeddings $\iota_i:X_i\to X$, $i=1,2$, see Definition \ref{def:sturmvar}.

\smallskip

\noindent{\bf (D)} \underline{A metric independent of chosen weights: the pointed Gromov-Wasserstein  distance}

In order to introduce a metric independent of a given weight we shall proceed as follows. First we fix a non-constant bounded concave function $c:[0,\infty)\to[0,\infty)$ with $c(0)=0$ and notice that given a metric space $(X,\sfd)$, the map $(x,y)\mapsto c(\sfd(x,y))$ is a bounded distance inducing the same topology. Then we introduce the 1-transportation distance $W_c$ built on top of it:
\[
W_c(\mu,\nu):=\inf\int c(\sfd(x,y))\,\d\ggamma(x,y),\qquad\text{ the inf being taken among all transport plans $\ggamma$.}
\]
Now for given (equivalence classes of) p.m.m.\ spaces $\cX_1=[X_1,\sfd_1,\mm_1,\bar  x_1]$ and $\cX_2=[X_2,\sfd_2,\mm_2,\bar x_2]$ with finite mass we can define the distance
\[
\pGW^{fm}(\cX_1,\cX_2):=\Big|\log\frac{\mm_1(X_1)}{\mm_2(X_2)}\Big|+\inf \sfd(\iota_1(\bar x_1),\iota_2(\bar x_2))+W_c((\iota_1)_\sharp\tilde\mm_1,(\iota_2)_\sharp\tilde\mm_2),
\]
where $\tilde\mm_i:=\mm_i(X_i)^{-1}\mm_i$, $i=1,2$, and the inf is taken as before among all spaces $(X,\sfd)$ and all isometric embeddings $\iota_i:X_i\to X$.

For the general case of \pmm~spaces of possibly infinite mass we adopt a cut-off argument based on the fact that, by assumption, in our spaces bounded sets have finite mass. Hence we  fix a Lipschitz compactly supported function  $\zeta:[0,\infty)\to[0,1]$ with $\zeta\equiv 1$ on $[0,1]$ and  then for given $\cX=[X,\sfd,\mm,\bar x]$ and  $k\in\Z$ we put 
\[
\d\mm_{[k]}(x):=\zeta(\sfd(x,\bar x)2^{-k})\d\mm(x),\qquad\text{ and }\qquad \cX_{[k]}:=[X,\sfd,\mm_{[k]},\bar x],
\]
and finally define the pointed-Gromov-weak distance as
\[
\pGW(\cX_1,\cX_2):=\sum_{k\in\Z}\frac1{2^{|k|}}\min\Big\{1,\pGW^{fm}(\cX_{1,[k]},\cX_{2,[k]})\Big\},
\]
see Definition \ref{def:pgw}. Notice that the choice to let $k$ vary in $\Z$ rather than in $\N$ is made to get completeness of the distance $\pGW$, see Theorem \ref{thm:completeness}.

\bigskip

\noindent\underline{Relation between pmG-convergence and pmGH-convergence} We shall see in Theorem \ref{thm:main_convergence} that the intrinsic convergence is equivalent to the extrinsic one and they also coincide with convergence w.r.t.\ $\pGW$ and $\D^\psi$ for  $\psi$ decreasing sufficiently fast (see also Remarks \ref{rem:diffpesi} and \ref{re:occhiopesi}).  As said, we shall call this convergence \emph{pointed measured Gromov} convergence, pmG-convergence in short.

Later, in Section \ref{se:relpmgh}, we shall analyze the relation between pmG-convergence and pointed measured Gromov Hausdorff one (pmGH for short). For compact spaces, the latter means that we can embed the given sequence of spaces in a common one to ensure pointed Gromov Hausdorff convergence at the level of metrics and weak convergence at the level of measures. The difference with pmG-convergence is thus in the fact that once the isometric embeddings are made, for pmGH-convergence both Hausdorff convergence of the metrics and weak convergence of measures are asked, while for pmG-convergence only weak convergence of measures is imposed. This is why we dropped the word `Hausdorff' in denoting the notion of convergence we are proposing.

Perhaps the best way to see the differences between pmG and pmGH convergences is with an example. Thus consider the sequence $n\mapsto ([0,1],\sfd_{\rm Eucl},\mm_n,\tfrac12)$, where $\sfd_{\rm Eucl}$ is the standard metric on $[0,1]$ and $\mm_n:=\frac1n\mathcal L^1\restr{[0,1]}+(1-\tfrac1n)\delta_{\frac12}$. As $n\to\infty$ this sequence pmG-converges to the (isomorphism class of a) one-point space equipped with a Dirac delta. Yet this is not the case for pmGH-convergence as certainly $n\mapsto([0,1],\sfd_{\rm Eucl})$ does not converge to a one-point space in the GH sense. What one could say in this situation is that the given sequence converges to $ ([0,1],\sfd_{\rm Eucl},\delta_{\frac12},\tfrac12)$ (which is the approach we are going to choose in Section \ref{se:relpmgh}) but the drawback of this choice is that the notion of convergence would not be defined for isomorphism classes of spaces, but would rather depend on the portion of space which is outside the support of the measure.

We therefore see that while pmGH-convergence always implies pmG-convergence (see Proposition \ref{prop:pmghD}), the converse is not always true. Yet, at least in the class of doubling spaces a behavior like that of the example proposed cannot occur and we shall see that in this case the two notions of convergence actually coincide (Proposition \ref{prop:Dpmgh}).

\begin{acknowledgement}
  {The authors warmly thank Luigi Ambrosio and Karl Theodor  Sturm for fruitful
    conversations on the topic, they also wish to thank  Takashi Shioya for pointing out the spectral stability property 
described in Remark \ref{rk:Shyoia}.

 They have been partially supported by
    ERC AdG GeMeThNES. G.S.~has also been partially supported by
    PRIN10-11 grant from MIUR for the project \emph{Calculus of Variations}. A.M.~has been supported by the ETH fellowship during the final stage of the work.}
\end{acknowledgement}

\section{Notation and preliminary results}

\subsection{Measures and weak/narrow convergence.}
We assume that the reader is familiar with the 
basic notions of measure theory and optimal transport 
on abstract metric spaces (standard references are \cite{Ambrosio-Gigli11}, \cite{Ambrosio-Gigli-Savare08} and \cite{Villani09}). Here we recall the basic definitions mainly in order to fix the notation.

Given a complete and separable metric space $(X,\sfd)$, we denote by
$\BorelSets X$ the collection of its Borel sets, by
$\Mloc X$ the set of Borel measures $\mu:\BorelSets X\to[0,\infty]$ 
which are finite on every bounded set, by 
$\M X$ the subset of finite measures $\mu$,
and 
by $\prob X$ the  collection of all Borel probability
measures.

On $\Mloc X$ we consider the topology induced by the 
duality with $\Cc X$, the space of bounded 
continuous functions with \emph{bounded
support}:
a sequence $(\mu_n)\subset \Mloc X$ converges weakly
to $\mui\in \Mloc X$ provided
\begin{equation}
\lim_{n\to\infty}\int f\,\d\mu_n=\int f\,\d\mui\qquad\forevery f\in \Cc
X.\label{eq:3}
\end{equation}
When $\mu_n,\mui\in \prob X$,  \eqref{eq:3} is equivalent to narrow
convergence, i.e.~convergence in $\M X$ with respect to the duality with the 
space $\Cb X$ of bounded continuous functions:
%
\[
\lim_{n\to\infty}\int f\,\d\mu_n=\int f\,\d\mui\qquad\forevery f\in \Cb
X.
\]
The following lemma provides a useful lower semicontinuity result. 
We set $\bar \N=\N\cup\{\infty\}$. 

\begin{lemma}\label{le:lemmino}
Let $(X,\sfd)$ be a complete and separable metric space and
$F_n:X\to[0,+\infty]$
(resp.~$G_n:X\to \R$), 
$n\in\bar\N$, 
lower semicontinuous 
(resp. continuous and uniformly bounded)
  functionals satisfying
\begin{equation}
\label{eq:lemmino}
\sfd(x_n,x_\infty)\to 0\qquad\Rightarrow\qquad F_\infty(x_\infty)\leq
\limi_{n\to\infty} F_n(x_n),\quad
G_\infty(x_\infty)=\lim_{n\to\infty}G_n(x_n).
\end{equation}
Then for every sequence $(\nu_n)\subset \M X$ \narrowly\  converging to $\nu_\infty\in\M X$ it holds
\begin{equation}
\int F_\infty\,\d\nu_\infty\leq\limi_{n\to\infty}\int
F_n\,\d\nu_n,\qquad
{\int G_\infty\,\d\nu_\infty=\lim_{n\to\infty}\int G_n\,\d\nu_n.
}\label{eq:87}
\end{equation}
\end{lemma}
\begin{proof}
  Providing $\bar\N$ by the usual polish topology, it is easy to check that 
  \eqref{eq:lemmino} yields the lower semicontinuity of the functional 
  $\bar F:X\times \bar \N\to[0,+\infty]$ defined by $\bar F(x,n):=F_n(x)$, $(x,n)\in X\times\bar\N$.
We then consider the product measures $\bar \nu_n:\nu_n\times\delta_n$, $n\in \bar \N$:
since $\bar\nu_n\to\bar\nu_\infty$ \narrowly\  in $\prob{X\times\bar N}$ as $n\to\infty$ 
and $\bar F_n$ are nonnegative,
we obtain
\[
\int F_\infty\,\d\nu_\infty=\int \bar F\,\d\bar\nu_\infty\leq 
\limi_{n\to\infty}\int \bar F\,\d\bar\nu_n=\limi_{n\to\infty}\int F_n\,\d\nu_n.
\]
The second property of \eqref{eq:87}
then follows by applying the 
previous lower semicontinuity result to $F_n(x):=G_n(x)-C$ and 
to $F_n(x):=C-G_n(x)$, where $C\ge 0$ is a sufficiently big constant
such that $-C\le G_n(x)\le C$ for every $x\in X$ and $n\in \N$.
\end{proof}

A simple application of Lemma \ref{le:lemmino} yields the lower semicontinuity of second moments
w.r.t.~\narrow\ convergence:
\begin{equation}
\label{eq:sem2mon}
\mu_n\to\mu_\infty\textrm{ \narrowly\  in $\prob X$,}\
\bar x_n\to \bar x_\infty\
\Rightarrow\ \int \sfd^2(\cdot, \bar x_\infty)\,\d\mu_\infty\leq\limi_{n\to\infty}\int \sfd^2(\cdot,\bar x_n)\,\d\mu_n.
\end{equation}
Convergence in \eqref{eq:sem2mon} holds if $(\mu_n)$ is \emph{$2$-uniformly
integrable}; recall that a subset $K\subset \probt X$ satisfies this property if
\begin{equation}
\lim_{R\to+\infty}\sup_{\mu\in K}\int_{B_R^c(\bar x)}\sfd^2(\cdot,\bar
x)\,\d\mu=0
\quad
\text{for some (and thus every) $\bar x\in X$.} 
\label{eq:2}
\end{equation}
In this case 
  \begin{equation}
    \label{eq:118}
    \lim_{n\to\infty}\int f\,\d\mu_n=\int f\,\d\mu
    \quad
    \forevery\
    f\in \rmC(X)\ \text{with quadratic growth,}
  \end{equation}
  i.e.~satisfying 
\[
    |f(x)|\le A(1+\sfd^2(x,\bar x))
    \text{ for some }A\ge 0,\ \bar x\in X\text{ and every }x\in X.
\]
Relative \narrow\ compactness in $\prob X$ can be characterized by Prokhorov's Theorem.
Let us first recall that set  $\mathcal K\subset\prob X$ is said to be tight provided for every $\eps>0$ 
there exists a compact set $K_\eps\subset X$ such that
\[
\mu(X\setminus K_\eps)\leq\eps\quad\forevery \mu\in \mathcal K.
\]
The we have the following classical result:
\begin{theorem}[Prokhorov]\label{thm:prok}
Let $(X,\sfd)$ be complete and separable and let $\mathcal
K\subset\prob X$.
Then the following are equivalent.
\begin{itemize}
\item $\mathcal K$ 
is precompact in the \narrow\ topology.
\item 
$\mathcal K$ 
is tight.
\item There exists a function $\psi:X\to[0,+\infty]$ with compact sublevels such that
\[
\sup_{\mu\in \mathcal K}\int\psi\,\d\mu<\infty.
\]
\end{itemize}
\end{theorem}

\subsection{Transport and $L^2$-Wasserstein distances.}
If $(Y ,\sfd )$ is a separable metric space and 
  $\tt :X\to Y $ is a Borel map, every probability measure $\mu\in
  \prob X$ can be pushed to a measure $\tt_\sharp \mu\in \prob Y $ by
  the formula 
  $(\tt_\sharp\mu)(B):=\mu(\tt^{-1}(B))$ for every $B\in \BorelSets
  \YY $. If $\tt^{-1}(B)$ is also bounded whenever $B$ is bounded (e.g.~when
  $\tt$ is an isometry), then the same formula applies to $\mu\in
  \Mloc X$ and defines a measure $\tt_\sharp \mu\in \Mloc \YY .$ 

 Given  $\mu_i\in \prob {X_i}$, the set $\Adm(\mu_1,\mu_2)
 \subset\prob{X_1\times X_2}$ is
 the set of admissible transport plans, i.e. the set of those
 $\ggamma$ such that $\pi^i_\sharp\ggamma=\mu_i$, $i=1,2$,
 where $\pi^1,\pi^2$ are the projections on
 the first and second coordinate respectively.

 By $\probt X$ we denote the space of probability measures with finite second moment, i.e. the set of those measures $\mu\in\prob X$ such that
 \[
 \int\sfd^2(x,\bar x)\,\d\mu(x)<\infty,\qquad\text{ for some (and thus any) }\bar x\in X.
 \]
We endow such space with the quadratic
transportation distance $W_2$ defined as
\[
W^2_2(\mu,\nu):=\inf_{\sggamma\in\Adm(\mu,\nu)}\int \sfd^2(x,y)\,\d\ggamma(x,y).
\]
It is well known that $(\probt X,W_2)$ is a complete and separable metric space and that $W_2$-convergence can be characterized as in the following proposition (for a proof, see e.g. Proposition 7.1.5 in \cite{Ambrosio-Gigli-Savare08}).
\begin{proposition}\label{prop:narw2}
Let $(X,\sfd)$ be a complete separable metric space and
$(\mu_n)\subset \probt X$ a sequence \narrowly\  converging to some
$\mu_\infty\in\prob X$. 
Then 
$\mu_\infty\in\probt X$ and 
$W_2(\mu_n,\mu_\infty)\to0$ 
if  either one of the following holds.
\begin{itemize}
\item $(\mu_n)$ is 2-uniformly integrable, i.e. \eqref{eq:2} holds.
\item 
  $(\mu_n)$ satisfies \eqref{eq:118}. \EEE
\item for some sequence $(\bar x_n)\subset X$ converging to $\bar x_\infty\in X$ it holds
\[
\lim_{n\to\infty}\int \sfd^2(\cdot, \bar x_n)\,\d\mu_n=\int \sfd^2(\cdot,\bar x_\infty)\,\d\mu_\infty<\infty.
\]
\item for some sequence $(\bar x_n)\subset X$ converging to $\bar x_\infty\in X$ it holds
\[
\lims_{n\to\infty}\int \sfd^2(\cdot, \bar x_n)\,\d\mu_n\leq \int \sfd^2(\cdot,\bar x_\infty)\,\d\mu_\infty<\infty.
\]
\end{itemize}
Conversely, if $(\mu_n)\subset\probt X$ is a sequence $W_2$-converging
to some $\mu_\infty\in\probt X$, then  the sequence also \narrowly\
converges and all the four properties above are true.
\end{proposition}
In order to freely deal with all measures in $\prob X$, rather then in the restricted set $\probt X$, it will be convenient to introduce the distance $\Wc$ as follows. We fix once and for all  a function $c\in \rmC([0,\infty))$ such that 
\begin{equation}
\label{subeq:cost}
    c\text{ is concave and non-constant},\quad
    c(0)=0,\quad \lim_{d\to\infty}c(d)<\infty,
\end{equation}
typical examples being $c(d):=\tanh (d)$ and $c(d):=d\land 1$.  Then if $(X,\sfd)$ is a
  metric space, the cost function
\[
    \sfc(x,y):=c(\sfd(x,y))\quad x,y\in X,
\]
is a bounded and complete distance on $X$, inducing the same
topology of $\sfd$.   The Kantorovich-Rubinstein-Wasserstein distance $\Wc$ is then introduced as
\begin{equation}
  \label{eq:86}
  \Wc(\mu,\nu):=\inf_{\sggamma\in\Adm(\mu_1,\mu_2)}\int
  \sfc(x,y)\,\d\ggamma(x,y),\quad
  \sfc\text{ defined as in (\ref{subeq:cost}a,b).}
\end{equation}
Recall that regardless of the choice of  $\sfc$ as in  \eqref{subeq:cost}
\begin{equation}
\label{eq:topwc}
\begin{split}
&\text{the space  $(\prob X,\Wc)$ is a complete and separable metric space}\\
&\text{whose   
topology induces the \narrow\ convergence of probability measures },
\end{split}
\end{equation}
see for instance Chapter 6 of \cite{Villani09} for a proof.

\section{Pointed metric measure spaces and their convergence}
\label{sec:pointedD}

\subsection{Pointed metric measure spaces 
  and their equivalence classes}  
  In this paper we will deal with \emph{pointed metric measure (\pmm)
     spaces}. A metric measure space  is a structure $(X,\sfd,\mm)$ 
  where 
  \begin{subequations} \label{pointed} \begin{equation} \label{eq:proprbase1}
      (X,\sfd)\textrm{ is a complete separable metric space},\quad \mm
      \in \Mloc X, \ \mm\neq 0
      \end{equation}
and a pointed metric measure space $(X,\sfd,\mm,\bar x)$ is a metric measure space together with 
      \begin{equation}
      \label{eq:pointed}  \text{a given  point $\bar x\in \supp(\mm)\subset X$.}  \end{equation} \end{subequations}
  Two metric measure spaces   $(X_i,\sfd_i,\mm_i)$, $i=1,2$ are called \emph{isomorphic}
  if there exists
  \begin{subequations} \label{eq:37}
  \begin{equation} \label{eq:37a} \text{ an isometric imbedding
      $\iota:\supp(\mm_1)\to X_2$\quad such that\quad  $\iota_\sharp\mm_1=\mm_2$,} 
\end{equation}
in case of pointed metric measure spaces $(X_i,\sfd_i,\mm_i,\bar x_i)$, $i=1,2$, we further require that
  \begin{equation} \label{eq:37b} \text{   $\iota(\bar x_1)=\bar x_2$.} 
\end{equation}
\end{subequations}
Any such $\iota$ is called isomorphism from $X_1$ to $X_2$.   

We shall typically denote by $\cX^*:=[X,\sfd,\mm]$ the equivalence class of the given metric measure space $(X,\sfd,\mm)$ and by $\cX:=[X,\sfd,\mm,\bar x]$ the one of the pointed metric measure space $(X,\sfd,\mm,\bar x)$. Also, we shall denote by $\X$ the  collection of all equivalence classes  of \pmm~spaces and by $\X_{fm}$  the subclass of (equivalence classes of) \pmm~spaces with finite mass, i.e. such that $\mm(X)<\infty$.
 
As usual, the foundational issue concerning the possibility of considering the equivalence classes of `all' p.m.m.s.\ can be easily avoided via the use of Kuratowski  embedding which allows to isometrically embed any separable metric space in $\ell^\infty$. We won't insist on this point any further.

  \begin{remark}[Only the support of the measure
    matters]\label{re:support}{\rm We stress that the portion of spaces outside the
      support of the measures is irrelevant, so that $\pmmX$ (resp. $(X,\sfd,\mm)$) is always
      isomorphic to $(\supp(\mm),\sfd,\mm,\bar x)$ (resp. $(\supp(\mm),\sfd,\mm)$).   } \fr\end{remark}
  %
	


In what comes next we shall need the following compactness result, which is of its own interest:
\begin{proposition}[Compactness of the class of isomorphisms]
  \label{prop:isocomp}
 Let $\cX_i=[X_i,\sfd_i,\mm_i,\bar x_i]$, $i=1,2$, be \pmm~spaces. Then the collection of isomorphisms from $X_1$ to $X_2$, i.e. of all maps $\iota:\supp(\mm_1)\to X_2$ as in \eqref{eq:37}, is sequentially compact w.r.t. pointwise convergence.
\end{proposition}
\begin{proof} 
Using the fact that $\mm_1$ is finite on bounded sets, it is easy to produce a continuous non-increasing function $\psi:[0,\infty)\to(0,\infty)$ such that the measure $\tilde\mm_1=\psi(\sfd_1(\cdot,\bar x_i))\mm_1$ has finite mass (see for instance the explicit construction given in Proposition \ref{prop:weights} below).
Fix such $\psi$, notice that the image measure $\tilde\mm_2:=\iota_\sharp\tilde\mm_1$ does not depend on the particular isomorphism $\iota$ chosen, let $(\iota_n)$ be a sequence of isomorphisms and define $\ggamma_n:=(\ii,\iota_n)_\sharp\tilde\mm_1\in\M{X_1\times X_2}$, $n\in\N$, $\ii$ being the identity map. Then we have $\pi^i_\sharp\ggamma_n=\tilde\mm_i$ for $i=1,2$ and $n\in\N$ and given that $\tilde\mm_1,\tilde\mm_2$ are tight (because they have finite mass), we easily deduce that the family $\{\ggamma_n\}_{n\in\N}$ is tight as well. 

Thus by Theorem \ref{thm:prok} we know that up to pass to a subsequence, not relabeled, we may assume that $(\ggamma_n)$ weakly converges to a limit plan $\ggamma\in\M{X_1\times X_2}$. Pick $(x_1',x_2'),(x_1'',x_2'')\in\supp(\ggamma)$ and notice that  the weak convergence of $(\ggamma_n)$ to $\ggamma$ ensures that there are $(x_{1,n}',x_{2,n}'),(x_{1,n}'',x_{2,n}'')\in\supp(\ggamma_n)$ such that $(x_{1,n}',x_{2,n}')\to (x_{1}',x_{2}')$ and  $(x_{1,n}'',x_{2,n}'')\to (x_{1}'',x_{2}'')$ in $X_1\times X_2$ as $n\to\infty$. Given that by construction it holds $\sfd_1(x_{1,n}',x_{1,n}'')=\sfd_2(x_{2,n}',x_{2,n}'')$ for every $n\in\N$, we deduce  that $\sfd_1(x_{1}',x_{1}'')=\sfd_2(x_{2}',x_{2}'')$, i.e. that $\ggamma$ is concentrated on the graph of an isometry $\iota$ defined $\mm_1$-a.e. with values in $X_2$. By continuity, we can extend $\iota$ to an isometry defined on the whole $\supp(\mm_1)$, so that it remains to show that $\iota_n(x)\to \iota(x)$ as $n\to\infty$ for every $n\in\N$ and $x\in\supp(\mm_1)$. But this is obvious, indeed for given $x\in\supp(\mm_1)$ we can find, using as before the weak convergence of the $\ggamma_n$'s, points $(x_n,\iota_n(x_n))\in\supp(\ggamma_n)$ converging to $(x,\iota(x))\in\supp(\ggamma)$ in $X_1\times X_2$ as $n\to\infty$, so that  we can conclude with
\[
\sfd_2(\iota(x),\iota_n(x))\leq\sfd_2(\iota(x),\iota_n(x_n))+\sfd_2(\iota_n(x_n),\iota_n(x))=\sfd_2(\iota(x),\iota_n(x_n))+\sfd_1(x_n,x)\to 0,
\]
as $n\to\infty$.
\end{proof}

The following simple lemma will be useful
to adapt some results 
already available for metric measure spaces (as Gromov reconstruction Theorem, see 
\cite[3$\frac 12$.7]{Gromov07} and
the next Theorem \ref{thm:Grecon})
to the case of \pmm~spaces.
\begin{lemma}
  \label{le:pvsnonp}
  Let $(X_i,\sfd_i,\mm_i,\bar x_i)$, $i=1,2$, be two
  \pmm~spaces with equal finite mass
  $m=\mm_1(X_1)=\mm_2(X_2)<\infty$.
  
  Then they are isomorphic 
  as \pmm~spaces if and only if 
  there exists $\lambda>m$ such that
  $(X_1,\sfd_1,\mm_1+\lambda \delta_{\bar x_1})$ 
  is isomorphic to 
  $(X_2,\sfd_2,\mm_2+\lambda \delta_{\bar x_2})$
  as m.m.~spaces.
\end{lemma}
\begin{proof}
  One implication is obvious; 
  in order to prove the converse one, let $\iota:
  \supp(\mm_1)\to X_2$ 
  (by \eqref{eq:pointed} 
  $\bar x_1\in \supp(\mm_1)$) be an isometry
  such that 
  $\iota_\sharp(\mm_1+\lambda \delta_{\bar x_1})=
  \mm_2+\lambda \delta_{\bar x_2}$ 
  and let $y=\iota(\bar x_1)\in \supp(\mm_2)$,
  $\nn=\iota_\sharp\mm_1$.
  Since $\iota_\sharp (\lambda \delta_{\bar x_1})=
  \lambda \delta_y$ we have
  \begin{displaymath}
    \nn+\lambda \delta_y=
    \mm_2+\lambda\delta_{\bar x_2}.
  \end{displaymath}
  Since $\mm_2(X_2)<\lambda$ 
  the only possibility is that $y=\bar x_2$ and therefore
  $\mm_2=\nn$.  
\end{proof}

\subsection{Equivalent notions of convergence of pointed metric measure spaces}\label{se:nozioni}
\subsubsection{A) Intrinsic approach: use of Gromov reconstruction theorem}

It is a remarkable result due to Gromov \cite{Gromov07}
that equivalence classes of normalized metric-measure spaces
can be characterized in terms of integral functionals
depending on the distances between points. In order to first state such result and then adapt it to the case of general \pmm~spaces we need to introduce some notation.

First of all we 
introduce the finite dimensional spaces of (pseudo)metrics
in $\R^{N\times N}$
\[
  \begin{aligned}
    \G^N:=\Big\{&\gg\in \R^{N\times N}:\gg=(g_{ij})_{1\le
      i,j< N}: \\&
  \  0=g_{ii}\le g_{ji}=g_{ij}\le g_{ih}+g_{hj}\quad
    \forevery i,j,h\in \{1,\cdots,N\}\Big\},
  \end{aligned}
\]
endowed with the canonical product distance and topology
inherited as convex subset of $\R^{N\times N}$;
we also consider the continuous map
\[
  \jey_X^N: X^{N}\to \G^N,\quad
  \jey_X^N:\xx=(x_i)_{1\le i\le N}\mapsto \gg=(g_{ij})_{1\le i,j
    \le N}:
  \quad g_{ij}:=\sfd(x_i,x_j),
\]
and the image measures
\begin{equation}
  \label{eq:92}
  \mm_\cX^N:=(\jey_X^N)_\sharp(\delta_{\bar x}\otimes \mm^{\otimes (N-1)})
  \in \Mloc{\G^N}.
\end{equation}
It is immediate to check that $\mm_\cX^N$ only depend
on the equivalence class $\cX=\pmmXclass$. 
Those measures naturally arise when one considers 
``cylindrical'' integral functionals of the type
\begin{equation}
  \label{eq:93}
  \dualspace\varphi\cX:=\int_{X^{N}}
  \varphi\big((\sfd(x_i,x_j))_{ij=1}^N\big)
  \,\d\delta_{\bar x}(x_1)\,\d\mm^{\otimes (N-1)}(x_2,\cdots,x_N)=
  \int_{\G^{N}}\varphi(\gg)\,\d \mm^{N}_\cX(\gg),
\end{equation}
for every $\varphi\in \Cc{\G^{N}}$. In fact, 
for $\cX_1,\cX_2\in \X$ by definition one has
\[
  \text{$\dualspace\varphi{\cX_1}=\dualspace\varphi{\cX_2}$
  for every $\varphi\in \Cc{\G^N}$, $N\in\N$
  \quad$\Leftrightarrow$\quad
  $\mm_{\cX_1}^N=\mm_{\cX_2}^N$ for every $N\in \N.$}
\]
\begin{remark}
  \label{rem:tedious-image}
  \upshape
  Notice that we could equivalently consider 
  the map 
  \begin{displaymath}
    \tilde j_\cX^N:X^{N-1}\to\G^N,\quad
    \xx=(x_i)_{2\le i\le N}\mapsto (g_{ij})_{1\le i,j\le N}:\quad
    g_{i,j}:=
    \begin{cases}
      \sfd(x_i,x_j)&\text{if }i,j\ge 2,\\
      \sfd(\bar x,x_j)&\text{if }i=1,\\
      \sfd(x_i,\bar x)&\text{if }j=1,
    \end{cases}
  \end{displaymath}
  and 
  $\mm_\cX^N=(\tilde j_\cX^N)_\sharp \mm^{\otimes (N-1)}$.
\fr\end{remark}
\begin{remark}  \label{rem:tedious-image2}{\rm
  When $\mm(X)<\infty$,
  \eqref{eq:93} makes sense for
  every $\varphi\in \rmC_b(\G^N)$ 
  and it is not difficult to check by `expanding' the product $(\lambda \delta_{\bar x}+\mm)^{\otimes N}$
  that the knowledge of 
  all the measures $(\mm_\cX^N)_{N\in\N}$ 
  also determines all the measures
  of the type
\[
    (\jey_\cX^N)_\sharp
    \Big(\big(\lambda \delta_{\bar x}+\mm\big)^{\otimes N}
    \Big)\in \M{\G^N}\quad
    \ N\in\N,
\]
for every $\lambda>0$.}\fr
\end{remark}
A similar construction is possible for non-pointed metric measure spaces: denoting by $\cX^*$ the equivalence class of $(X,\sfd,\mm)$ and given $n\in\N$ and $\varphi\in \Cc{\G^{N}}$, we define
\[
\varphi^*[\cX^*]:=\int_{X^N}\varphi\big((\sfd(x_i,x_j))_{ij=1}^N\big)
  \,\d\mm^{\otimes N}(x_1,\cdots,x_N)=
  \int_{\G^{N}}\varphi(\gg)\,\d (\jey_X^N)_\sharp\mm^{\otimes N}(\gg).
\]
Gromov reconstruction theorem for metric measure spaces can then be formulated as follows, see  \cite[3$\frac 12$.7]{Gromov07} for the proof:
\begin{theorem}[Gromov reconstruction for m.m.~spaces of finite mass]\label{thm:Greconbase}
Let $\cX_i^*:=[X_i,\sfd_i,\mm_i]$, $i=1,2$, be two metric measure spaces with finite mass. Assume that
\[
\varphi^*[\cX^*_1]=\varphi^*[\cX^*_2]\qquad\text{for every $N\in\N$ and $\varphi\in \Cc{\G^{N}}$.}
\]
Then $\cX_1^*=\cX_2^*$.
\end{theorem}
The compactness given by Proposition \ref{prop:isocomp} allows  to extend  the reconstruction theorem for general \pmm~spaces:
\begin{theorem}[Gromov reconstruction for p.m.m.~spaces]
  \label{thm:Grecon}
  Let $\cX_i=[X_i,\sfd_i,\mm_i,\bar x_i]$, $i=1,2$, be \pmm~spaces 
  such that 
  \[
  \dualspace\varphi{\cX_1}=\dualspace\varphi{\cX_2}\qquad\text{for every $N\in\N$ and $\varphi\in \Cc{\G^{N}}$.}
  \]
Then
  $\cX_1=\cX_2$.
\end{theorem}
\begin{proof}
  Let $\zeta:[0,\infty)\to[0,1]$ be a Lipschitz cut-off function identically 1 on $[0,1]$, positive on $[1,2)$ and identically 0 on $[2,\infty)$. Fix $R>0$ and let $\bar \zeta_{R,i}:X_i\to[0,\infty)$, $i=1,2$, be given by $\bar\zeta_{R,i}(x):=\zeta(\sfd_i(\cdot,\bar x_i)/R)$, and $Z_R:\G^N\to \R$
  defined by
  \begin{equation}
    \label{eq:98}
    Z_R(\gg):=\zeta\big(\frac{g_{12}}R\big)\,\zeta\big(\frac {g_{13}}R\big)\cdots
    \zeta\big(\frac{g_{1N}}R\big),
    \quad
    \gg=(g_{ij})_{1\le i,j\le
    N}\in \G^N.
  \end{equation}
Now observe that
   \begin{equation}
    (\bar \zeta_{R,i} \mm_i)_{\cX_i}^N:=
    (\jey_X^N)_\sharp\Big(\delta_{\bar x_i}\otimes (\bar \zeta_{R,i}
    \mm_i)^{\otimes N}\Big)= Z_R\cdot (\mm_i)_{\cX_i}^N,\qquad\forall N\in\N,
    \label{eq:97}
  \end{equation}
indeed, for every $\varphi\in \Cc{\G^N}$ we have
  \begin{align*}
    \int_{\G^N}\varphi(\gg)\,\d  (\bar \zeta_{R,i} \mm_i)_{\cX_i}^N
    &=
    \int_{X^{N+1}}\varphi\big((\sfd_i(x_j,x_k))_{j,k}\big)
     \prod_{\ell=1}^N\zeta\big(\tfrac{\sfd_i(\bar x_i
    ,x_\ell)}R\big)\,\d(\delta_{\bar x_i}\otimes
    \mm_i^{\otimes N}) 
    \notag \\&=
    \int_{X^{N+1}}\varphi\big((\sfd_i(x_j,x_k))_{j,k}\big)
    Z_R\big((\sfd_i(x_j,x_k))_{j,k} \big)\,\d(\delta_{\bar x_i}\otimes
    \mm_i^{\otimes N}) 
        \notag \\&=
    \int_{\G^N}\varphi(\gg)Z_R(\gg)\,\d  (\mm_i)_{\cX_i}^N.
  \end{align*}
From \eqref{eq:97}, the fact that $\bar \zeta_{R,i} \mm_i(X_i)<\infty$, $i=1,2$, and taking into account Remark \ref{rem:tedious-image2} we deduce that
\[
 (\jey_{\cX_1}^N)_\sharp    \Big(\big(\lambda \delta_{\bar x_1}+\bar \zeta_{R,1} \mm_1\big)^{\otimes N}    \Big)=
 (\jey_{\cX_2}^N)_\sharp    \Big(\big(\lambda \delta_{\bar x_2}+\bar \zeta_{R,2} \mm_2\big)^{\otimes N}    \Big),
    \]
for every $\lambda>0$, $N\in\N$ and $\varphi\in \Cc{\G^N}$. Thus we can apply  Gromov reconstruction Theorem \ref{thm:Greconbase}   for metric measure space of finite mass to
  obtain
  $[X_1,\sfd_1,\bar\zeta_{R,1}\mm_1+\lambda\delta_{\bar x_1}]=
  [X_2,\sfd_2,\bar\zeta_{R,2}\mm_2+\lambda\delta_{\bar x_2}]$ for every $\lambda>0$. 
  
  Using  Lemma \ref{le:pvsnonp}
  we deduce that there exists  
  an isometry 
  \[
  \iota_R:B_{2 R}(\bar x_1)\cap \supp(\mm_1)\to 
  B_{2 R}(\bar x_2)\cap \supp(\mm_2)\]
  such that
  \begin{equation}
\label{eq:ufficio1}
  \iota_R(\bar x_1)=\bar x_2,\qquad\text{ and }\qquad
  (\iota_R)_\# (\bar\zeta_{R,1}\mm_1)=\bar\zeta_{R,2}\mm_2.
\end{equation}
From the definition of $\bar\zeta_{R,i}$, $i=1,2$, and the fact that $\iota_R$ is an isometry we also have the consistency property
\begin{equation}
\label{eq:ufficio2}
(\iota_R)_\# (\bar\zeta_{R',1}\mm_1)=\bar\zeta_{R',2}\mm_2,\qquad\forall R'<R.
\end{equation}
Now let $R_n\uparrow\infty$ be a chosen sequence and $\iota_{R_n}$ the corresponding sequence of isomorphisms. For any given $k\in\N$, the validity of \eqref{eq:ufficio1} and \eqref{eq:ufficio2} ensures that the maps $\{\iota_{R_n}\}_{n\geq k}$ are isomorphisms of $(X_1,\sfd_1,\bar\zeta_{R_k,1}\mm_1,\bar x_1)$ and $(X_2,\sfd_2,\bar\zeta_{R_k,2}\mm_2,\bar x_2)$, hence applying Proposition \ref{prop:isocomp} we deduce that there exists a subsequence of $(R_n)$ whose corresponding isomorphisms pointwise converge to a limit isomorphism, call it $\tilde\iota_k$, of $(X_1,\sfd_1,\bar\zeta_{R_k,1}\mm_1,\bar x_1)$ and $(X_2,\sfd_2,\bar\zeta_{R_k,2}\mm_2,\bar x_2)$.

By a standard diagonalization argument we can ensure that  the limit map $\tilde\iota_k$ does not depend on $k$. Denote it as $\tilde\iota$ and notice that the construction and \eqref{eq:ufficio1} and \eqref{eq:ufficio2} grant that
\[
\tilde  \iota(\bar x_1)=\bar x_2,\qquad\text{ and }\qquad
  \tilde\iota_\# (\bar\zeta_{R,1}\mm_1)=\bar\zeta_{R,2}\mm_2,\qquad\forall R>0,
\]
which, by definition of $\bar\zeta_{R,i}$, $i=1,2$, ensures that $\tilde\iota$ is an isomorphism of $(X_1,\sfd_1,\mm_1,\bar x_1)$ and $(X_2,\sfd_2,\mm_2,\bar x_2)$.
\end{proof}

Theorem \ref{thm:Grecon} suggests the following notion of convergence:
\begin{definition}[Intrinsic notion of convergence of equivalence classes of p.m.m.s.]\label{def:intrconv}
Let \linebreak  $\cX_n=\pmmXcseq n \in \X$, $n\in\N$, and  $\cX_\infty=\pmmXcseq\infty\in \X$. 

We say that the sequence $n\mapsto\cX_n$ intrinsically converges to $\cX_\infty$ provided:
\[
\textrm{for any $N\in\N$ and $\varphi\in \Cc{\G^N}$ we have}\qquad
 \lim_{n\to\infty}\varphi[\cX_n]=\varphi[\cX_\infty].
 \]
 \end{definition}

\subsubsection{B) Extrinsic approach: embedding everything in a common space}
Recalling the definition of weak convergence of measures in $\Mloc X$ given in \eqref{eq:3}, where $(X,\sfd)$ is a complete and separable metric space, we can give the following definition of convergence:

\begin{definition}[Extrinsic notion of convergence of equivalence classes of p.m.m.s.]
  \label{def:Dconv}
  Let 
  $(X_n,\sfd_n,\mm_n,\bar{x}_n)$, $n\in\bar\N$, be
  a sequence of \pmm~spaces. 
   We say that the corresponding sequence of classes $n\mapsto \cX_n$ extrinsically converges to $\cX_\infty$ provided there exists 
   a complete and 
   separable metric space $(X ,\sfd)$
   and isometric embeddings $\iota_n:X_n\to
   X$, $n\in\bar\N$, such that
  \begin{equation}
    \label{eq:19bis}
    (\iota_n)_\sharp\mm_n\to(\iota_\infty)_\sharp\mm_\infty\text{  weakly in $\Mloc X $},\quad
   \iota_n(\bar x_n)\to \iota_\infty(\bar x_\infty)\in \supp((\iota_\infty)_\sharp\mm_\infty)
    \quad\text{in }X.
  \end{equation}
  The system $\big((X ,\sfd ,\iota_n)\big)_{n\in\bar\N}$ is called
  an \emph{effective realization} of the extrinsic convergence.
  \end{definition}
Notice that a priori it is not clear that this definition is well-posed, because the isometries $\iota_n$ are required to be defined on the whole spaces $X_n$ and not just on the supports of the measures $\mm_n$, so that a priori the convergence might depend on the chosen elements in the equivalence classes. In fact, this is not an issue thanks to the following simple result:
\begin{proposition}[Extension of isometric immersions]\label{le:extension}
Let $(X_n,\sfd_{X_n})$, $n\in\N$, and $(Y,\sfd_Y)$ be separable metric spaces, $F_n\subset X_n$ given subsets and $\iota_{F_n}:F_n\to Y$ isometric embeddings. Then there 
exists  a complete separable metric space $(Z,\sfd_Z)$ and isometric embeddings $\iota_Y:Y\to Z$, $\iota_{X_n}:X_n\to Z$, $n\in\N$, such that
\[
\iota_{X_n}\restr{F_n}=\iota_Y\circ\iota_{F_n},\qquad\forall n\in\N.
\]
\end{proposition}
\begin{proof} Put $(Z_0,\sfd_{Z_0}):=(Y,\sfd_Y)$. Then define the set $Z_{1}:=Z_0\sqcup (X_{0}\setminus F_{0})$ and the distance $\sfd_{Z_{1}}$ on it by requiring $\sfd_{Z_1}$  to coincide with $\sfd_{Z_0}$ on $(Z_0)^2$, with $\sfd_{X_0}$ on $(X_{0}\setminus F_{0})^2$, defined as
\[
\sfd_{Z_{1}}(z',z''):=\inf_{x\in F_0} \sfd_{Z_0}(z',\iota_{F_0}(x))+\sfd_{X_0}(x,z'')
\]
on $Z_0\times(X_0\setminus F_0)$ and extended by symmetry on $(X_{0}\setminus F_{0})\times Z_0$. It is immediate to check that $(Z_1,\sfd_{Z_1})$ is a separable metric space, that the natural inclusion $\iota_{Z_0}^{Z_1}$ of $Z_0$ in $Z_1$ is an isometry and that the map $\iota_{X_0}^{Z_1}:X_0\to Z_1$ defined as
\[
\iota_{X_0}^{Z_1}(x):=\iota_{Z_0}^{Z_1}(\iota_{F_0}(x)),\quad\text{ if $x\in F_0$},\qquad\qquad\iota_{X_0}^{Z_1}(x):=x\quad\text{ if $x\in X_0\setminus F_0$},
\]
is an isometry as well for which the identity $\iota_{X_0}^{Z_1}\restr{F_0}=\iota_{Z_0}^{Z_1}\circ \iota_{F_0}$ holds.

Iterating the construction we produce a sequence of separable metric spaces $(Z_{n},\sfd_{Z_{n}})$, $n\in\N$, and isometric embeddings $\iota_{Z_{n-1}}^{Z_{n}}$, $\iota_{X_{n-1}}^{Z_{n}}$ of $Z_{n-1},X_{n-1}$ respectively in $Z_n$ satisfying
\[
\iota_{X_{n-1}}^{Z_{n}}\restr{F_{n-1}}=\iota_{Z_{n-1}}^{Z_{n}}\circ \cdots\circ \iota_{Z_0}^{Z_1}\circ \iota_{F_0}.
\]
Eventually defining  $Z:=\cup_{n\in\N}Z_n$, equipping it with the natural distance induced by the inclusions of the $Z_n$'s and taking, if necessary, its completion, we conclude the proof.
\end{proof}

\subsubsection{C) A variant of Sturm's distance $\D$}
In \cite{Sturm06I} Sturm worked in the class of m.m. spaces $(X,\sfd,\mm)$ \emph{normalized} and with \emph{finite variance}, these two meaning that $\mm\in\probt X$. For this sort of spaces he defined the distance $\D$ as
\[
\D\Big((X_1,\sfd_1,\mm_1),(X_2,\sfd_2,\mm_2)\Big):=\inf W_2\big((\iota_1)_\sharp\mm_1,(\iota_2)_\sharp\mm_2\big),
\]
the infimum being taken among all metric spaces $(X ,\sfd )$ 
and all isometric immersions $\iota_i:X_i\to X$, $i=1,2$.  In fact, in the aforementioned reference the isometries $\iota_1,\iota_2$ were required to be defined only on $\supp\mm_i$, $i=1,2$, but as for the case of the extrinsic notion of convergence, Proposition \ref{le:extension} shows that it is not restrictive to assume them to be defined on the whole $X_1$, $X_2$.

The idea to define a distance on pointed non-normalized spaces is the following.  We fix a 
\begin{equation}
\label{eq:proprpsi}
\text{non-increasing and absolutely continuous function $\psi:[0,\infty)\to(0,\infty)$ with $\lim_{r\to\infty}\psi(r)=0$}
\end{equation}
 and introduce the class $\X^\psi\subset \X$ of those (equivalence classes of) p.m.m.\ spaces $(X,\sfd,\mm,\bar x)$ such that
\[
\int (1+\sfd^2(\cdot,\bar x))\psi(\sfd(\cdot,\bar x))\,\d\mm<\infty.
\]
For $[X,\sfd,\mm,\bar x]\in\X^\psi$ we define the normalization constant $\z_\psi\in(0,\infty)$ and the renormalized measure $\mm_\psi\in\probt X$ as
\begin{equation}
\label{eq:normalizzare}
\z_\psi:=\int\psi(\sfd(\cdot,\bar x))\,\d\mm, \qquad \qquad \qquad\d\mm_\psi:=\frac1{\z_\psi}\psi(\sfd(\cdot,\bar x))\,\d\mm.
\end{equation}
Then we introduce the distance:
\begin{definition}[The distance $\D^\psi$ on $\X^\psi$]\label{def:sturmvar}
For  $\cX_i:=[X_i,\sfd_i,\mm_i,\bar x_i]$, $i=1,2$, with $\cX_1,\cX_2\in\X^\psi$ we define
\[
\D^\psi(\cX_1,\cX_2):=    \left|\log\frac{\z_{1,\psi}}{\z_{2,\psi}}\right|+
    \inf
    \Big(W_2\big((\iota_1)_\sharp\mm_{1,\psi},(\iota_2)_\sharp\mm_{2,\psi}\big)+
    \sfd \big(\iota_1(\bar x_1),\iota_2(\bar x_2)\big)\Big),
\]
the infimum being taken among all separable metric spaces $(X ,\sfd )$
and isometric embeddings $\iota_1:X_1\to X $ and $\iota_2:X_2\to
X $. 
\end{definition}
It is clear that the definition is well posed, in the sense that $\D^\psi(\cX_1,\cX_2)$ only depends on $\cX_1,\cX_2$ and not on the chosen elements in the equivalence class. Also, using the fact that $\D$ is a distance on the class of normalized spaces with finite variance (see \cite{Sturm06I}), it is immediate to verify that $\D^\psi$ is a distance on $\X^\psi$. We omit the simple proof.

\medskip

It is obvious that spaces in $\X^\psi$ have a controlled volume growth. In the next statement we show that also the converse implication holds, a fact which will turn useful later on. Notice that in the bound \eqref{eq:psifromphi} below the third power of the distance appears: this will help in the proof of Theorem \ref{thm:main_convergence} to gain 2-uniform integrability (see also Remark \ref{re:occhiopesi}).
\begin{proposition}\label{prop:weights}
Let $\phi:[0,\infty)\to[0,\infty)$ be a given continuous non-decreasing function and $[X,\sfd,\mm,\bar x]\in X$ be such that
\[
\mm(B_R(\bar x))\leq \phi(R),\qquad\forall R\geq 0.
\]
Define $\psi:[0,\infty)\to(0,\infty)$ as
\begin{equation}
\label{eq:daphiapsi}
\psi(r):=\int_r^{\infty}\frac1{(1+s^3)^2\phi(s)}\,\d s.
\end{equation}
Then $\psi$ is as in \eqref{eq:proprpsi}, $[X,\sfd,\mm,\bar x]\in X^\psi$ and
\begin{equation}
\label{eq:psifromphi}
\int (1+\sfd^3(\cdot,\bar x))\psi(\sfd(\cdot,\bar x))\,\d\mm\leq \int_0^\infty\frac1{1+r^3}\,\d r<\infty.
\end{equation}
\end{proposition}
\begin{proof} It is obvious that formula \eqref{eq:daphiapsi} defines a  non-increasing and absolutely continuous map  such that $\lim_{r\to\infty}\psi(r)=0$. Conclude using the identity 
\[
\int \psi(\sfd(x,\bar x)) f(x)\,\d\mm(x)=-\int_0^\infty\psi'(r)\left(\int_{B_r(\bar x)} f(x)\,\d\mm(x)\right)\,\d r,
\]
valid for any such $\psi$ and  every continuous  $f:X\to[0,\infty)$.
\end{proof}

\subsubsection{D) A metric independent of chosen weights: the $\pGW $ distance}
One drawback of the distances $\D^\psi$ is that they depend on the chosen weight function $\psi$. In order to build a `universal' distance we can use the distance $W_c$ (recall the definition \eqref{eq:86}) in place of $W_2$ and proceed as follows.

For two normalized metric measure  spaces $\cX_i^*:=[X_i,\sfd_i,\mm_i]$, $i=1,2$, we define
\[
\GW(\cX^*_1,\cX^*_2):=\inf W_c((\iota_1)_\sharp\mm_1,(\iota_2)_\sharp\mm_2),
\]
where here and below the $\inf$ being taken, as usual, among all complete and separable spaces $(X,\sfd)$ and isometric immersions $\iota_i:X_i\to X$, $i=1,2$.

In case $\cX^*_1,\cX^*_2$ have finite mass, the natural variant of $\GW$ is given by
\[
\GW^{fm}(\cX^*_1,\cX^*_2):=\Big|\log\frac{\mm_1(X_1)}{\mm_2(X_2)}\Big|+\inf W_c((\iota_1)_\sharp\tilde\mm_1,(\iota_2)_\sharp\tilde\mm_2),
\]
where the measures $\tilde\mm_i:=\frac{\mm_i}{\mm_i(X_i)}$ are the normalizations of $\mm_i$, $i=1,2$. 

For pointed metric measure spaces of finite mass $\cX_i:=[X_i,\sfd_i,\mm_i,\bar x_i]$, $i=1,2$,  we can then consider the distance
\[
\pGW^{fm}(\cX_1,\cX_2):=\Big|\log\frac{\mm_1(X_1)}{\mm_2(X_2)}\Big|+\inf \sfd(\iota_1(\bar x_1),\iota_2(\bar x_2))+W_c((\iota_1)_\sharp\tilde\mm_1,(\iota_2)_\sharp\tilde\mm_2).
\]
The fact that all these are  separable distances  on the corresponding classes can be seen arguing exactly as in \cite{Sturm06I,Greven09}, we omit the details.

For the general case of \pmm~spaces of possibly infinite mass we adopt a cut-off argument:  fix once and for all a function 
\begin{equation}
\label{eq:zeta}
\text{$\zeta:[0,\infty)\to[0,1]$ Lipschitz, \quad with $\zeta\equiv 1$ on $[0,1]$ and $\zeta\equiv 0$ on $[2,\infty)$.}
\end{equation}
Then for given $(X,\sfd,\mm,\bar x)\in\X$ and  $k\in\Z$ we consider the rescaled measures $\mm_{[k]}$ on $X$ defined as
\begin{equation}
\label{eq:mk}
\d\mm_{[k]}(x):=\zeta(\sfd(x,\bar x)2^{-k})\d\mm(x),
\end{equation}
and the corresponding isomorphism class 
\begin{equation}
\label{eq:mk2}
\cX_{[k]}:=[X,\sfd,\mm_{[k]},\bar x]. 
\end{equation}
Then we define:
\begin{definition}[The distance $\pGW$]\label{def:pgw}
\[
\pGW(\cX',\cX''):=\sum_{k\in\Z}\frac1{2^{|k|}}\min\Big\{1,\pGW^{fm}(\cX'_{[k]},\cX''_{[k]})\Big\}.
\]
\end{definition}
Knowing that $\pGW^{fm}$ is a distance, it is obvious that $\pGW$ is a distance on $\X$ and that $\pGW(\cX_1,\cX_2)\to 0$ if and only if $\pGW^{fm}(\cX_{1,[k]},\cX_{2,[k]})\to 0$ for every $k\in\Z$.

Notice that we are letting $k$ vary on $\Z$ rather than on $\N$: we did so to get completeness of $\pGW$, see Theorem \ref{thm:completeness}.

\subsection{Proof of the equivalence of the various notions of convergence}

Aim of this section is to prove that the various notions of convergence previously introduced coincide. As already said in the introduction, such equivalence is already known under particular circumstances, like compactness or normalization of the spaces and our scope is to generalize such previously known facts. In doing so, we shall use the fact that, shortly said, `intrinsic convergence implies  convergence w.r.t. the distance $\GW$ for normalized spaces' as proved in Proposition 10.5, Proposition 10.1, and Theorem 5 of \cite{Greven09}. The rigorous statement is the following:
\begin{theorem}\label{thm:greven}
Let $\cX^*_n:=(X_n,\sfd_n,\mm_n)$, $n\in\bar\N$, be normalized metric measure spaces. Assume that 
\[
\lim_{n\to\infty}\varphi^*[\cX_n^*]=\varphi^*[\cX_\infty^*],
\]
for every $N\in\N$ and $\varphi\in \Cc{\G^{N}}$. Then $\lim_{n\to\infty}\GW(\cX^*_n,\cX^*_\infty)= 0$.
\end{theorem}
We can now turn to the main result of this section:
\begin{theorem}
  \label{thm:main_convergence}
  Let $\cX_n=[X_n,\sfd_n,\mm_n,\bar x_n]$, $n\in  \bar\N$.  Then the following are equivalent:
  \begin{enumerate}[\rm (A)]
\item The sequence $n\mapsto\cX_n$ intrinsically converges to $\cX_\infty$ in the sense of Definition \ref{def:intrconv}.
\item The sequence $n\mapsto\cX_n$ extrinsically converges to $\cX_\infty$ in the sense of Definition \ref{def:Dconv}.
\item There exists $\psi:[0,\infty)\to[0,\infty)$ nonincreasing and absolutely continuous such that $\cX_n\in \X^\psi$ for every $n\in\bar\N$ and  $\lim_{n\to\infty}\Dmeas\psi(\cX_n,\cX_\infty)= 0$.
\item $\lim_{n\to\infty}\pGW(\cX_n,\cX_\infty)=0$.
\end{enumerate}
\end{theorem}
\begin{proof}  $\ $\\
\noindent  \textbf{(C) $\Rightarrow$ (B)}  According to \eqref{eq:normalizzare} for every $n\in\bar\N$ we define 
\[
\z_{n,\psi}:=\int\psi(\sfd_n(\cdot,\bar x_n))\,\d\mm_n,\qquad\qquad\mm_{n,\psi}:=\frac1{\z_{n,\psi}}\psi(\sfd_n(\cdot,\bar x_n))\mm_n.
\]
By definition of $\Dmeas\psi$ we know that for every $n\in\N$ there exists a complete separable metric space  $(Y_n,\sfd_{Y_n})$ and isometric embeddings $\iota_n,\iota^\infty_{n}$ of $(X_n,\sfd_n)$, $(X_\infty,\sfd_\infty)$ respectively in $Y_n$ such that
\begin{equation}
\label{eq:divano}
\left|\log\frac{\z_{n,\psi}}{\z_{\infty,\psi}}\right|+\sfd_{Y_n}(\iota_n(\bar x_n),\iota^\infty_n(\bar x_\infty))+ W_2((\iota_n)_\sharp\mm_{n,\psi},(\iota^\infty_n)_\sharp\mm_{\infty,\psi})\leq \Dmeas\psi(\cX_n,\cX_\infty)+\frac1n.
\end{equation}
Now define $Y:=\sqcup_{n\in\bar\N}X_n$ and a pseudodistance $\sfd_Y$ on $Y$ by declaring that
\[
\begin{split}
&\sfd_Y(y,y'):=\\
&\left\{
\begin{array}{ll}
\sfd_n(y,y'),&\quad\textrm{ if $y,y'\in X_n$ for some $n\in\bar\N$},\\
\sfd_{Y_n}(\iota_n(y),\iota^\infty_n(y')),&\quad\textrm{ if $y\in X_n$ for some $n\in\N$, $y'\in X_\infty$},\\
\sfd_{Y_n}(\iota^\infty_n(y),\iota_n(y')),&\quad\textrm{ if $y\in X_\infty$, $y'\in X_n$ for some $n\in\N$},\\
\inf\limits_{x\in X_\infty}\sfd_{Y_n}(\iota^\infty_n(y),\iota^\infty_n(x))+\sfd_{Y_m}(\iota^\infty_m(y),\iota^\infty_m(x)),&\quad\textrm{ if $y\in X_n$, $y'\in X_m$ for some $n,m\in\N$}.
\end{array}
\right.
\end{split}
\]
Next we  identify points  $x,y\in Y$ such that $\sfd_Y(x,y)=0$: it is readily verified that the resulting space, which we shall still denote by $(Y,\sfd_Y)$ is separable and up to passing to the completion we can also assume that $Y$ is complete. By construction, the set  $\iota_n(X_n)\cup\iota^\infty_n(X_\infty)\subset Y_n$ endowed with the distance $\sfd_{Y_n}$ is canonically  isometrically embedded in $(Y,\sfd_Y)$ so that there are canonical isometric embeddings $\iota'_n$ of $X_n$ into $Y$, $n\in\bar \N$. We claim that $(Y,\sfd_Y)$ and the $\iota_n'$'s provide an effective realization of the extrinsic convergence. The fact that $\sfd_Y(\iota'_n(\bar x_n),\iota'_\infty(\bar x_\infty))\to 0$ is obvious by \eqref{eq:divano} and similarly we obtain that $\z_{n,\psi}\to\z_{\infty,\psi}$. For the weak convergence of $(\iota_n')_\sharp\mm_n$ to $(\iota_\infty')_\sharp\mm_\infty$ fix $\varphi\in\rmC_{bs}(Y)$ and notice that the boundedness of $\supp(\varphi)$ and the continuity of $\psi$ ensure that
\begin{equation}
\label{eq:divano2}
\text{the sequence }\quad\frac{\z_{n,\psi}\varphi(\cdot)}{\psi(\sfd_Y(\cdot,\iota_n'(\bar x_n)))}\quad\text{ uniformly converges to }\quad\frac{\z_{\infty,\psi}\varphi(\cdot)}{\psi(\sfd_Y(\cdot,\iota_\infty'(\bar x_\infty)))}
\end{equation}
on $\supp(\varphi)$ as $n\to\infty$. Therefore
\[
\begin{split}
\lim_{n\to\infty}\int\varphi\,\d(\iota_n')_\sharp\mm_n&=\lim_{n\to\infty}\int\frac{\z_{n,\psi}\varphi(\cdot)}{\psi(\sfd_Y(\cdot,\iota_n'(\bar x_n)))}\,\d(\iota_n')_\sharp\mm_{n,\psi}\\
&=\int\frac{\z_{\infty,\psi}\varphi(\cdot)}{\psi(\sfd_Y(\cdot,\iota_\infty'(\bar x_\infty)))}\,\d(\iota_\infty')_\sharp\mm_{\infty,\psi}=\lim_{n\to\infty}\int\varphi\,\d(\iota_\infty')_\sharp\mm_\infty,
\end{split}
\]
having used \eqref{eq:divano2} and the weak convergence of $(\iota_n')_\sharp\mm_{n,\psi}$ to $(\iota_\infty')_\sharp\mm_{\infty,\psi}$ granted by the fact that $W_2((\iota_n')_\sharp\mm_{n,\psi},(\iota_\infty')_\sharp\mm_{\infty,\psi})\to 0$ (from \eqref{eq:divano}) and Proposition \ref{prop:narw2}.

\noindent  \textbf{(B) $\Rightarrow$ (A)} Keeping the same notation of \eqref{eq:92}, we observe that 
(A) is equivalent to 
\begin{equation}
  \begin{gathered}
    \text{for every $N\in \N$ the sequence } \mm^N_n:=
    \big(\jey^N_{X_n}\big)_\sharp\big(\delta_{\bar x_n}\otimes
    (\mm_n)^{\otimes (N-1)}\big),\ n\in \N,\\ 
    \text{weakly converges to } 
    \mm^N_\infty:=
    \big(\jey^N_{X_\infty}\big)_\sharp
    (\delta_{\bar x_\infty}\otimes
    (\mm_\infty)^{\otimes (N-1)})\text{ in
      $\Mloc {\G^N}$ as $n\to\infty$}.
  \end{gathered}
\label{eq:126}
\end{equation}
Now notice that by (B) and Definition \eqref{def:Dconv}, 
  there exists a separable metric space
  $(Y,\sfd_Y)$ and isometries
  $\iota_n:X_n\to Y$ such that 
  $\nn_n=(\iota_n)_\sharp \mm_n$ and 
  $\bar y_n=\iota_n(\bar x_n)$ satisfy \eqref{eq:19bis}. 
  In particular the measures $\mm_n^N$ of 
  \eqref{eq:126} admit the representations 
  $\mm^N_n=\big(\jey^N_{Y}\big)_\sharp
  \big(\delta_{\bar y_n}\otimes(\nn_n)^N\big)$
  for every $n\in \bar \N$, 
  so that \eqref{eq:126} follows immediately from \eqref{eq:19bis}.

\noindent  \textbf{(A) $\Rightarrow$ (D)}   It is sufficient to prove that for every $k\in\Z$ we have $\pGW^{fm}(\cX_{n,[k]},\cX_{\infty,[k]})\to0$ as $n\to\infty$. Thus fix $k\in\Z$, recalling the definition of $Z_R$ given by \eqref{eq:98} and 
  set $\varphi_k:=Z_{2^k}\varphi$ 
  to deduce that
   $\varphi \big[\cX_{n,[k]}\big]=
  \varphi_k \big[\cX_n\big]$ for every $n\in \bar \N$.
  Hence by assumption 
  for every function $\varphi\in \rmC_b(\G^N)$  
  we have $\lim_{n\to\infty}\varphi \big[\cX_{n,[k]}\big]=
  \varphi \big[\cX_{\infty,[k]}\big]$.

  Now pick $\lambda>\sup_{n\in\bar\N}\mm_{n,[k]}(X_n)$ and consider the (nonpointed) metric measure spaces
  $\cX^*_{n,[k],\lambda}:=
  [X_n,\sfd_n,\mm_{n,[k]}+\lambda \delta_{\bar x_n}]$: by Remark \ref{rem:tedious-image2}, Theorem \ref{thm:greven} and taking into account that $\mm_{n,[k]}(X_n)\to\mm_{\infty,[k]}(X_\infty)$ we deduce that $\GW(\cX^*_{n,[k],\lambda},\cX^*_{\infty,[k],\lambda})\to0$ as $n\to\infty$.

With the very same construction used to prove the implication (C)$\Rightarrow$(B) above, we can produce a complete separable space $(Y,\sfd_Y)$ and isometric immersions $\iota_n:X_n\to Y$, $n\in\bar \N$ such that $(\iota_n)_\sharp(\mm_{n,[k]}+\lambda \delta_{\bar x_n})$ weakly converges to  $(\iota_\infty)_\sharp(\mm_{\infty,[k]}+\lambda \delta_{\bar x_\infty})$. 

To conclude is therefore sufficient to show that $\iota_n(\bar x_n)\to \iota_\infty(\bar x_\infty)$, as this would also give that  $(\iota_n)_\sharp \mm_{n,[k]}$ weakly converges to  $(\iota_\infty)_\sharp \mm_{\infty,[k]}$. But this is obvious by the choice of $\lambda$. Indeed, for $r>0$ we can consider the map $\nchi_r\in C_{bs}(Y)$ given by $\nchi_r(y):=0\lor (1-\frac{\sfd(y,\iota_\infty(x_\infty))}r)$ so that
\[
\lambda+\limi_{n\to\infty}\nchi_r(\iota_n(\bar x_n))>\limi_{n\to\infty}\int\nchi_r\,\d(\iota_n)_\sharp(\mm_{n,[k]}+\lambda \delta_{\bar x_n})=\int\nchi_r\,\d(\iota_\infty)_\sharp(\mm_{\infty,[k]}+\lambda \delta_{\bar x_\infty})\geq\lambda,
\]
which gives $\limi_{n\to\infty}\nchi_r(\iota_n(\bar x_n))>0$ and thus $\lims_{n\to\infty}\sfd(\iota_n(\bar x_n),\iota_\infty(\bar x_\infty))<r$. Being $r>0$ arbitrary, the claim follows.

\noindent  \textbf{(D) $\Rightarrow$ (C)} By the assumption we have $\pGW^{fm}(\cX_{n,[k]},\cX_{\infty,[k]})\to 0$ as $n\to\infty$ for every $k\in\Z$. Given $R>0$ we can find $k\in \Z$ such that $R<2^{k}$  so that recalling the definition of the measures $\mm_{[k]}$ given in \eqref{eq:mk} we deduce
\[
 \lim_{n\to\infty}\int 0\lor\big(1- \sfd_n(\cdot,B_R(\bar x_n))  \big)\,\d\mm_n=\int 0\lor\big(1- \sfd_\infty(\cdot,B_R(\bar x_\infty))  \big)\,\d\mm_\infty<\infty.
\]
Using the bound  $\mm_n(B_R(\bar x_n))\leq\int 0\lor(1- \sfd_n(\cdot,B_R(\bar x_n))  )\,\d\mm_n$ valid for every $n\in\N$, we deduce  that there exists a continuous non-decreasing $\phi:[0,\infty)\to[0,\infty)$ such that
\[
\sup_{n\in\bar\N}\mm_n(B_R(\bar x_n))\leq \phi(R),\qquad\forall R>0.
\]
Now define $\psi:[0,\infty)\to[0,\infty)$ as in \eqref{eq:daphiapsi} and use Proposition \ref{prop:weights} to deduce that $\cX_n\in\X^\psi$ for every $n\in\bar\N$ and
\[
\sup_{n\in\bar\N}\int (1+\sfd_n^3(\cdot,\bar x_n))\psi(\sfd(\cdot,\bar x_n))\,\d\mm_n\leq\int_0^\infty\frac1{1+r^3}\,\d r =:{\sf S}<\infty,
\]
from which it follows that
\begin{equation}
\label{eq:lezione}
\int\limits_{X_n\setminus B_R(\bar x_n)}\psi(\sfd_n(\cdot,\bar x_n))\,\d\mm_n\leq\frac{\sf S}{R^3}, \qquad \qquad \int\limits_{X_n\setminus B_R(\bar x_n)}\sfd^2_n(\cdot,\bar x_n)\psi(\sfd_n(\cdot,\bar x_n))\,\d\mm_n\leq\frac{\sf S}{R},
\end{equation}
for every $n\in\bar \N$.
Fix $k\in\Z$, recall that by definition of $\pGW$ we have $\pGW^{fm}(\cX_{n,[k]},\cX_{\infty,[k]})\to 0$  as $n\to\infty$ and then use the very same construction used in the proof of (C)$\Rightarrow$(B) above to deduce the existence of a complete and separable metric space $(Y,\sfd_Y)$ and isometric embeddings $\iota_n$ of $X_n$ in $Y$, $n\in\bar\N$, such that
\begin{equation}
\label{eq:losappiamo}
\lim_{n\to\infty}\Big|\log\frac{\mm_{n,[k]}(X_n)}{\mm_{\infty,[k]}(X_\infty)} \Big|+\sfd_Y(\iota_n(\bar x_n),\iota_\infty(\bar x_\infty))+W_c((\iota_n)_\sharp\tilde\mm_{n,[k]},(\iota_\infty)_\sharp\tilde\mm_{\infty,[k]})=0,
\end{equation}
where $\d\mm_{n,[k]}:=\zeta(\sfd_n(\cdot,\bar x_n)2^{-k})\d\mm_n$ ($\zeta$ being given in \eqref{eq:zeta}) and $\tilde\mm_{n,[k]}=\mm_{n,[k]}(X_n)^{-1}\mm_{n,[k]}$, for every $n\in\bar\N$. In particular, the sequence of measures $n\mapsto (\iota_n)_\sharp\mm_{n,[k]}=\mm_{n,[k]}(X_n)\tilde\mm_{n,[k]}$ weakly converges to $(\iota_\infty)_\sharp\mm_{\infty,[k]}=\mm_{\infty,[k]}(X_\infty)\tilde\mm_{\infty,[k]}$ and therefore
\[
\lim_{n\to\infty}\int \psi(\sfd_n(\cdot,\bar x_n))\zeta(\sfd_n(\cdot,\bar x_n)2^{-k})\,\d\mm_n=\int\psi(\sfd_\infty(\cdot,\bar x_\infty))\zeta(\sfd_\infty(\cdot,\bar x_\infty)2^{-k})\,\d\mm_\infty,
\]
from which it follows, taking into account that $\supp(1-\zeta(\sfd_n(\cdot,\bar x_n)2^{-k}))\subset X_n\setminus B_{2^k}(\bar x_n)$ for every $n\in\bar \N$ and the first in \eqref{eq:lezione}, that
\[
\lims_{n\to\infty}\left|\int\psi(\sfd_n(\cdot,\bar x_n))\,\d\mm_n-\int\psi(\sfd_\infty(\cdot,\bar x_\infty))\,\d\mm_\infty\right|\leq \frac{{\sf S}}{2^{3k-2}}.
\]
Introducing the constants $\z_{n,\psi}$ and the probability measures $\mm_{n,\psi}$ as in \eqref{eq:normalizzare}, the last inequality can be rewritten as $\lims_{n\to\infty}|\z_{n,\psi}-\z_{\infty,\psi}|\leq \frac{{\sf S}}{2^{3k-2}}$, and given that this holds for every $k\in\Z$ we deduce
\begin{equation}
\label{eq:perz2}
\lim_{n\to\infty}\z_{n,\psi}=\z_{\infty,\psi}.
\end{equation}
Now define $\rho_n:=\frac{\d\tilde\mm_{n,[k]}}{\d\mm_{n,\psi}}$, $n\in\bar\N$, and build a transport plan $\ggamma_{n,k}\in\Adm(\mm_{n,\psi},\tilde\mm_{n,[k]})$ by `letting the mass in common stand still and uniformly distributing the rest' or, more rigorously, define 
\[
\ggamma_{n,k}:=(\ii,\ii)_\sharp(\mm_{n,\psi}\restr{\{\rho_{n}\leq 1\}})+\frac{\mm_{n,\psi}\restr{\{\rho_{n}>1\}} \otimes(\tilde\mm_{n,[k]}-\mm_{n,\psi}\restr{\{\rho_{n}\leq 1\}})}{\mm_{n,\psi}\restr{\{\rho_{n}> 1\}}(X_n)}.
\]
It is readily checked that indeed $\ggamma_{n,k}\in\Adm(\mm_{n,\psi},\tilde\mm_{n,[k]})$. Also, by construction, for $(x,y)\in\supp(\ggamma_{n,k})$ and $x\neq y$ we have $x\notin B_{2^k}(\bar x_n)$ and $y\in B_{2^k}(\bar x_n)$ and thus $\sfd_n(x,y)\leq 2\sfd(x,\bar x_n)$. Taking into account the second in \eqref{eq:lezione} we deduce
\begin{equation}
\label{eq:unifk}
\begin{split}
W^2_2(\mm_{n,\psi},\tilde\mm_{n,[k]})&\leq \int_{\{x\neq y\}}\sfd_n^2(x,y)\,\d\ggamma_{n,k}(x,y)\\
&\leq  \frac{4}{\z_{n,\psi}}\int_{X\setminus B_{2^k}(\bar x_n)}\sfd^2(\cdot,\bar x_n)\psi(\sfd_n(\cdot,\bar x_n))\,\d\mm_n \leq \frac{{\sf c}}{2^k},\qquad\forall n\in\bar\N,
\end{split}
\end{equation}
where ${\sf c}:=\sup_{n\in\bar\N}4(\z_{n,\psi})^{-1}$ (notice that  ${\sf c}<\infty$ by \eqref{eq:perz2} and the fact that $\z_{\infty,\psi}\neq 0$). To conclude, recall that $(\iota_n)_\sharp\tilde\mm_{n,[k]}$ weakly converges to $(\iota_{\infty})_\sharp\tilde\mm_{\infty,[k]}$ as $n\to\infty$ in $\prob Y$ and thus, given that they have uniformly bounded support, by Proposition \ref{prop:narw2} we get that $W_2((\iota_n)_\sharp\tilde\mm_{n,[k]},(\iota_{\infty})_\sharp\tilde\mm_{\infty,[k]})\to0 $ as $n\to\infty$. Hence
\[
\begin{split}
\lims_{n\to\infty}W_2((\iota_n)_\sharp\mm_{n,\psi},(\iota_{\infty})_\sharp\mm_{\infty,\psi})&\leq\lims_{n\to\infty}W_2((\iota_n)_\sharp\mm_{n,\psi},(\iota_{n})_\sharp\tilde\mm_{n,[k]})\\
&\qquad\qquad+\lims_{n\to\infty}W_2((\iota_\infty)_\sharp\tilde\mm_{\infty,[k]},(\iota_{\infty})_\sharp\mm_{\infty,\psi})\stackrel{\eqref{eq:unifk}}\leq2\sqrt{\frac{{\sf c}}{2^k}}.
\end{split}
\]
In summary, choosing $(Y,\sfd_Y)$ and the embeddings $\{\iota_n\}_{n\in\bar\N}$ in the definition of $\D^\psi$, by \eqref{eq:losappiamo}, \eqref{eq:perz2} and the last bound we get $\lims_{n\to\infty}\D^\psi(\cX_n,\cX_\infty)\leq  2\sqrt{\frac{{\sf c}}{2^k}}$. Eventually letting $k\to+\infty$ we conclude.
\end{proof}

By virtue of Theorem \ref{thm:main_convergence} we can propose the following definition:
\begin{definition}[Pointed measured Gromov (pmG) convergence]\label{def:pGw}
  Let $\cX_n=[X_n,\sfd_n,\mm_n,\bar x_n]$, $n\in  \bar\N$. We say that the sequence $n\mapsto\cX_n$ converges to $\cX_\infty$ in the pointed measured Gromov  sense (pmG-sense, for short) provided any of the 4 equivalent statements in Theorem \ref{thm:main_convergence} holds.
\end{definition}

\subsection{Basic topological properties of the pmG-convergence and some comments}\label{se:topcomm}
We collect here some basic properties of the space of \pmm\ spaces equipped with pmG-convergence. We start with the following result:
\begin{theorem}
  \label{thm:completeness}
  The space 
  $\X$ endowed with the distance $\pGW$ is
  a complete and separable metric space.
\end{theorem}
\begin{proof}$\ $\\
\noindent{\bf Completeness.} Let $n\mapsto\cX_n=[X_n,\sfd_n,\mm_n,\bar x_n]\in \X$   be a $\pGW$-Cauchy sequence. Then for $k\in\Z$ the sequence $(\cX_{n,[k]})\subset \X_{fm}$ is $\pGW^{fm}$-Cauchy. With the same gluing procedure used in the proof of (C)$\Rightarrow$(B) in Theorem \ref{thm:main_convergence} we can find a complete separable space $(Y_k,\sfd_{Y_k})$ and isometric embeddings $\iota_{n,k}:X_n\to Y_k$, $n\in\bar \N$, such that
\begin{equation}
\label{eq:percompl}
\lim_{n,m\to\infty}\Big|\log\frac{\mm_{n,[k]}(X_n)}{\mm_{m,[k]}(X_m)}\Big|+\sfd_{Y_k}(\iota_{n,k}(\bar x_n),\iota_{m,k}(\bar x_m))+W_c((\iota_{n,k})_\sharp\tilde\mm_{n,[k]},(\iota_{m,k})_\sharp\tilde\mm_{m,[k]})=0,
\end{equation}
the measures $\mm_{n,[k]}$ being defined as in \eqref{eq:mk}  and $\tilde \mm_{n,[k]}$ being their normalization. In particular, we see that
\[
\text{ the sequence $n\mapsto \mm_{n,[k]}(X_n)$ has limit $\z_k\in(0,\infty)$ for every $k\in\Z$,}
\]
and that
\[
\text{ the sequence $n\mapsto \iota_{n,k}(\bar x_n)$ has limit $\bar y_k\in Y_k$ for every $k\in\Z$,}
\]
and from the completeness of $(\prob{Y_k},W_c)$ (recall \eqref{eq:topwc}) we also deduce that there exists $\nn_k\in\prob{Y_k}$ such that $W_c(\nn_k,(\iota_{n,k})_\sharp\tilde\mm_{n,[k]})\to 0$ as $n\to\infty$. It is then clear that the sequence of nonpointed metric measure spaces  $n\mapsto[X_n,\sfd_n,\mm_{n,[k]}]$ converges to $[Y_k,\sfd_{Y_k},\z_k\nn_k]$ w.r.t. $\GW^{fm}$. 

Now fix $k'<k-2$ and let $f_n:Y_k\to\R$ be given by $f_n(x):=\zeta\big(\sfd_{Y_k}(x,\iota_{n,k}(\bar x_n))2^{-k'-1}\big)$. Let $\bar y_k\in Y_k$ be the limit of $n\mapsto\iota_{n,k}(\bar x_n)$ (whose existence is granted by \eqref{eq:percompl}) and notice that $(f_n)$ uniformly converges to $f(x):=\zeta\big(\sfd_{Y_k}(x,\bar y_k)2^{-k'-1}\big)$ from which it follows that $f_n\iota_{n,k}\mm_{n,[k]}$ weakly converges to $\z_kf\nn_k$. In other words, for $k'<k$ the space $(Y_k,\sfd_{Y_k})$ and the embeddings $\{\iota_{n,k}\}_{n\in\N}$ can be used as competitor in the definition of $\GW^{fm}([X_n,\sfd_n,\mm_{n,[k']}],[X_m,\sfd_m,\mm_{m,[k']}])$ to check that such sequence is $\GW^{fm}$-Cauchy.  It is then clear that it converges to $[Y_k,\sfd_{Y_k},\z_kf\nn_k]$.

The same argument also gives
\[
\begin{split}
(\z_k\nn_k)(B_{2^{k'+2}}(\bar y_k))&\geq \int f \,\d\z_k\nn_k=\lim_{n\to\infty}\int f_n\,\d\mm_{n,[k]}\\
&\geq\lims_{n\to\infty} \mm_{n,[k']}(B_{2^{k'+1}}(\bar x_n))=\lims_{n\to\infty} \mm_{n,[k']}(X_n)=\z_{k'}>0,
\end{split}
\]
and being this true for every $k'<k-2$, letting $k'\downarrow-\infty$ we obtain $\bar y_k\in\supp(\nn_k)$. This shows that $n\mapsto\cX_{n,[k]}$ converges to $(Y_k,\sfd_k,\nn_k,\bar y_k)$ w.r.t. $\pGW^{fm}$.

To conclude, we need to show that there exists a \pmm\ space $\cY:=[Y,\sfd,\nn,\bar y]$ such that $\cY_{[k]}=[Y_k,\sfd_k,\nn_k,\bar y_k]$, but this also follows by the above compatibility argument. Indeed, for $k'<k$ we know that there exists an isometry $\iota_{k'}^k:(\supp\nn_{k'},\sfd_{Y_{k'}})\to (\supp\nn_{k},\sfd_{Y_k})$ such that $\iota_{k'}^k(\bar y_{k'})=\bar y_k$ and $(\iota_{k'}^k)_\sharp\nn_{k'}=\zeta(\sfd_{Y_k}(\cdot,\bar y_k)2^{-k'})\nn_k$, hence the conclusion follows with the same gluing argument we already used in the proof of (C)$\Rightarrow$(B) of Theorem \ref{thm:main_convergence}.

\noindent{\bf Separability.} We simply 
  observe that the map
\[
\begin{array}{rcl}
B:\X&\to&  \displaystyle{\prod_{N=1}^\infty\Mloc{\G^N}}\\
\\
\cX & \mapsto& (\mm_\cX^N)_{N\in\N} 
\end{array}
\]
  is an homeomorphism of $\X$
  with its image $B(\X)$ endowed with the product
  topology inherited by 
  $\prod_{N=1}^\infty \Mloc{\G^N}.$
  Since this topology is separable and metrizable,
  we conclude that $\X$ is separable as well.    
\end{proof}

\begin{remark}{\rm
It is worth to underline that for any given $\psi$ as in \eqref{eq:proprpsi}, the space $(\X^\psi,\D^\psi)$ is \emph{not} complete. The problem is that the condition $\bar x\in\supp(\mm)$ can be not satisfied  in the limit.
}\fr\end{remark}

\begin{remark}[Different weight functions]\label{rem:diffpesi}{\rm
In the proof of the implication (D)$\Rightarrow$(C) of Theorem \ref{thm:main_convergence} we used the explicit formula \eqref{eq:daphiapsi} for the weight function $\psi$ only to deduce the bounds \eqref{eq:lezione}. The very same arguments used there actually show that the following slightly stronger statement holds: if $n\mapsto\cX_n=[X_n,\sfd_n,\mm_n,\bar x_n]$ is a sequence converging to $\cX_\infty=[X_\infty,\sfd_\infty,\mm_\infty,\bar x_\infty]$ in the pmG-sense and $\psi$ as in  \eqref{eq:proprpsi} is such that
\begin{equation}
\label{eq:mah}
\cX_n\in\X^\psi,\ \forall n\in \N,\qquad\qquad\text{and }\qquad\qquad \lim_{R\to\infty}\sup_{n\in\N}\int_{X_n\setminus B_R(\bar x_n)} \sfd^2(\cdot,\bar x_n)\psi(\sfd(\cdot,\bar x_n))\,\d\mm_n=0,
\end{equation}
then 
\[
\lim_{n\to\infty}\D^\psi(\cX_n,\cX_\infty)=0.
\]
In particular, we see that if   $\psi_1.\psi_2$ are as in  \eqref{eq:proprpsi} and  $\psi_1(r)\leq \psi_2(r)$ for every $r$ sufficiently big, then $\X^{\psi_2}\subset \X^{\psi_1}$ and
\[
(\cX_n)_{n\in\bar\N}\subset \X^{\psi_2},\qquad\lim_{n\to\infty}\D^{\psi_2}(\cX_n,\cX_\infty)= 0\qquad\Rightarrow\qquad\lim_{n\to\infty}\D^{\psi_1}(\cX_n,\cX_\infty)= 0.
\]
}\fr\end{remark}

\begin{remark}[A limit case concerning weights]\label{re:occhiopesi}{\rm In connection with Remark \ref{rem:diffpesi} above it is worth to underline that it may happen that $(\cX_n)$ $\pGw$-converges to $\cX_\infty$, that for some $\psi$ as in \eqref{eq:proprpsi} we have $\cX_n\in\X^\psi$ for every $n\in\bar\N$, yet that $\D^\psi(\cX_n,\cX_\infty)$ does not go to 0.

An explicit example in this direction is given by $\cX_n:=[\R,\sfd_{\rm Eucl},\mu_n, 0]$, $n\in\bar \N$, where $\sfd_{\rm Eucl}$ is the standard Euclidean distance and $n\mapsto\mu_n\in\probt \R$ is a sequence  weakly converging to  $\mu_\infty\in\probt\R$  which is not $W_2$-converging. Then the choice $\psi\equiv 1$ produces the desired behavior.

It is then easy to see that if $\cX_n\in\X^\psi$ for every $n\in\bar\N$ and $\cX_n\stackrel{\pGw}\to\cX_\infty$ but $\D^\psi(\cX_n,\cX_\infty)$  does not go to 0, then for every $\psi'$ such that $\lim_{r\to\infty}\frac{\psi'(r)}{\psi(r)}=0$ we have  $\D^{\psi'}(\cX_n,\cX_\infty)\to 0$, as the faster decrease of $\psi'$ ensures the 2-uniform integrability of the rescaled measures.
}\fr\end{remark}
\begin{remark}[Distortion distances] \label{rk:Dist}{\rm
In \cite{Sturm-12} Sturm introduced the distortion distance between normalized m.m.\ spaces. The induced notion of convergence of sequences of m.m.\ spaces can easily be adapted to the case of p.m.m.\ spaces along the same lines used here. Then Corollary 2.10 of \cite{Sturm-12} shows that for a given sequence $(\cX_n)$ of normalized p.m.m.\ spaces with uniformly bounded diameter, convergence in the distortion distance to a limit space $\cX_\infty$ is equivalent to $\pGw$-convergence.
}\fr\end{remark}

By a diagonal argument, 
the characterization of Gromov-weak convergence
given by Theorem \ref{thm:main_convergence}(C), 
and Proposition 7.1 in \cite{Greven09}
we obtain the following compactness result.
\begin{corollary}[Compactness]
\label{rem:Compactness} 
Let  $\cX_n=[X_n,\sfd_n,\mm_n,\bar{x}_n]$, 
$n\in\N$ be a sequence of \pmm~spaces 
and let $\cX_{n,[k]}$ be defined as in \eqref{eq:mk2}.
$(\cX_n)_{n\in\N}$ is precompact in $\X$ 
if and only if for every $k\in\Z$ 
the sequence $(\cX_{n,[k]})_{n\in\N}$ is precompact in 
$\X_{fm}$.

In particular $(\cX_n)_{n\in\N}$ is precompact in $\X$  
if and only if
\begin{enumerate}[1.]
\item for every $k\in \Z$ 
the sequence 
$n\mapsto \log\big(\mm_n(B_{2^k}(\bar x_k))\big)$
is bounded;
\item
for every $k\in \Z$ and 
$\varepsilon>0$ 
there exists $N_{\varepsilon,k} \in \N$ 
and subsets $X_{\varepsilon,k,n} \subset
B_{2^k}(\bar x_n)\subset X_n$, $n\in \N$, such that 
$\mm_n(B_{2^k}(\bar x_n)\setminus X_{\varepsilon,k,n})\le \varepsilon$
and 
$X_{\varepsilon,k,n}$ 
can be covered by at most $N_{\varepsilon,k}$ balls of radius $\varepsilon.$

\end{enumerate}

\end{corollary}

We conclude the section pointing out that the extrinsic approach to pmG-convergence is the only one, out of the 4 that we analyzed, which relies on some additional 
structure: it needs not only the equivalence classes $\{\cX_n\}_{n\in\bar\N}$, but also  the space $(X,\sfd)$ and the isometric embeddings $\{\iota_n\}_{n\in\bar\N}$. 

As said, we  refer to such $(X,\sfd)$ and $\iota_n$'s as an \emph{effective realization} of the pmG-convergence. To choose an effective realization amounts, in a sense, to choose `the way the $\cX_n$'s are converging to $\cX_\infty$'. To see this, consider the following definition:
\begin{definition}[Convergence of points belonging to converging spaces]\label{def:convpoint} Let $\cX_n=[X_n,\sfd_n,\mm_n,\bar x_n]$ be converging to $\cX_\infty=[X_\infty,\sfd_\infty,\mm_\infty,\bar x_\infty]$ in the pmG-sense and $(X,\sfd)$,  $\{\iota_n\}_{n\in\bar\N}$ an effective realization of the convergence. 

Then we say that $n\mapsto x_n\in \supp(\mm_n)$ is converging to $x_\infty\in\supp(\mm_\infty)$ via such realization provided
\[
\sfd\big(\iota_n(x_n),\iota_\infty(x_\infty)\big)\to 0,
\]
as $n\to\infty$.
\end{definition}
In other words, we are identifying $(X_n,\sfd_n,\mm_n,\bar x_n)$ with the isomorphic space $(X,\sfd,(\iota_n)_\sharp\mm_n,\iota_n(\bar x_n))$, $n\in\bar \N$, and then reading the convergence at the level of points in $X$.

It is possible to see that if the limit space $\cX_\infty$ admits no non-trivial automorphisms, then such notion of convergence is in fact independent on the effective realization (this is in fact an if-and-only-if, the argument makes use of the compactness  given by Proposition \ref{prop:isocomp}, we omit the details). Yet, in general, to fix an effective realization truly affects this convergence; to see this just consider the case where all the $\cX_n$'s are $\R^2$ equipped with the Euclidean distance, the Lebesgue measure and  pointed at the origin. It is then clear that in embedding all of them in a common, say, $\R^2$, we are free to `rotate' each one of any desired angle.

In this sense, to fix an effective realization is a non-intrinsic choice. Yet, from the technical point of view to have at disposal the above notion of convergence of points is very useful in order both to state and to prove results about converging sequences of spaces. Hence we shall adopt this point of view when speaking  of pmG-convergence of spaces with a lower bound on the Ricci, see Remark \ref{re:assnot}. The fact that the results proven this way are intrinsic can then be recovered noticing that they do not rely on the particular effective realization chosen.

\subsection{Relation with pointed measured Gromov-Hausdorff convergence}\label{se:relpmgh}


We now analyze the relation between  $\pGw$-convergence and
pointed measured Gromov-Hausdorff convergence. The definition below is
adapted from Definition 8.1.1 in \cite{Burago-Burago-Ivanov01} adding
the requirement of \narrow\ convergence at the level of measures. 

\begin{definition}[Pointed measured Gromov Hausdorff convergence]\label{def:pmGH}
Let $(X_n,\sfd_n,\mm_n,\bar x_n)$, $n\in\bar\N$,
be pointed metric measure spaces as in \pmmaxioms.
We say that $(X_n,\sfd_n,\mm_n,\bar x_n)\to
(X_\infty,\sfd_\infty,\mm_\infty,\bar x_\infty)$ in the pointed
measured Gromov Hausdorff (pmGH) sense, provided for any $\eps,R>0$ there exists $N({\eps,R})\in \N$ such that for all $n\geq N({\eps,R})$ there exists a Borel map $f^{R,\eps}_n:B_R(\bar x_n)\to X_\infty$ such that
\begin{itemize}
\item $f^{R,\eps}_n(\bar x_n)=\bar x_\infty$,
\item $\sup_{x,y\in B_R(\bar x_n)}|\sfd_n(x,y)-\sfd_\infty(f^{R,\eps}_n(x),f^{R,\eps}_n(y))|\leq\eps$,
\item the $\eps$-neighborhood of $f^{R,\eps}_n(B_R(\bar x_n))$ contains $B_{R-\eps}(\bar x_\infty)$,
\item  $(f^{R,\eps}_n)_\sharp(\mm_n\restr{B_R(\bar x_n)})$ \narrowly\  converges to $\mm_\infty\restr{B_R(x_\infty)}$ as $n\to\infty$, for a.e. $R>0$.
\end{itemize}
\end{definition}
\begin{remark}[All the space matters]\label{re:pmgh}{\rm
Notice the technical difference between this definition and the one of  $\pGw$-convergence. Here the entire sets $X_n$, $n\in\bar\N$, matter, and not only the portion $\supp(\mm_n)$.  This means in particular that this notion of convergence is only defined for p.m.m.\ spaces and not for their equivalence classes.

It is certainly possible to force the choice of the representative in the class by requiring to deal with spaces $(X,\sfd,\mm,\bar x)$ such that $\supp(\mm)=X$, but this would produce examples of non-convergence like the one discussed at the end of the introduction.
}\fr\end{remark}
\begin{remark}[Uniform vs weak convergence]\label{re:l2li}{\rm
Another, more important, difference between pmGH convergence and $\pGw$-convergence is that in the pmGH case distances are asked to converge uniformly on bounded sets, whereas for $\pGw$  the convergence of distances is only required implicitly via the weak convergence of measures. This means that in the pmG-case distances are only required to converge in an appropriate weighted, or weak, sense.
}\fr\end{remark}
\begin{remark}{\rm
Due to the aforementioned uniform convergence on bounded sets, typically  when speaking about pmGH-convergence one assumes  the spaces to be proper (i.e. bounded closed set are
compact).
We didn't do so just to keep a slightly higher level of generality, but this is not the main point in the discussion.
}\fr\end{remark}
It is worth to point out that the maps $f_n^{R,\eps}$ in the definition of pmGH-convergence can be chosen to be independent on $R,\eps$:
\begin{proposition}[Equivalent definition of pointed mGH convergence]\label{prop:equivpmgh}
Let $(X_n,\sfd_n,\mm_n,\bar x_n)$, $n\in\bar\N$,
be pointed metric measure spaces as in \pmmaxioms.
Then the following are equivalent.
\begin{itemize}
\item[A)] $(X_n,\sfd_n,\mm_n,\bar x_n)\to (X_\infty,\sfd_\infty,\mm_\infty,\bar x_\infty)$ in the pointed measured Gromov-Hausdorff sense.
\item[B)] There are sequences $R_n\uparrow+\infty$, $\eps_n\downarrow 0$ and Borel maps $f_n:X_n\to X_\infty$ such that 
\begin{itemize}
\item[1)] $f_n(\bar x_n)=\bar x_\infty$,
\item[2)] $\sup_{x,y\in B_{R_n}(\bar x_n)}|\sfd_n(x,y)-\sfd_\infty(f_n(x),f_n(y))|\leq\eps_n$,
\item[3)] the $\eps_n$-neighborhood of $f_n(B_{R_n}(\bar x_n))$ contains $B_{R_n-\eps_n}(\bar x_\infty)$,
\item[4)] for any $\varphi\in \Cb{X_\infty}$ with bounded support it holds $\lim\limits_{n\to\infty}\int \varphi\circ f_n\,\d\mm_n=\int\varphi\,\d\mm_\infty$.
\end{itemize}
\end{itemize}
\end{proposition}
\begin{proof}
The implication $(B)\Rightarrow (A)$ is obvious, so we prove $(A)\Rightarrow(B)$. Let $\bar R_k\uparrow +\infty$, $\bar\eps_k\downarrow0$ be two arbitrary sequences and put $N_k:=N(\bar\eps_k,\bar R_k)$. It is not restrictive to assume that $N_{k+1}>N_k$ for any $k\in\N$ and that $N_k\uparrow+\infty$. Define $R_n\uparrow+\infty$ and $\eps_n\downarrow 0$ as: $R_n$ and $\eps_n$ are arbitrary for $n<N_1$, while for $N_k\leq n<N_{k+1}$ we put $R_n:=\bar R_k$ and $\eps_n:=\bar \eps_k$. 

The functions $f_n:X_n\to X_\infty$ are then defined as follows: $f_n$ is arbitrary for $n<N_1$, for $N_k\leq n<N_{k+1}$, $f_n$ is arbitrary on $B_{R_n}^c(\bar x_n)$ as well, while on $B_{R_n}(\bar x_n)$ we put $f_n:=f_n^{R_n,\eps_n}$.

Possibly redefining $(R_n)$, $(\eps_n)$ and $(f_n)$ for $n<N_1$ in order to achieve $(B2)$ and $(B3)$ even for small $n$'s, it is readily checked that the construction gives the thesis.
\end{proof}
\begin{remark}\label{re:pmghgeod}{\rm
With a little bit of work one can see that if the spaces $(X_n,\sfd_n)$, $n\in\N$ are length spaces (i.e. the distance is realized as infimum of length of curves), then $(B)$ above can be replaced by 
\begin{itemize}
\item[B')] There are sequences $R_n\uparrow+\infty$, $\eps_n\downarrow
  0$ and Borel maps $f_n:X_n\to X_\infty$ such that 
\begin{itemize}
\item[1')] $f_n(\bar x_n)=\bar x_\infty$,
\item[2')] $\sup_{x,y\in B_{R_n}(\bar x_n)}|\sfd_n(x,y)-\sfd_\infty(f_n(x),f_n(y))|\leq\eps_n$ and $f(B_{R_n}(\bar x_n))\subset B_{R_n}(\bar x_\infty)$
\item[3')] the $\eps_n$-neighborhood of $f_n(B_{R_n}(\bar x_n))$ contains $B_{R_n}(\bar x_\infty)$
\item[4')] for any $\varphi\in \Cb{X_\infty}$ with bounded support it holds $\lim\limits_{n\to\infty}\int \varphi\circ f_n\,\d\mm_n=\int\varphi\,\d\mm_\infty$.
\end{itemize}
\end{itemize}
Indeed, starting from $(B)$ we double  $\eps_n$ to achieve $(3')$, then we modify a bit $f_n$ so that the image of $f_n(B_{R_n}(\bar x_n))$ of $B_{R_n}(\bar x_n)$ is contained in $B_{R_n}(\bar x_\infty)$, in this way we achieve $(2)'$: this is done by shrinking a bit the original image along almost minimizing curves connecting to $\bar x_\infty$. With this procedure we alter distances of at most $3\eps_n$, thus we keep the desired convergence.

The approach $(B')$ is the one chosen by Villani in Definition 27.30 (third part) of \cite{Villani09}, so that all the stability statements proved in Chapter 29 of \cite{Villani09} under pointed mGH convergence are still true in our framework provided the involved spaces are length spaces (which will be always the case in the paper).

Let us finally mention that $(B')$ gives the same convergence as   the  Gromov-Hausdorff-Prokhorov metric studied in \cite{Abraham} in case of  locally compact length spaces endowed with locally finite measures; here the authors prove that such a metric induces a structure of Polish space  on the equivalence classes of p.m.m.s. as above.
}\fr\end{remark}

Due to Remark \ref{re:l2li}, the following result is not surprising.
\begin{proposition}[From pmGH to  $\pGw$-convergence]\label{prop:pmghD}
 Let
 $(X_n,\sfd_n,\mm_n,\bar x_n)$, $n\in\N$,
be  \pmm~spaces 
converging to $(X_\infty,\sfd_\infty,\mm_\infty,\bar x_\infty)$ in the pmGH-sense. Then 
$[X_n,\sfd_n,\mm_n,\bar x_n]\to [X_\infty,\sfd_\infty,\mm_\infty,\bar
x_\infty]$
 w.r.t.~ $\pGw$-convergence.
\end{proposition}
\begin{proof}
Let $(R_n)$, $(\eps_n)$ and $(f_n)$ as in part $(B)$ of Proposition \ref{prop:equivpmgh}. Define $\YY:=\sqcup_{n\in\bar\N}X_n$ and the separable pseudodistance $\sfd$ on $\YY$ as\[
\sfd(y_1,y_2):=\left\{
\begin{array}{ll}
\sfd_n(y_1,y_2),&\qquad\textrm{ if }y_1,y_2\in X_n,\\
\sfd_\infty(y_1,y_2),&\qquad\textrm{ if }y_1,y_2\in X_\infty,\\
\inf\limits_{y_1'\in B_{R_n}(x_n)}\sfd_n(y_1',y_1)+\sfd_\infty(f_n(y_1'),y_2),&\qquad\textrm{ if }y_1\in X_n,\ y_2\in X_\infty,\\
\sfd(y_2,y_1),&\qquad\textrm{ if }y_1\in X_\infty,\ y_2\in X_n.
\end{array}
\right.
\]
Identifying points $x,y\in Y$ such that $\sfd(x,y)=0$ and taking if necessary the completion, it is easy to see that $(Y,\sfd)$ is complete and separable. Also, it holds
\[
\sfd(y,f_n(y))\leq\eps_n\qquad\forall y\in X_n\subset \YY ,\ \textrm{ such that }\sfd(y,\bar x_n)\leq R_n.
\]
By definition we have  $\sfd(\bar x_n,\bar
x_\infty)=\eps_n\downarrow0$. 

Let us now prove that
\[
  \phi(r):=\sup_n\mm_n(B_r(\bar x_n))<\infty\quad\forevery r\ge0;
\]
Let us set $\eps:=\sup_n\eps_n$ and observe that 
choosing $\varphi(y):=0\lor (r+1+\eps-\sfd(y,\bar x_\infty))\land 1$ 
we have 
\begin{displaymath}
  \varphi(f_n(y))=1\quad\text{if }y\in X_n\cap B_r(\bar x_n)\text{ and
  }R_n\ge r,
\end{displaymath}
so that B4) yields
\begin{displaymath}
  \limsup_{n\to\infty}\mm_n(B_r(\bar x_n))\le 
  \limsup_{n\to\infty} \int_{X_n} \varphi\circ f_n\,\d\mm_n=
  \int_X \varphi\,\d\mm_\infty\le \mm_\infty(B_{r+1+\eps}(\bar x_\infty)).
\end{displaymath}
In order to conclude our proof it is sufficient to check that 
\begin{equation}
  \label{eq:29}
  \lim_{n\to\infty}\int\zeta\,\d\mm_n=\int \zeta\,\d\mm_\infty
\end{equation}
for every $\zeta\in \Cc \YY $. A standard approximation argument 
(see e.g.~\cite[\S~5.1]{Ambrosio-Gigli-Savare08}) 
shows that in is not restrictive to assume $\zeta$ $1$-Lipschitz.
If $\supp(\zeta)\subset B_R(\bar x_\infty)$ 
and $R+\eps\le R_n$ we have
\begin{displaymath}
  \Big|\int \zeta\,\d\mm_n-\int \zeta\circ f_n\,\d\mm_n\Big|\le 
  \int_{B_{R+\eps}(\bar x_n)} |\zeta-\zeta\circ f_n|\,\d\mm_n\le 
  \phi(R+\eps)\eps_n\rightarrow0\quad\text{as }n\up\infty.
\end{displaymath}
\eqref{eq:29} then follows by B4).
\end{proof}
It is worth underlying the difference between `uniform' and `weak' convergence of distances alluded to in Remark \ref{re:l2li} with an explicit example (see also the end of the introduction).
\begin{example}{\rm
Let $X_n:=[0,1-1/n]\cup[2,2+1/n]\subset\R$ and $X_\infty:=[0,1]$, all of them endowed with the (restriction of) Euclidean distance and Lebesgue measure. The fact that ${\rm diam}X_n\geq 2$ for every $n\in\N$ and ${\rm diam }X_\infty=1$ shows that there is no (pointed) mGH convergence of $X_n$ to $X_\infty$. On the other hand,  $\pGw$-convergence is obvious.
}\fr\end{example}
Hence in general  $\pGw$-convergence does not imply pointed mGH convergence. A natural assumption which allows to have such reverse implication is to assume the spaces to be uniformly doubling. Recall that $(X,\sfd,\mm)$ is called ${\sf c}$-doubling provided it holds
\begin{equation}
\label{eq:doubling}
\mm(B_{2R}(x))\leq {\sf c}\,\mm(B_R(x)),\qquad\forall x\in X,\ R>0.
\end{equation}
Notice that the doubling condition imposed as in \eqref{eq:doubling} forces $\supp(\mm)=X$ and in particular it does not pass to the quotient to a condition on equivalence classes of p.m.m.\ spaces (to get such property one could impose \eqref{eq:doubling} only for $x\in\supp(\mm)$).

It is easy to see that  a family of bounded and ${\sf c}$-doubling spaces is uniformly totally bounded (i.e. for any $\eps$ there exists $n_\eps$ such that each of the spaces can be covered by at most $n_\eps$ balls of radius $\eps$), whence the following compactness criterion holds (see e.g.\ Theorem 8.1.10 of \cite{Burago-Burago-Ivanov01} for a proof).
\begin{lemma}\label{le:compmgh}
Let $(X_n,\sfd_n,\mm_n,\bar x_n)$, $n\in\N$, be a sequence of 
\pmm~spaces as in \eqref{eq:proprbase1}, \eqref{eq:pointed}. Assume
that for some ${\sf c}>0$, all of them are ${\sf c}$-doubling 
and there exists $r>0$ such that 
$\sup_n \mm_n(B_r(\bar x_n))<\infty$.

Then the sequence is precompact in the pmGH topology and any limit space $(X_\infty,\sfd_\infty,\mm_\infty,\bar x_\infty)$ is ${\sf c}$-doubling as well.
\end{lemma}
Under the doubling assumption, we can show that $\pGw$-convergence implies pmGH convergence (following the lines used by Sturm in \cite{Sturm06I}).
\begin{proposition}[From  $\pGw$ to pmGH]\label{prop:Dpmgh}
Let $(X_n,\sfd_n,\mm_n,\bar x_n)$, $n\in\bar\N$,
be \pmm~spaces as in \pmmaxioms. Assume that 
\begin{itemize}
\item[i)]  $[X_n,\sfd_n,\mm_n,\bar x_n]\to [X_\infty,\sfd_\infty,\mm_\infty,\bar x_\infty]$ w.r.t.~$\pGw$-convergence,
\item[ii)] $\supp(\mm_\infty)=X_\infty$,
\item[iii)] for some ${\sf c}>0$ the spaces $(X_n,\sfd_n,\mm_n)$ are all ${\sf c}$-doubling.
\end{itemize}
Then  $(X_n,\sfd_n,\mm_n,\bar x_n)\to (X_\infty,\sfd_\infty,\mm_\infty,\bar x_\infty)$ in the pmGH sense.
\end{proposition}
\begin{proof}
Assumption $(iii)$ and Lemma \ref{le:compmgh} give that some
subsequence $k\mapsto (X_{n_k},\sfd_{n_k},\mm_{n_k},\bar x_{n_k})$
converges to some $(X'_\infty,\sfd'_\infty,\mm'_\infty,\bar
x'_\infty)$ in the pointed measured Gromov-Hausdorff sense. 

Proposition \ref{prop:pmghD} shows that 
$k\mapsto [X_{n_k},\sfd_{n_k},\mm_{n_k},\bar x_{n_k}]$ converges to
$[X'_\infty,\sfd'_\infty,\mm'_\infty,\bar x'_\infty]$ w.r.t.~
$\pGw$-convergence as well. Hence $[X'_\infty,\sfd'_\infty,\mm'_\infty,\bar x'_\infty]=[X_\infty,\sfd_\infty,\mm_\infty,\bar x_\infty]$ and recalling the approach to pmG-convergence given by the extrinsic approach (Definition \ref{def:Dconv}), Proposition \ref{prop:equivpmgh} and using assumption (ii) we easily deduce that  $k\mapsto (X_{n_k},\sfd_{n_k},\mm_{n_k},\bar x_{n_k})$ pmGH-converges to  $(X_\infty,\sfd_\infty,\mm_\infty,\bar x_\infty)$.

Being this result independent on the converging subsequence chosen at the beginning, the thesis is proved.
\end{proof}

\section{Stability of lower Ricci curvature bounds under  $\pGw$-convergence}

\subsection{Preliminaries}

\subsubsection{Relative entropy}


From now on, we shall only work with \pmm\ spaces as in \eqref{pointed} such that the bound
\begin{equation}
\label{eq:control1}
 \mm(B_r(\bar x))\le \sfc_1 \rme^{\sfc_2 r^2},\qquad\forall r>0,
\end{equation} 
holds for some $\sfc_1,\sfc_2$. As we shall see in Remark  \ref{re:exp} below, such restriction is fully justified when working on $\CD(K,\infty)$ spaces.

Given a \pmm~space $(X,\sfd,\mm, \bar x)$ 
as in \pmmaxioms~satisfying  \eqref{eq:control1},
the relative entropy functional $\entv:\probt X\to\R\cup\{+\infty\}$ 
  is defined as
\begin{equation} 
\label{eq:relent}
\entv(\mu):=
%
\left\{
\begin{array}{ll}
\displaystyle{\int\rho\log(\rho)\,\d\mm}&\textrm{ if
}\mu=\rho\mm\ll\mm,
\\
+\infty&\textrm{ otherwise}.
\end{array}
\right.
\end{equation} 
To check that this is a good definition, recall that if the reference measure $\mm$ is a probability measure, then it is well known that $\entv$ is lower semicontinuous w.r.t.~\narrow\ convergence and non negative.  For spaces satisfying \eqref{eq:control1}, we pick $\Co>\sfc_2$ and define
\begin{equation}
  \label{eq:expcontr}
  \z:=\int \rme^{-\Co\sfd^2(x,\bar x)}\,\d\mm(x)\in(0,\infty),\quad
  \text{so that }\quad
  \tilde\mm:=\frac 1\z \,\rme^{-\Co\sfd^2(\cdot,\bar x)}\mm\in \probt X.
\end{equation}
Then  we observe that for $\mu=\rho\mm$ it holds $\mu=\z\rho \rme^{\Co\sfd^2(\cdot,\bar x)}\tilde\mm$ and that since $\tilde\mm\in\prob X$,  the negative part of $\rho\log(\rho)$ is in $L^1(\mm)$, so that we get
\begin{equation}
\label{eq:chiave}
\entv(\mu)=\Entt(\mu)-\Co\int\sfd^2(\cdot,\bar x)\,\d\mu-\log \z.
\end{equation}
In particular, the definition in \eqref{eq:relent} is well-posed and $\entv$ is lower semicontinuous w.r.t.~$W_2$.

We will denote by $D(\entv)\subset\probt X$ the set of $\mu$'s such that $\entv(\mu)<\infty$.

A crucial property of the relative entropy w.r.t.~probability measures is that it is jointly lower semicontinuous w.r.t.~\narrow\ convergence on the two variables, i.e.
\begin{equation}
\label{eq:joint}
\left.
\begin{aligned}
(\tilde\mm_n),(\mu_n)\subset\prob X&,\\
\tilde\mm_n\to\tilde\mm_\infty\in\prob X&\ \textrm{ \narrowly\  } \\
\mu_n\to\mu_\infty\in\prob X&\ \textrm{ \narrowly\  }\\
\end{aligned}
\right\}\qquad\Rightarrow\qquad\Entti(\mu_\infty)\leq\limi_{n\to\infty}\Enttn(\mu_n).
\end{equation}
See for instance \cite[Theorem 29.20]{Villani09}, \cite[Lemma 6.2]{Ambrosio-Savare-Zambotti09} 
for a proof.
\subsubsection{Compactness criterions}
Let $(Z,\sfd_Z)$ be a complete and separable metric space.

A useful fact that we will often use in the following is that sublevels  of the relative entropy w.r.t.~probability measures are automatically tight:
\begin{proposition}\label{prop:enttight}
Let $(Z,\sfd_Z)$ be a complete separable metric space and let $(\tilde\mm_n),(\mu_n)\subset \prob Z$ be two sequences of measures. Assume that $(\tilde\mm_n)$ is tight and that $\sup_n\Enttn(\mu_n)<\infty$.\\
Then $(\mu_n)$ is tight as well.
\end{proposition}
\begin{proof}
Notice that $r\mapsto r\log(r)$ is convex and bounded from below by $-\frac1e$. Hence by Jensen's inequality, for any Borel set $E\subset Z$ and for $\mu_n=\rho_n\tilde\mm_n$ it holds
\[
\mu_n(E)\log\left(\frac{\mu_n(E)}{\tilde\mm_n(E)}\right)\leq \int_E \rho_n\log(\rho_n)\,\d\tilde\mm_n\leq \Enttn(\mu_n)+\frac{\tilde\mm_n(Z\setminus E)}{e}\leq \sup_n\Enttn(\mu_n)+\frac1e.
\]
Hence if $\sup_n\tilde\mm_n(E_k)\to 0$, then also $\sup_n\mu_n(E_k)\to 0$. Conclude using the tightness of $(\tilde\mm_n)$.
\end{proof}


\subsubsection{Geodesics}
Let $(Z,\sfd_Z)$ be a complete and separable metric space and $z^0,z^1\in Z$. A curve $[0,1]\ni t\mapsto z_t\in Z$ is said geodesic connecting $z_0$ to $z_1$ provided $z_0=z^0$, $z_1=z^1$ and
\[
\sfd_Z(z_t,z_s)=|s-t|\sfd_Z(z_0,z_1),\qquad\forall t,s\in[0,1].
\]
The set of all geodesics is denoted by $\geo (Z)$ and 
 it is a closed subset of the complete and separable 
metric space $\rmC^0([0,1];Z)$ endowed with the
$\sup$ distance
(possibly  reduced to trivial constant curves, 
if $Z$ has no geodesics connecting two distinct points).
We will denote by $\e_t:\rmC^0([0,1];Z)\to Z$, $t\in [0,1]$,
the evaluation map $\e_t(z):=z_t$.

A geodesic $(\nu_t)\subset \probt Z$ w.r.t.~the $W_2$ distance can
always be lifted to a `geodesic plan' on $\prob{\geo(Z)}$ in the
following sense (see \cite{Lisini07}  for a proof):
\begin{proposition}\label{prop:geod}
Let  $(\nu_t)\subset \probt Z$ be a geodesic in $(\probt Z,W_2)$. Then there  exists a plan $\ppi\in\prob{\geo(Z)}$ such that
\[
\begin{split}
(\e_t)_\sharp\ppi&=\nu_t,\qquad\forall t\in[0,1],\\
W_2^2(\nu_0,\nu_1)&=\int\sfd_Z^2(\gamma_0,\gamma_1)\,\d\ppi(\gamma).
\end{split}
\]
\end{proposition}

\subsubsection{$\CD(K,\infty)$ spaces}

\begin{definition}[$\CD(K,\infty)$ spaces]\label{def:cd}
Let $K\in \R$. We say that a (pointed) 
metric measure space $(X,\sfd,\mm, \bar x)$ as in 
\pmmaxioms~with \eqref{eq:expcontr} has Ricci curvature bounded from below by $K\in\R$ provided for any $\mu^0,\mu^1\in D(\entv)$ there exists a $W_2$-geodesic $(\mu_t)$ such that $\mu_0=\mu^0$, $\mu_1=\mu^1$ and
\[
\entv(\mu_t)\leq (1-t)\entv(\mu_0)+t\,\entv(\mu_1)-\frac K2t(1-t)W_2^2(\mu_0,\mu_1).
\]
\end{definition}
Notice that in this definition, only the portion of the space in
$\supp(\mm)$ plays a role, because if $\mu(X\setminus\supp(\mm))>0$
then certainly $\entv(\mu)=\infty$. It follows that the property of being a $\CD(K,\infty)$ space passes to the quotient and is well defined for equivalence classes $\cX^*,\cX$ of m.m.\ spaces and p.m.m.\ spaces respectively.

\begin{remark}\label{re:exp}{\rm
In the present approach, we defined the relative entropy only on spaces satisfying the exponential growth condition \eqref{eq:expcontr}, which grants, as already noticed, that $\entv$ is well defined and $W_2$-lower semicontinuous on $(\probt X,W_2)$.

It would be possible to define the relative entropy functional 
also on general \pmm~spaces not satisfying \eqref{eq:expcontr}, 
e.g.~by restricting its domain of definition: if $\mm$ is
finite on bounded sets, then certainly $\entv$ is well defined on
probability measures with bounded support (see also, e.g.,
\cite{Sturm06I} for a different approach). It is however important to
underline that whatever - meaningful - definition of $\entv$ one
chooses, on a $\CD(K,\infty)$ space $(X,\sfd,\mm)$,
 for every $\bar x\in \supp(\mm)$ the 
volume function $v(r):=\mm(B_r(\bar x))$ satisfies
the 
exponential growth rate
\begin{equation}
  v(R)\le \sfc(\eps,v(2\eps),v(2\eps)/v(\eps))\exp\Big((1+
  K_-)
  R^2\Big)
  \quad\text{for every }R\ge 2\eps>0,
  \label{eq:32}
\end{equation}
where $\sfc:[0,\infty)\times [0,\infty)\times (0,\infty)\to(0,\infty)$
is a constant continuously depending on its parameters,
see the argument in Theorem 4.24 in \cite{Sturm06I}. In particular, 
\eqref{eq:expcontr} always holds on $\CD(K,\infty)$
spaces.}\fr\end{remark} 
%
%

\subsection{Setting up the problem}
We now discuss the stability of lower Ricci curvature bounds under $\pGw$-convergence. To this aim, notice that Proposition \ref{prop:Dpmgh} and Remark \ref{re:pmghgeod} imply that all the curvature dimension conditions $\CD(K,N)$, $\CD^*(K,N)$, $\mathrm{MCP}(K,N)$ with finite $N$ are stable w.r.t.~$\pGw$-convergence (see \cite{Sturm06II}, \cite{Lott-Villani09}, \cite{Bacher-Sturm10} for the various definitions). Indeed if $(X,\sfd,\mm)$ fulfills any of these, it is  doubling and $(\supp(\mm),\sfd)$ is geodesic. Hence we can apply first Proposition \ref{prop:Dpmgh} to get that the considered sequence is converging in the pointed mGH sense, then use Remark \ref{re:pmghgeod} to reduce to the definition of pointed mGH convergence used by Villani and finally use his stability statements Theorem 29.25 in \cite{Villani09} (he deals only with $\CD(K,N)$, but the proofs directly generalize to $\CD^*(K,N)$ and $\mathrm{MCP}(K,N)$ - also, it should be pointed out that there are slight variants in the definition of $\CD(K,N)$, but all of these are stable w.r.t.~pointed mGH convergence and hence w.r.t.~$\pGw$-convergence).

Therefore the only case left out of the analysis is the one of $\CD(K,\infty)$ spaces, as this condition does not enforce doubling nor any sort of local compactness. In the rest of the work we thus concentrate on convergence of this sort of spaces.

In order to fix the ideas and the notation that will be used later on, we start with the following simple proposition: 
\begin{proposition}
  \label{cor:DconvCDK}
  Let 
  $\pmmXa n$ be a sequence of $\CD(K,\infty)$ spaces
  converging to $\pmmXa\infty$ under $\pGw$-convergence. 
 
  Then there exists a constant $\sfc_1>0$ such that 
      \begin{equation}
       \label{eq:expcontr2}
       \mm_n(B_R(\bar x_n))\le \sfc_1\exp\Big((1+K_-)\,R^2\Big)\quad
       \forevery n\in \N,\ R\ge0,
     \end{equation}
     and choosing
\[
    \psi(r):=\rme^{-\Co r^2}\text{ with } \Co>1+K_-,
\]
   we have 
\[
\text{$\cX_n:=[X_n,\sfd_n,\mm_n,\bar x_n]\in\X^\psi$ for every $n\in\bar\N$\qquad and}\qquad
     \lim_{n\to\infty}
     \Dmeas\psi\big(\cX_n,\cX_\infty\big)=0.
\]
   \end{proposition}
   \begin{proof}
     Choosing e.g.\ $\eps=1$ in \eqref{eq:32} and observing that 
     $$\liminf_{n\to\infty}\mm_n(B_1(\bar x_n))\ge \mm_\infty(B_1(\bar
     x_\infty))>0,\quad
     \limsup_{n\to\infty}\mm_n(B_2(\bar x_n))\le 
     \mm_\infty(B_3(\bar x_\infty))<\infty,$$ 
     \eqref{eq:32} yields the uniform bound \eqref{eq:expcontr2}. The other assertions follow from 
     Theorem \ref{thm:main_convergence} and Remark \ref{rem:diffpesi}, noticing that the property \eqref{eq:mah} holds for the chosen $\psi$.
\end{proof}
We collect in the next remark the main assumptions and notation that we are going to use in the rest of the paper.
\begin{remark}[Main assumptions and notation]\label{re:assnot}$\ $ {\rm
\begin{itemize}
\item[-]
$\cX_n:=[X_n,\sfd_n,\mm_n,\bar x_n]$, $n\in \N$, is a given family of p.m.m.\ spaces pmG-converging to a limit space $\cX_\infty:=[X_\infty,\sfd_\infty,\mm_\infty,\bar x_\infty]$.
\item[-] There is a complete and separable space $(X,\sfd)$ containing isometrically all the $(X_n,\sfd_n)$, $n\in\bar \N$ and which realizes the extrinsic convergence as in Definition \ref{def:Dconv}. In particular
\[
\lim_{n\to\infty}\sfd(\bar x_n,\bar x_\infty)=0.
\]
\item[-] For some $\Co>0$ we have $\cX_n\in\X^\psi$ for every $n\in\bar \N$, where $\psi(r):=e^{-\Co r^2}$ and $\D^\psi(\cX_n,\cX_\infty)\to 0$ as $n\to\infty$. In particular, for some $\sfc_1,\sfc_2>0$ all the spaces $\cX_n$, $n\in\bar\N$, satisfy the bound \eqref{eq:control1}.
\item[-] According to \eqref{eq:expcontr}, for $n\in\bar\N$ we define
\[
  \z_n:=\int \rme^{-\Co\sfd_n^2(x,\bar x_n)}\,\d\mm_n(x),\quad
  \text{ and }\quad
  \tilde\mm_n:=\frac 1{\z_n} \,\rme^{-\Co\sfd_n^2(\cdot,\bar x_n)}\mm_n\in\probt{X_n}\subset\probt X,
  \]
  so that the convergence in $(\X^\psi,\D^\psi)$ gives
\[
\lim_{n\to\infty}\z_n=\z_\infty,\qquad\qquad \lim_{n\to\infty}W_2(\tilde\mm_n,\tilde\mm_\infty)=0.
\]
\item[-] We shall say that $n\mapsto x_n\in \supp(\mm_n)$ converges to $x_\infty\in\supp(\mm_\infty)$ provided $\lim_{n\to\infty}\sfd(x_n,x_\infty)=0$.
\item[-] We shall say that  $n\mapsto\mu_n\in\prob{\supp(\mm_n)}$  weakly converges  to $\mu\in\prob{\supp(\mm_\infty)}$ provided $(\mu_n)$ weakly converges to $\mu_\infty$ as sequence in $\prob X$. Similarly for measures in $\probt{\supp(\mm_n)}$ and  $W_2$-convergence.
\end{itemize}
}\fr\end{remark}
Notice that we are directly assuming that $(X,\sfd)$ contains isometrically the $X_n$'s rather then asking for the existence of isometric embeddings. Up to identify $(X_n,\sfd_n)$ with its isometric image this is always possible, but this particular choice allows to avoid to introduce the isometric embeddings $\iota_n$, thus simplifying the notation.

In line with the discussion made at the end of Section \ref{se:topcomm}, we also point out that although the choice of the space $(X,\sfd)$ is non-intrinsic and might affect the notion of convergence of points and measures as given above (in line with Definition \ref{def:convpoint}), in fact all the statements that we are going to prove are intrinsic in nature, being independent of the particular choice made. We won't insist on this point any further.

\subsection{Stability of $\CD(K,\infty)$}
We start our discussion on the stability of Ricci curvature
bounds. The first result is about so-called $\Gamma$-convergence of
the relative entropies w.r.t.~$W_2$-convergence.
\begin{proposition}[$\Gamma$-convergence of the Entropies]\label{prop:GammaConvergence} With the same assumptions and notation as in Remark \ref{re:assnot}, the sequence of relative entropies $(\Entn)$ $\Gamma$-converges to
$\Enti$ w.r.t.~the $W_2$-convergence, i.e.:
\begin{itemize}
\item\underline{$\Gamma-\liminf$} for any $\mu_\infty\in\probt{X_\infty}$ and any  sequence $(\mu_n)$ such that $W_2(\mu_n,\mu_\infty)\to 0$ we have
\[
\Enti(\mu_\infty)\leq \limi_{n\to\infty}\Entn(\mu_n),
\]
\item\underline{$\Gamma-\limsup$} for any $\mu_\infty\in\probt {X_\infty}$ there exists a  sequence $(\mu_n)$ such that $W_2(\mu_n,\mu_\infty)\to 0$ and
\begin{equation}
\label{eq:glims}
\Enti(\mu_\infty)\geq \lims_{n\to\infty}\Entn(\mu_n).
\end{equation}
\end{itemize}
\end{proposition}

\begin{proof}$\ $\\
\noindent{$\mathbf{ \Gamma}${\bf -lim\,inf.}}  
 Let $(\mu_n)\subset\probt X$ be a sequence $W_2$-converging to some $\mu_\infty\in\probt X$. By Proposition \ref{prop:narw2} we have  $\int \sfd^2(\cdot,\bar x_n)\,\d\mu_n\to \int \sfd^2(\cdot,\bar x_\infty)\,\d\mu_\infty$. Thus from \eqref{eq:chiave} and the fact that $\z_n\to \z_\infty\in(0,\infty)$, to conclude it is enough to check that
\[
\limi_{n\to\infty}\Enttn(\mu_n)\geq\Entti(\mu_\infty),
\]
which follows by \eqref{eq:joint}.

\noindent{$\mathbf{ \Gamma}${\bf -lim\,sup.}}  Using  \eqref{eq:chiave} and   $\z_n\to \z_\infty$, to conclude it is enough to prove that for any $\mu_\infty\in\probt X$ there exists a sequence $(\mu_n)\subset\probt X$ $W_2$-converging to $\mu_\infty$ and such that
\[
\lims_{n\to\infty}\Enttn(\mu_n)\leq \Entti(\mu_\infty).
\]
 The proof of this fact is not new (see for instance \cite[Lemma 6.2]{Ambrosio-Savare-Zambotti09}), but for completeness let us recall the arguments.

 If $\Entti(\mu_\infty)=\infty$ there is nothing to prove. Thus we can assume that $\mu_\infty\ll\tilde\mm_\infty$. Let $\mu_\infty=\rho\tilde\mm_{\infty}$ and assume for the moment that $\rho$ is bounded. 

For any $n\in \N$, let $\ggamma_n \in \Opt(\tmm_\infty, \tmmn)$ be an optimal plan. Define $\ggamma_n'\in\probt{X^2}$ as $\d\ggamma_n'(x,y):=\rho(x)\d\ggamma_n(x,y)$ and $\mu_n:=\pi^2_\sharp\ggamma_n'\in\probt X$. By construction, $\ggamma_n'\ll\ggamma_n$, hence $\mu_n\ll\pi^2_\sharp\ggamma_n=\tilde\mm_n$. Let $\mu_n=\eta_n\tilde\mm_n$. It is readily checked from the definition that it holds $\eta_n(y)=\int \rho(x)\,\d(\ggamma_n)_{y}(x)$, where $\{(\ggamma_n)_y\}$ is the disintegration of $\ggamma_n$ w.r.t. the projection on the second marginal. By Jensen's inequality applied to the convex function $u(z)=z\log(z)$ we have
\[
\begin{split}
\Enttn(\mu_n)&=\int u(\eta_n)\,\d\tilde\mm_n=
\int u\left(\int \rho(x)\,\d(\ggamma_n)_{y}(x)\right)\,\d\tilde\mm_n(y)\\
&\leq \iint u(\rho(x))\,\d(\ggamma_n)_{y}(x)\,\d\tilde\mm_n(y)=\iint u(\rho(x))\,\d\ggamma_n(x,y)\\
&=\int u(\rho)\,\d\pi^1_\sharp\ggamma_n=\int u(\rho)\,\d\tilde\mm_\infty=\Entti(\mu_\infty).
\end{split}
\]
Since by construction we have $\ggamma_n'\in\Adm(\mu_\infty,\mu_n)$, it holds
\[
W_2^2(\mu_\infty,\mu_n)\leq \int\sfd^2(x,y)\,\d\ggamma_n'(x,y)= \int\rho(x)\sfd^2(x,y)\,\d\ggamma_n(x,y)\leq\big(\sup\rho\big)  W_2^2(\tilde\mm_\infty,\tilde\mm_n),
\]
and therefore $W_2(\mu_\infty,\mu_n)\to 0$. Thus in this case the thesis is proved.

If $\rho$ is not bounded, for $k\in\N$ define $\rho^k:=c_k\min\{\rho,k\}$, $c_k$ being such that  $\mu^k:=\rho^k\tilde\mm_\infty\in\probt X$. Clearly, it holds
\[
\lims_{k\to\infty}\Entti(\mu^k)\leq \Entti(\mu_\infty),\qquad \lim_{k\to\infty}W_2(\mu^k,\mu_\infty)=0.
\]
Then apply the previous argument to $\mu^k$ and conclude with a diagonalization argument.
\end{proof}
We shall also make use of the following general 2-uniform integrability criterion:
\begin{proposition}\label{prop:2equi}
Let $(Y,\sfd)$ be a complete and separable metric space and $(\mu_n),(\nu_n) \subset \probt Y$ two 2-uniformly integrable sequences. Assume that for every $n\in\N$ there exists a $W_2$-geodesic $t\mapsto\mu_{n,t}$ connecting $\mu_n$ to $\nu_n$. 

Then the family $\{\mu_{n,t}\}_{n\in\N,t\in[0,1]}$ is  $2$-uniformly integrable.
\end{proposition}
\begin{proof}
For every $n\in\N$, let $\ppi_n\in \prob{\geo(Y)}$ be a  plan representing the geodesic $t\mapsto\mu_{n,t}$ as in Proposition \ref{prop:geod}. Clearly it holds
\begin{equation}\label{eq:d2pi}
\int_{B_R^c(\bar x)} \sfd^2(\bar{x},\cdot )\, \d \mu_{n,t}= \int_{\{\gamma:\sfd^2(\bar{x},\gamma_t)>R\}} \sfd^2(\bar{x},\gamma_t) \; \d \ppi_n(\gamma),\qquad\forall R>0,\ t\in[0,1],\ n\in\N.
\end{equation}
The trivial inequality
\[
\sfd(\gamma_t, \bar{x})\leq 2\big( \sfd(\gamma_0,\bar{x})+ \sfd(\gamma_1,\bar{x}) \big),
\]
ensures that $\{\gamma:\sfd(\gamma_t,\bar x)>R \}\subset A^0_R\cup A^1_R$, where $A^0_R,A^1_R\subset \geo(Y)$ are given by 
\[
\begin{split}
A^0_R&:=\big\{\gamma\ :\ \sfd(\gamma_0,\bar x)\geq R/4,\ \sfd(\gamma_0,\bar x)\geq \sfd(\gamma_1,\bar x)\big\},\\
A^1_R&:=\big\{\gamma\ :\ \sfd(\gamma_1,\bar x)\geq R/4,\ \sfd(\gamma_1,\bar x)\geq \sfd(\gamma_0,\bar x)\big\}.
\end{split}
\] 
From the bound
\[
\sfd(\gamma_t,\bar x)\leq 2\big( \sfd(\gamma_0,\bar{x})+ \sfd(\gamma_1,\bar{x}) \big)\leq 4\sfd(\gamma_0,\bar x),\qquad\forall \gamma\in A^0_R,
\]
we get
\[
\int_{A^0_R}\sfd^2(\gamma_t,\bar x)\,\d\ppi_n(\gamma)\leq 16\int_{A^0_R}\sfd^2(\gamma_0,\bar x)\,\d\ppi_n(\gamma)\leq 16\int_{B^c_{R/4}(\bar x)}\sfd^2(\cdot, \bar x)\,\d \mu_n. 
\]
Arguing symmetrically for $A^1_R$ and taking \eqref{eq:d2pi} into account we obtain
\[
\int_{B_R^c(\bar x)} \sfd^2(\bar{x},\cdot ) \d \mu_{n,t}\leq 16\int_{B^c_{R/4}(\bar x)}\sfd^2(\cdot, \bar x)\,\d \mu_n+16\int_{B^c_{R/4}(\bar x)}\sfd^2(\cdot, \bar x)\,\d \nu_n,
\]
for any $R>0$, $t\in[0,1]$, $n\in\N$, and the conclusion follows from the 2-uniform integrability of $(\mu_n)$, $(\nu_n)$.
\end{proof}

With these tools at disposal, we can now prove the general stability of the $\CD(K,\infty)$ condition along the same lines used by Lott-Villani \cite{Lott-Villani09} and Sturm \cite{Sturm06I}:
\begin{theorem}[Stability of $\CD(K,\infty)$]\label{thm:stabilityCD} 
Let $\cX_n$,  $n\in \N$, be a sequence of 
$\CD(K,\infty)$ \pmm~spaces converging to  $\cX_\infty$ in the  $\pGw$-convergence.
Then $\cX_\infty$  is a $\CD(K,\infty)$ space as well.
\end{theorem}
\begin{proof}
By Proposition  \ref{cor:DconvCDK}, without loss of generality we may use the assumptions and notation  of Remark \ref{re:assnot}.

Let $\mu_\infty,\nu_\infty\in D(\Enti)$. Use  the $\Gamma-\lims$
inequality in Proposition \ref{prop:GammaConvergence} to find
sequences $(\mu_n),(\nu_n)\subset\probt X $ which $W_2$-converge to
$\mu_\infty,\nu_\infty$ respectively and 
satisfy
\begin{equation}
\label{eq:glimsstab}
\lims_{n\to\infty}\Entn(\mu_n)\leq\Enti(\mu_\infty),\qquad\qquad\lims_{n\to\infty}\Entn(\nu_n)\leq\Enti(\nu_\infty).
\end{equation}
In particular, for $n$ large enough we have $\mu_n,\nu_n\in D(\Entn)$, thus by assumption there are $W_2$-geodesics $t\mapsto\mu_{n,t}$ connecting $\mu_n$ to $\nu_n$ such that
\begin{equation}
\label{eq:geodent}
\Entn(\mu_{n,t})\leq(1-t)\Entn(\mu_n)+t\,\Entn(\nu_n)-\frac K2t(1-t)W_2^2(\mu_n,\nu_n).
\end{equation}
Therefore, for $n$ large enough the values of $\Entn(\mu_{n,t})$ are
uniformly bounded in $n,t$. Since the second moments of $\mu_{n,t}$
are also - clearly - uniformly bounded in $n,t$, from
\eqref{eq:chiave} and $\z_n\to \z_\infty\in(0,\infty)$ we deduce that
the values of $\Enttn(\mu_{n,t})$ are uniformly bounded in $n,t$, as
well. Hence by Proposition \ref{prop:enttight}, for every $t\in[0,1]$
the sequence $n\mapsto\mu_{n,t}$ is tight. From Proposition
\ref{prop:2equi} we also know that it is 2-uniformly integrable, thus
it is  relatively  compact in $(\probt X ,W_2)$ (Theorem \ref{thm:prok} and Proposition \ref{prop:narw2}).

Since geodesics are equi-Lipschitz, applying the
metric version of Arzel\`a-Ascoli Theorem 
we find a subsequence $n_k\up+\infty$ 
and a limit geodesic $(\mu_{\infty,t})\subset \probt X $ 
such that 
$W_2(\mu_{n_k,t},\mu_{\infty,t})\to 0$ as $k\to\infty$ for every $t\in
[0,1]$. 
Hence  the  $\Gamma-\limi$ inequality in Proposition \ref{prop:GammaConvergence} yields
\begin{equation}
\label{eq:glimistab}
\limi_{k\to\infty}\Entn(\mu_{n,t})\geq
\Enti(\mu_{\infty,t})\quad\forevery t\in [0,1].
\end{equation}
The inequalities \eqref{eq:glimsstab}, \eqref{eq:geodent} and \eqref{eq:glimistab} give
\[
\Enti(\mu_{\infty,t})\leq(1-t)\Enti(\mu_\infty)+t\,\Enti(\nu_\infty)-\frac K2t(1-t)W_2^2(\mu_\infty,\nu_\infty),
\]
for any $t\in[0,1]$. 
\end{proof}

\section{Compactness and stability of the heat flows}
Aim of this section is to show that the gradient flow of the relative entropy is stable  w.r.t.~$\pGw$-convergence under a uniform $\CD(K,\infty)$ condition. We call such gradient flow `heat flow' by analogy with the common Euclidean setting. 
\subsection{Preliminaries}
\subsubsection{Absolutely  continuous curves}
Let $(Z,\sfd_Z)$ be a complete separable metric space and $I\subset\R$ a non trivial interval. A curve $I\ni t\mapsto z_t\in Z$ is absolutely continuous (resp. locally absolutely continuous) provided there exists a function $f\in L^1(I)$ (resp. in $L^1_{loc}(I)$) such that
\begin{equation}
\label{eq:ac}
\sfd_Z(z_t,z_s)\leq \int_t^sf(r)\,\d r,\qquad\forall t,s\in I,\ t<s.
\end{equation}

For absolutely continuous (resp. locally a.c.) curves, the limit
$\lim_{h\to 0}\frac{\sfd_Z(z_{t+h},z_t)}{|h|}$ exists for a.e. $t\in
I$, defines an $L^1(I)$ (resp. $L^1_{loc}(I)$) function denoted by
$|\dot z_t|$ called metric speed, and it is the minimal function -
in the a.e. sense - that can be put in the right hand side of
\eqref{eq:ac}  (see Theorem 1.1.2 in \cite{Ambrosio-Gigli-Savare08}
for a proof).
 We denote by $\AC p IZ$ the space of all absolutely continuous
  curves
with metric derivative in $L^p(I)$.

The map
\[
\rmC(I,Z)\ni \gamma\mapsto
 \Ecurve \gamma:=\left\{
\begin{array}{ll}
\displaystyle{\int_I|\dot \gamma_t|^2\,\d t},&\qquad\textrm{if
  $\gamma\in \AC2 IZ$},\\
+\infty,&\qquad\textrm{otherwise},
\end{array}
\right.
\]
is lower semicontinuous. 
In the following we will often write the expression $\int_I|\dot \gamma_t|^2\,\d t$ even for curves $\gamma$ not absolutely continuous, in this case the value will be understood as $+\infty$ implicitly.

Given an absolutely continuous curve $\mu\in \AC{}I{(\probt Z,W_2)}$,
we will denote by $|\dot\mu_t|$ its metric speed 
in the space $(\probt Z,W_2)$.
If $\ppi\in\prob{C(I,Z)}$ is a plan 
satisfying $(\e_t)_\sharp\ppi=\mu_t$ for any $t\in I$, it is easy to see that it holds
\begin{equation}
\label{eq:nonopt}
\int_I|\dot\mu_t|^2\,\d t\leq \int \Ecurve\gamma
\,\d\ppi(\gamma),
\end{equation}
In \cite{Lisini07} it has been showed that with an appropriate selection of $\ppi$ equality can hold in \eqref{eq:nonopt}:
\begin{proposition}\label{prop:lisini}
Let $(Z,\sfd_Z)$ be complete and separable and  
$\mu\in \AC 2I{(\probt Z,W_2)}.$ 
Then there exists $\ppi\in\prob{C(I,Z)}$ such that 
\begin{equation}
\label{eq:optimal}
\text{$(\e_t)_\sharp\ppi=\mu_t$ for any $t\in I$},\qquad
\int_I|\dot\mu_t|^2\,\d t= \int \Ecurve\gamma\,\d\ppi(\gamma).
\end{equation}
\end{proposition}

\subsubsection{Bits of Sobolev calculus}
In what follows (specifically, Proposition \ref{prop:glimisl}), we will need to deal with a little bit of Sobolev calculus. Here we recall the basic definitions and facts used later.

There are several, equivalent, definitions of Sobolev functions from a
metric measure space $(X,\sfd,\mm)$ to $\R$: 
we consider here only the case $p=2$ 
following the approach proposed in \cite{Ambrosio-Gigli-Savare-preprint11a};
we refer to 
\cite{Ambrosio-Gigli-Savare-preprint11c,Ambrosio-DiMarino} for the case 
of a general summability exponent $p\in [1,\infty)$
and to \cite{Gigli2012} for further developments and 
a deeper analysis of the duality
relations between weak differentials and gradients in metric 
measure spaces.
\begin{definition}[Test plans]
  Let $(X,\sfd,\mm,\bar x)$ be a \pmm~space as in
  \pmmaxioms. 
We say that
$\ppi\in\prob{C([0,1],X)}$ is a test plan provided there exists a
constant $c>0$ such that 
\[
\begin{split}
(\e_t)_\sharp\ppi\leq c\mm\quad \forevery  t\in[0,1],\qquad
\int \Ecurve\gamma
\d\ppi(\gamma)<\infty.
\end{split}
\]
\end{definition}
\begin{definition}[Sobolev class and weak upper gradients]
Let $(X,\sfd,\mm,\bar x)$ be a \pmm~space as in \pmmaxioms\
and let $f:X\to\R$ be a Borel function. We say that $f$ belongs to the Sobolev class $\s^2(X,\sfd,\mm)$   provided there exists  $G\in L^2(X,\mm)$ such that 
\[
\int|f(\gamma_1)-f(\gamma_0)|\,\d\ppi(\gamma)\leq \iint_0^1G(\gamma_t)|\dot\gamma_t|\,\d t\,\d\ppi(\ggamma),\qquad\forevery \textrm{ test plan }\ppi.
\]
Any such  $G$ is called weak upper gradient of $f$.
\end{definition}
It turns out that this notion is invariant with respect 
to modification of $f$ in $\mm$-negligible sets and
for $f\in\s^2(X,\sfd,\mm)$ there exists a minimal - in the $\mm$-a.e. sense - weak upper gradient $G$, which we  will denote by $\weakgrad f$. Notice that if $f:X\to\R$ is Lipschitz, then certainly the function identically equal to the Lipschitz constant $\Lip(f)$ of $f$ is a weak upper gradient, so we get the inequality
\begin{equation}
\label{eq:weaklip}
\weakgrad f\leq \Lip(f),\qquad\mm-a.e..
\end{equation}
Basic calculus rules are
\begin{equation}
\label{eq:calcrule}
\begin{split}
\weakgrad{(\alpha f+\beta g)}&\leq |\alpha|\weakgrad f+|\beta|\weakgrad g,\qquad\forall f,g\in\s^2(X,\sfd,\mm),\ \alpha,\beta\in\R,\\
\weakgrad{(fg)}&\leq |f|\weakgrad g+|g|\weakgrad f,\qquad\forall f\in\s^2(X,\sfd,\mm),\ g\textrm{ Lipschitz},
\end{split}
\end{equation}
these inequalities being valid $\mm$-a.e..

The Sobolev space $W^{1,2}(X,\sfd,\mm)$ is then defined as $L^2(X,\mm)\cap \s^2(X,\sfd,\mm)$ endowed with the norm
\[
\|f\|_{W^{1,2}}^2:=\|f\|^2_{L^2}+\|\weakgrad f\|^2_{L^2}.
\]
Notice that in general $W^{1,2}(X,\sfd,\mm)$ 
is not an Hilbert space (consider for instance the case of finite dimensional Banach spaces). 
In particular, in general there is no natural Dirichlet form on $L^2(X,\mm)$. 
The potentially non quadratic object which replaces the Dirichlet energy is called - following \cite{Ambrosio-Gigli-Savare-preprint11a} - Cheeger energy: it is the 
 lower semicontinuous and convex functional $\C:L^2(X,\mm)\to[0,+\infty]$ defined by
\begin{equation}
\label{eq:cheeger}
\C(f):=\left\{
\begin{array}{ll}
\displaystyle{\frac12\int\weakgrad f^2\,\d\mm},&\qquad\textrm{if }f\in\s^2(X,\sfd,\mm),\\
+\infty,&\qquad\textrm{otherwise}.
\end{array}
\right.
\end{equation}
 Notice that the notions of (minimal) weak upper gradient and 
of Cheeger energy are invariant under isomorphisms of 
\pmm~spaces, in the sense that if $\pmmXa i$, $i=1,2$ are isomorphic
according to \eqref{eq:37} then
\[
  \C_{X_1}(f\circ \iota)=\C_{X_2}(f)\quad
  \forevery f\in L^2(X_2,\mm_2).
\]
 As usual, in the following we will follow the notation of Remark \ref{re:assnot} and deal with spaces
$(X_n,\sfd_n,\mm_n,\bar x_n)$, $n\in \bar\N$,
which are all subspaces of a common space $(X ,\sfd)$. In particular,   the spaces $(X ,\sfd,\mm_n,\bar x_n)$,
$n\in \bar \N$, are isomorphic to $\pmmXa n$ and we will consider the Cheeger energies
 $\C_n: L^2(\Y ,\mm_n)\to[0,\infty]$, $n\in\bar\N$.

\subsubsection{$W_2$-Gradient flow of the entropy}

Let $(X,\sfd,\mm,\bar x)$ be a \pmm~space as in 
\pmmaxioms\ satisfying the growth condition
\eqref{eq:expcontr2}. The (descending) slope $|\Dm^-\entv|:\probt X \to[0,+\infty]$ of the relative entropy $\entv$ is defined as identically $+\infty$ outside $D(\entv)$, as $0$ at isolated measures in $\probt X$ and in all other cases as
\[
|\Dm^-\entv|(\mu):=\lims_{W_2(\nu,\mu)\to 0}\frac{\big(\entv(\mu)-\entv(\nu)\big)^+}{W_2(\mu,\nu)},
\]
where by $(\cdot)^+$ we intend the positive part. 

If $(X,\sfd,\mm,\bar x)$ is a $\CD(K,\infty)$ space for some $K\in\R$, the slope admits the representation
\begin{equation}
\label{eq:repslope}
|\Dm^-\entv|(\mu):=
\sup_{\nu\neq\mu}\left(\frac{\entv(\mu)-\entv(\nu)}{W_2(\mu,\nu)}+\frac K2W_2(\mu,\nu)\right)^+\quad\forevery\mu\in\probt X,
\end{equation}
see \cite[Chap.~2]{Ambrosio-Gigli-Savare08} which shows in particular that
\begin{equation}
\label{eq:slopesemi}
\textrm{If }(X,\sfd,\mm,\bar x)\textrm{ is a }\CD(K,\infty)\textrm{ space, then }|\Dm^-\entv|\textrm{ is $W_2$-lower semicontinuous}.
\end{equation}
A deep result obtained in \cite{Ambrosio-Gigli-Savare-preprint11a} is that on $\CD(K,\infty)$ spaces the slope admits the representation formula
\begin{equation}
\label{eq:slopefish}
|\Dm^-\entv|^2(\rho\mm)=8\C(\sqrt{\rho})=
\int_{\{\rho>0\}}\frac{\weakgrad \rho^2}\rho\,\d\mm
\end{equation}
which is the standard representation formula for the slope of the entropy in terms of the Fisher information, well known in a smooth setting.

Another  non trivial consequence (see Corollary 2.4.10 in \cite{Ambrosio-Gigli-Savare08} for a proof) of the $K$-geodesic convexity of $\entv$ is that  the slope is an upper gradient for $\entv$, i.e.~it holds
\[
|\entv(\mu_0)-\entv(\mu_1)|\leq \int_0^1|\dot\mu_t||\Dm^-\entv|(\mu_t)\,\d t,
\]
whenever $(\mu_t)\subset D(\entv)$, $t\in [0,1]$,
 is an absolutely continuous curve.
In particular, for any locally absolutely continuous curve $(\mu_t)\subset D(\entv)$ defined on some interval $I\subset \R$ it holds
\begin{equation}
\label{eq:pergf}
\entv(\mu_t)\leq \entv(\mu_s)+\frac12\int_t^s|\dot\mu_r|^2\,\d r+\frac12\int_t^s|\Dm^-\entv|^2(\mu_r)\,\d r,\qquad\forall t,s\in I,\ t<s.
\end{equation}
Gradient flows are defined as those curves for which equality holds:
\begin{definition}[Gradient flow of the relative entropy]\label{def:gf}
Let $(X,\sfd,\mm,\bar x)$ be a $\CD(K,\infty)$ space and $\bar \mu\in
D(\entv)$. A curve $\mu:[0,\infty)\to D(\entv)\subset \probt X$
is the $W_2$-gradient flow of $\entv$ starting from $\bar\mu$ provided 
 it is  locally absolutely continuous in $(\probt X,W_2)$, 
$\mu_0=\bar \mu$  and 
\begin{equation}
\label{eq:defgf}
\entv(\mu_t)= \entv(\mu_s)+\frac12\int_t^s|\dot\mu_r|^2\,\d
r+\frac12\int_t^s|\Dm^-\entv|^2(\mu_r)\,\d r\quad\forall\, 0<t<s.
\end{equation}
\end{definition}
Thanks to \eqref{eq:pergf}, 
\eqref{eq:defgf} is equivalent to
\begin{equation}
\label{eq:defgf2}
\entv(\bar \mu)\geq \entv(\mu_T)+\frac12\int_0^T|\dot\mu_t|^2\,\d
t+\frac12\int_0^T|\Dm^-\entv|^2(\mu_t)\,\d t\quad\forevery T>0.
\end{equation}
We remark that  this definition is invariant under isomorphisms
of \pmm~spaces, since it just involves the distance and
measures in the domain of the entropy: notice that, 
with the notation of \eqref{eq:37},
\[
  \Ent_{\mm_2}(\iota_\sharp\mu)=\Ent_{\mm_1}(\mu)\quad\forevery 
  \mu\in \probt{X_1},\ \mu\ll\mm_1.
\]
Therefore the property of being a gradient flow of the entropy is not
affected if our space $(X,\sfd,\mm,\bar x)$ is isometrically embedded
into some bigger space $X$ as in Remark \ref{re:assnot}.

The following theorem collects some of the main properties of the
gradient flow of the entropy and a useful a priori estimate.
\begin{theorem}\label{thm:gf}
Let $(X,\sfd,\mm,\bar x)$ be a $\CD(K,\infty)$ \pmm~space. 
Then for every $\bar\mu\in D(\entv)$ there exists a unique 
$W_2$-gradient flow $\mu_t:=\h_t\bar\mu$, $t\ge0$, of the entropy starting from $\bar\mu$. The curve $(\mu_t)$ has also the following properties:
\begin{align}
\label{eq:kinugsl}
|\dot\mu_t|=|\Dm^-\entv|(\mu_t)&\qquad \text{for a.e.\ }t>0,\\
\label{eq:slopereg}
t\mapsto \rme^{Kt}|\Dm^-\entv|(\mu_t)&\qquad\textrm{is lower semicontinuous and non increasing}.
\end{align}
Moreover, if $\Co>1+K_-$ so that \eqref{eq:expcontr} holds, 
and $T:=\frac 1{8\Co}$,
then 
\begin{equation}
\label{eq:speedcontr}
\frac12\int_0^T|\dot\mu_t|^2\,\d t\leq 2\entv(\bar \mu)+4\Co\int\sfd^2(\cdot,\bar x)\,\d\bar \mu+2\log(\z).
\end{equation}
\end{theorem}

\begin{proof}
  Here we limit to check the estimate
  \eqref{eq:speedcontr}, 
  referring to 
  \cite{Ambrosio-Gigli-Savare-preprint11a} for a proof of the main
  part of the statement (see also
  \cite{Ambrosio-Gigli-Savare-compact} for a survey in the case of
  compact spaces). 

  By definition of gradient flows and \eqref{eq:chiave} we know that
\[
\begin{split}
\entv(\mu_0)&\geq\entv(\mu_{T})+\frac12\int_0^T|\dot\mu_{t}|^2\,\d t= \Entt(\mu_{T})-\Co\int\sfd^2(\cdot,\overline x)\,\d\mu_{T}-\log(\z)+\frac12\int_0^T|\dot\mu_{t}|^2\,\d t.
\end{split}
\]
We also have the simple bound
\[
\begin{split}
\int\sfd^2(\cdot,\overline x)\,\d\mu_{T}&=W_2^2(\mu_T,\delta_{\bar x})\leq \left(W_2(\mu_0,\delta_{\bar x})+\int_0^T|\dot\mu_t|\,\d t\right)^2\leq 2W_2^2(\mu_0,\delta_{\bar x})+2T\int_0^T|\dot\mu_t|^2\,\d t\\
&=2\int\sfd^2(\cdot,\bar x)\,\d\bar \mu+\frac1{4\Co}\int_0^T|\dot\mu_t|^2\,\d t.
\end{split}
\]
\eqref{eq:speedcontr} comes from these two inequalities  taking into
account that $\Entt\geq 0$, since $\tilde\mm$ is a probability measure.
\end{proof}

\subsubsection{$L^2$-gradient flow of the Cheeger energy}
\label{subsub:L2g}
The main identification result of \cite{Ambrosio-Gigli-Savare-preprint11a} 
shows that in $\CD(K,\infty)$ spaces 
the Wasserstein gradient flow of the entropy coincides
with the $L^2$-gradient flow $(\sfH_t)_{t\ge0}$ of the Cheeger energy: 
the latter is the semigroup of contractions $\sfH_t:L^2(X,\mm)\to
L^2(X,\mm)$ whose trajectories $f_t=\sfH_t\bar f$, $t\ge0$,
belongs to $\Lip_{\rm loc}((0,\infty);L^2(X,\mm))$ 
and are the unique solution of the differential inclusion
(see e.g.~\cite{Brezis73} and \cite[\S 4]{Ambrosio-Gigli-Savare-preprint11a})
\[
  \frac\d{\d t}f_t+\partial\C(f_t)\ni0\quad\text{a.e.\ in $(0,\infty)$},\quad
  \lim_{t\down0}f_t=\bar f\quad\text{in }L^2(X,\mm).
\]
$\partial\C(f)$ denotes the (possibly multivalued) subdifferential 
of $\C$ in $L^2(X,\mm)$, it is a convex subset of 
$L^2(X,\mm)$; when $\partial \C(f)$ is not empty, 
its element of minimal norm defines the Laplacian $\Delta_{\sfd,\mm}$ of $f$.
\begin{theorem}[\cite{Ambrosio-Gigli-Savare-preprint11a}, Thm~9.3]
  Let $(X,\sfd,\mm,\bar x)$ be a $\CD(K,\infty)$ \pmm~space. 
  If $\bar\mu=\bar f\mm\in \probt X$ with $f\in L^2(X,\mm)$ then
\[
    \h_t\bar\mu=(\sfH_t\bar f)\mm\quad\forevery t\ge0.
\]
\end{theorem}

\subsection{Convergence of Heat flows in the general non linear case}

This section is devoted to the proof 
of the following general convergence result:

\begin{theorem}[Convergence of heat flows]
  \label{thm:stabgf1} Let $\cX_n$, $n\in\N$, be a sequence of $\CD(K,\infty)$ spaces converging to a limit space $\cX_\infty$ in the pmG-sense. Then with the notation of Remark \ref{re:assnot} the following holds.
   
  Let $(\bar\mu_n)\subset \probt X$ be such that
  \begin{equation}
    \Entn(\bar\mu_n)\to\Enti(\bar\mu_\infty)<\infty,\qquad\qquad 
      \bar\mu_n\stackrel{W_2}\longrightarrow \bar\mu_\infty,
      \label{eq:forconv}
  \end{equation}
as $n\to\infty$.  Then the 
  solutions $\mu_{n,t}=\h_{n,t}(\bar\mu_n)$, $t\ge0$,
  of the $W_2$ gradient flow of $\Entn$
  satisfy
\begin{subequations}
  \label{subeq:conclusions}
  \begin{align}
    \mu_{n,t}&\stackrel{W_2}\longrightarrow\mu_{\infty,t}\quad
  &\forevery t> 0,\\
  \Entn(\mu_{n,t})&\longrightarrow\Enti(\mu_{\infty,t})\quad
  &\forevery  t> 0,\\
  |\Dm^-\Entn|(\mu_{n,t})&\longrightarrow|\Dm^-\Enti|(\mu_{\infty,t})\quad
  &\forevery t\in (0,\infty)\setminus\mathcal S,\\
  |\dot\mu_{n,t}|&\longrightarrow|\dot\mu_{\infty,t}|&\quad
  \text{for a.e.\ } t> 0,
\end{align}
\end{subequations}
 where $\mathcal S$ is the set (at most countable) of discontinuity
points
of $t\mapsto |\rmD^-\Enti|(\mu_{\infty,t})$.
\end{theorem}

\noindent
We split the \emph{proof of Theorem \ref{thm:stabgf1}} in various
  steps. 

\subsubsection*{Tightness and compactness results}
\begin{lemma}\label{le:sollevo}
Let $\psi:\Y \to[0,+\infty]$ be a function with compact sublevels. Then for every non trivial interval $I=[a,b]\subset \R$ the functional $\Psi:\rmC(I,\Y )\to[0,+\infty]$ defined by
\[
\Psi(\gamma):=\int_I \psi(\gamma_t)\,\d t+\Ecurve\gamma=
\left\{
\begin{array}{ll}
  \displaystyle{\int_I\Big(\psi(\gamma_t)+|\dot\gamma_t|^2\Big)\,\d t,}&\qquad\textrm{ if }\gamma\in \AC{2}I\Y ,\\
+\infty&\qquad\textrm{ otherwise},
\end{array}
\right.
\]
has compact sublevels on $\rmC(I,\Y )$ (recall that the latter is endowed with the $\sup$ distance).
\end{lemma}
\begin{proof}
Since $\Ecurve\cdot$ is a l.s.c.~functional,
  Fatou's Lemma and the lower semicontinuity of $\psi$ on $\Y $
  ensure that $\Psi$ is lower semicontinuous on $\rmC(I,\Y )$. 

Let $(\gamma^n)\in \AC2I\Y $ be a sequence satisfying 
  $\Psi(\gamma^n)\le C<\infty$.
  The bound $\Ecurve{\gamma^n}\le C$ 
  easily yields the H\"older equicontinuity estimate 
  \begin{equation}
    \label{eq:44}
    \sfd(\gamma^n_{t+h},\gamma^n_t)\leq \sqrt {C\,h},\quad
    \text{whenever}\quad
    a\le t\le t+h\le b.
  \end{equation}
  By \cite[Theorem 2, (1.24)]{Rossi-Savare03} 
  $(\gamma^n)$ is relatively compact with respect to the convergence
  in measure: in particular, we can find a subsequence 
  $(\gamma^{n_k})$, 
  a $\Leb 1$-negligible set $N\subset I$, 
  and a limit curve $\gamma^\infty:I\setminus N\to \Y $ such that 
  $\lim_{k\to\infty}\sfd(\gamma^{n_k}_t,\gamma^\infty_t)=0$ 
  for every $t\in I\setminus N$.
  The uniform bound \eqref{eq:44} shows that $\gamma^\infty$
  is also $1/2$ H\"older continuous, it can therefore be extended
  to $I$ by the density of $I\setminus N$ and the completeness of $\Y $,
  and the resulting convergence of $\gamma^{n_k}$ to $\gamma^\infty$ is
  uniform.
\end{proof}

\begin{proposition}[Compactness in $\prob{C([0,T),\Y )}$]\label{prop:comp}
Let $(\bar\mu_n)\subset \probt \Y $ be such that
\begin{equation}
\label{eq:percomp}
\sup_{n\in\N}\Entn(\bar\mu_n)<\infty,\qquad\qquad\sup_{n\in\N}\int\sfd^2(\cdot,\bar x_n)\,\d\bar\mu_n<\infty,
\end{equation}
and let $\mu_{n,t}=\h_{n,t}(\bar\mu_n)$, $t\ge0$.
For every $0<T\le \frac 1{8\Co}$ the sequence of 
plans $\ppi_n\in\prob{C([0,T],\Y )}$ associated to the absolutely continuous curve 
$[0,T]\ni t\mapsto \mu_{n,t}$ via Proposition \ref{prop:lisini}, is 
%
tight in $\prob{C([0,T),\Y )}$.
\end{proposition}
\begin{proof}$\ $\\
\noindent{\bf Step 1: tightness of $\{\mu_{n,t}\}_{n,t}$.} Taking into account that $t\mapsto\Entn(\mu_{n,t})$ is non increasing and \eqref{eq:chiave}, for any $t\in[0,T]$ we get
\[
\begin{split}
\Enttn(\mu_{n,t})&=\Entn(\mu_{n,t})+\Co W_2^2(\mu_{n,t},\delta_{\bar x_n})+\log(\z_n)\\
&\leq\Entn(\bar \mu_{n})+\Co\left( W_2(\bar \mu_{n},\delta_{\bar x_n})+\int_0^t|\dot\mu_{n,s}|\,\d s\right)^2+\log(\z_n)\\
&\leq \Entn(\bar \mu_{n})+2\Co W^2_2(\bar \mu_{n},\delta_{\bar x_n})+2T\int_0^T|\dot\mu_{n,s}|^2\,\d s+\log(\z_n).
\end{split}
\]
Hence from  the bound \eqref{eq:speedcontr}, assumption \eqref{eq:percomp} and the fact that $\sup_n\log(\z_n)<\infty$  we  conclude that 
\[
\sup_{n\in \N\atop t\in[0,T]}\Enttn(\mu_{n,t})<\infty,
\]
and the tightness of $\{\mu_{n,t}\}_{n\in\N,t\in[0,T]}$ follows from Proposition \ref{prop:enttight}.\\
\noindent{\bf Step 2: tightness of $(\ppi_n)$}. Since the set $\{\mu_{n,t}\}_{t\in[0,T],n\in\N}$ is tight, by Theorem \ref{thm:prok} there exists a function $\psi:\Y \to[0,+\infty]$ with compact sublevels such that
\begin{equation}
\label{eq:tightmeasures}
\sup_{t\in[0,T],n\in\N}\int\psi\,\d\mu_{n,t}<\infty.
\end{equation}
Define $\Psi:\rmC([0,T],\Y )\to[0,+\infty]$ by $\Psi(\gamma):=\Ecurve\gamma+\int_0^T\psi(\gamma_t)\,\d t$. Lemma \ref{le:sollevo} ensures that the sublevels of $\Psi$ are compact in $\rmC([0,T],\Y )$. Using  \eqref{eq:optimal} we get
\[
\begin{split}
\int\Psi\,\d\ppi_n&=\iint_0^T\psi(\gamma_t)\,\d t\,\d\ppi_n(\gamma)+
\int\Ecurve\gamma\,\d\ppi_n(\gamma)=\int_0^T\Big(\int\psi\,\d\mu_{n,t}+|\dot\mu_{n,t}|^2\Big)\,\d t.
\end{split}
\]
The right hand side of this expression is uniformly bounded in $n$ thanks to \eqref{eq:tightmeasures} and \eqref{eq:speedcontr}, thus the conclusion follows by Prokhorov's Theorem \ref{thm:prok}.
\end{proof}
\subsubsection*{$\Gamma$-$\liminf$ estimates}

The next two propositions are valid without any Ricci curvature assumption
 (but still assuming the growth condition \eqref{eq:expcontr2}) 
 and are of independent interest.
\begin{proposition}[$\Gamma$-$\limi$ for Entropy + second moment]\label{prop:glimkinent} 
With the same assumptions and notation of Remark \ref{re:assnot},  let $(\mu_n)\subset\probt X $ be a sequence \narrowly\  converging to some $\mu_\infty\in\probt X $.
Then
\[
\Enti(\mu_\infty)+\Co\int\sfd^2(\cdot,\bar{x}_\infty)\,\d\mu_\infty\leq\limi_{n\to\infty}\left(\Entn(\mu_n)+\Co\int\sfd^2(\cdot,\bar x_n)\,\d\mu_n\right).
\]
\end{proposition}
\begin{proof}
Direct  consequence of \eqref{eq:chiave}, $\z_n\to \z_\infty$ and \eqref{eq:joint}.
\end{proof}
\begin{proposition}[$\Glimi$ for kinetic energy minus second moment]\label{prop:w2nar} 
With the same assumptions and notation of Remark \ref{re:assnot}, let $T\leq \frac{1}{4\Co}$ and $(\ppi_n)\subset \prob{C([0,T],X )}$ a sequence \narrowly\  converging to some $\ppi_\infty\in\prob{ C([0,T],X )}$. Assume also that 
\[
(\e_0)_\sharp\ppi_n\in\probt X ,\ \forall n\in\bar\N,
\quad\textrm{ with }\quad W_2\big( (\e_0)_\sharp\ppi_n,(\e_0)_\sharp\ppi_\infty\big)\to0.
\]
Then 
\begin{equation}
\label{eq:pezzo2}
\begin{split}
&\int\Big(\frac12\Ecurve \gamma-\Co\sfd^2(\gamma_T,\bar x_\infty)\Big)\,\d\ppi_\infty(\gamma)
\leq\limi_{n\to\infty}\int\Big(\frac12\Ecurve{\gamma}-\Co\sfd^2(\gamma_T,\bar x_n)\Big)\,\d\ppi_n(\gamma).
\end{split}
\end{equation}
If we further assume that
\begin{equation}
\label{eq:kinug}
\lim_{n\to\infty}\iint_0^T|\dot\gamma_t|^2\,\d t\,\d\ppi_n(\gamma)=\iint_0^T|\dot\gamma_t|^2\,\d t\,\d\ppi_\infty(\gamma)<\infty,
\end{equation}
then $(\e_t)_\sharp\ppi_n$ converges to $(\e_t)_\sharp\ppi_\infty$ in $(\probt \Y ,W_2)$ for any $t\in[0,T]$.
\end{proposition}
\begin{proof}
  Let $F_n:\rmC([0,T],\Y )\to[0,+\infty]$, $n\in\bar\N$ be the l.s.c.~functionals
  given by  
  \[
  F_n(\gamma):= \frac12\Ecurve \gamma
  -\Co\sfd^2(\gamma_T,\bar x_n)+2\Co\sfd^2(\gamma_0,\bar x_n).
  \] 
  Since $\bar x_n\to\bar x_\infty$, we easily get
  \[
  \sup_{t\in[0,T]}\sfd(\gamma_{n,t},\gamma_{\infty,t})\to 0,\qquad\Rightarrow\qquad F_\infty(\gamma_\infty)\leq\limi_nF_n(\gamma_n).
  \]
  We claim that the $F_n$'s are non negative. Indeed, from $T\leq\frac{1}{4\Co}$ we have
  \[
  \begin{split}
    \Co\sfd^2(\gamma_T,\bar x_n)\leq \Co\left(\int_0^T|\dot\gamma_t|\,\d t+\sfd(\gamma_0,\bar x_n)\right)^2\leq\frac12\int_0^T|\dot\gamma_t|^2\,\d t+2\Co\sfd^2(\gamma_0,\bar x_n).
  \end{split}
  \]
  Therefore applying Lemma \ref{le:lemmino} with $\rmC([0,T],\Y )$ in place of $X$, $\nu_n:=\ppi_n$, 
  and recalling that $(\e_0)_\sharp\ppi_n\to(\e_0)_\sharp\ppi_\infty$ in $(\probt \Y ,W_2)$ by assumption, we get \eqref{eq:pezzo2}.
  
  For the second part of the statement, notice that the lower semicontinuity of the kinetic energy and assumption \eqref{eq:kinug} give that $\lim_{n\to\infty}\iint_0^t|\dot\gamma_s|^2\,\d s\,\d\ppi_n(\gamma)=\iint_0^t|\dot\gamma_s|^2\,\d s\,\d\ppi_\infty(\gamma)$ for any $t\in[0,T]$. Hence, up to diminishing $T$, we can assume without loss of generality that $t=T$. Now observe that 
  the positivity of $F_n$, \eqref{eq:pezzo2} and  \eqref{eq:kinug} give
  \[
  -\infty<-\Co\int \sfd^2(\cdot,\bar x_\infty)\,\d(\e_T)_\sharp\ppi_\infty\leq\limi_{n\to\infty}-\Co\int \sfd^2(\cdot,\bar x_n)\,\d(\e_T)_\sharp\ppi_n.
  \]
  The \narrow\ convergence of $\ppi_n$ to $\ppi$ gives the \narrow\ convergence of $(\e_T)_\sharp\ppi_n$ to $(\e_T)_\sharp\ppi_\infty$, thus the conclusion comes from Proposition \ref{prop:narw2}.
\end{proof}
In the next proposition, to get the $\Glimi$ inequality for the slope, 
the uniform lower Ricci curvature bound plays a crucial role.
\begin{proposition}[$\Gamma$-$\limi$ for the slope]\label{prop:glimisl} 
 With the same notation and assumption as in Remark \ref{re:assnot}, assume furthermore that for some $K\in\R$ the spaces $\cX_n$ are all $\CD(K,\infty)$ spaces (so that by Theorem \ref{thm:stabilityCD} also the limit space $\cX_\infty$ is a $\CD(K,\infty)$ space). 
  
   Then for any sequence $n\mapsto\mu_n\in\probt{X_n}  $ \narrowly\  converging to some $\mu_\infty\in\probt \Y $ it holds
  \[
  |\Dm^-\Enti|(\mu_\infty)\leq \limi_{n\to\infty}|\Dm^-\Entn|(\mu_n).
  \]
\end{proposition}
\begin{proof}$\ $\\
\noindent{\bf Step 1: From \narrow\ to $W_2$ convergence.}
For $R>0$, let $h_R:\R^+\to[0,1]$ be given by $h_R(z):=\max\{\{\min\{2-z/R,1\},0\}$ and define the $1/R$-Lipschitz functions $\nchi_{n,R}:\Y \to[0,1]$, $n\in \bar\N$,
 as $\nchi_{n,R}(x):=h_R(\sfd_n(x,\overline x_n))$.

Clearly, for some $R_0,n_0$, the probability measures $\mu_{\infty,R}:=c_{\infty,R}\nchi_{\infty,R}^2\mu_\infty$ and  $\mu_{n,R}:=c_{n,R}\nchi_{n,R}^2\mu_n$ are well defined for any $R>R_0$ and $n\geq n_0$, where $c_{\infty,R},c_{n,R}$ are the normalization constants, and it holds $c_{n,R}\to 1$ as $R\to\infty$ and $c_{n,R}\to c_{\infty,R}$ as $n\to\infty$.

From the convergence of $\bar x_n$ to $\bar x_\infty$, we get the \narrow\ convergence of $(\mu_{n,R})$ to $\mu_{\infty,R}$ as $n\to\infty$. Therefore, since their supports are uniformly bounded, the convergence is in the $W_2$-topology by Proposition \ref{prop:narw2}.
\\*
\noindent{\bf Step 2: control of the slope under cutoff.} Given that the slope is $+\infty$ outside the domain of the entropy, up to pass to subsequences, with no loss of generality we can assume that $\mu_n\ll\mm_n$ for every $n\in\N$. Let $\mu_n=f_n^2\mm_n$, so that $\mu_{n,R}=c_{n,R}(f_n\nchi_{n,R})^2\mm_n$. From \eqref{eq:weaklip} and \eqref{eq:calcrule} we get
 \[
 \weakgrad{(f_n\nchi_{n,r})}\leq\weakgrad{f_n}+|f_n|/R.
 \]
Squaring, integrating, recalling that $\int f_n^2\,\d\mm_n=1$ and  the definition of $\C_n$ given in \eqref{eq:cheeger} and the discussion thereafter, we get 
\[
\C_n({\sqrt{c_{n,R}}f_n\nchi_{n,R}})\leq c_{n,R}\left(\C_n(f_n)\Big(1+\frac1{2R}\Big)+\frac1{2R}+\frac1{2R^2}\right),\qquad\forall R\geq R_0, n\geq n_0.
\]
Therefore using \eqref{eq:slopefish} twice, we get that for $R>R_0$  it holds
\begin{equation}
\label{eq:pezzo1}
|\Dm^-\Entn|^2(\mu_{n,R})\leq c_{n,R}\left(|\Dm^-\Entn|^2(\mu_{n})\Big(1+\frac1{2R}\Big)+\frac4{R}+\frac4{R^2}\right).
\end{equation}
\noindent{\bf Step 3: $\Glimi$ under $W_2$-convergence.} Fix $\nu\in D(\Enti)$ and use the $\Glims$ inequality in Proposition \ref{prop:GammaConvergence} to find a sequence $(\nu_n)\subset\probt \Y $ such that $\Entn(\nu_n)\to \Enti(\nu)$ and $W_2(\nu_n,\nu)\to 0$ as $n\to\infty$. Using the $W_2$-convergence of $\mu_{n,R}$ to $\mu_{\infty,R}$ as $n\to\infty$ and the $\Glimi$ inequality in Proposition \ref{prop:GammaConvergence} we have 
\[
\begin{split}
\frac{\Enti(\mu_{\infty,R})-\Enti(\nu)}{W_2(\mu_{\infty,R},\nu)}+\frac K2W_2(\mu_{\infty,R},\nu)&\leq\limi_{n\to\infty}\frac{\Entn(\mu_{n,R})-\Entn(\nu_n)}{W_2(\mu_{n,R},\nu_n)}+\frac K2W_2(\mu_{n,R},\nu_n)\\
&\leq \limi_{n\to\infty}|D^-\Entn|(\mu_{n,R}),
\end{split}
\]
having used \eqref{eq:repslope} in the last step. Taking the positive parts, the supremum over $\nu\in D(\Enti)$ and using again  \eqref{eq:repslope}  we conclude
\begin{equation}
\label{eq:glimisl}
|\Dm^-\Enti|(\mu_{\infty,R})\leq \limi_{n\to\infty}|\Dm^-\Entn|(\mu_{n,R}),\qquad\forall R\geq R_0.
\end{equation}
\noindent{\bf Step 4: back to $\mu$ and conclusion.} From the definition it is clear that $W_2(\mu_{\infty,R},\mu_\infty)\to 0$ as $R\to\infty$ (because there is \narrow\ convergence and 2-uniform integrability), hence recalling the $W_2$-lower semicontinuity of $|\Dm^-\Enti|$ (see \eqref{eq:slopesemi})  we get 
\begin{equation}
\label{eq:sempre}
|\Dm^-\Enti|(\mu_\infty)\leq \limi_{R\to\infty}|\Dm^-\Enti|(\mu_{\infty,R}).
\end{equation}
The inequalities \eqref{eq:pezzo1}, \eqref{eq:glimisl} and \eqref{eq:sempre} give the conclusion.
\end{proof}
\begin{remark}{\rm
It is worth underlying that although the structure of the proof of Proposition \ref{prop:glimisl} is very similar to the one presented in \cite{Gigli10} for the compact case, actually it is conceptually more involved. Indeed, to gain $W_2$-convergence from \narrow\ one we used a cut-off argument, and to control  the variation of the slope under this procedure (inequality \eqref{eq:pezzo1}) we used as crucial tool  the non-trivial identification of the squared slope with the Fisher information (identity \eqref{eq:slopefish}), which in turn is a consequence of the fine analysis carried out in \cite{Ambrosio-Gigli-Savare-preprint11a}.
}\fr\end{remark}

\subsubsection*{Conclusion of the proof of Theorem \ref{thm:stabgf1}}

\begin{proof}
Let $T:=\frac1{8C}$ and, for every $n\in\N$, $\ppi_n\in\prob{C([0,T],X )}$ a plan associated to $[0,T]\ni t\mapsto \mu_{n,t}$ via Proposition \ref{prop:lisini}, so that in particular
\begin{equation}
\label{eq:benen}
\int_0^T|\dot\mu_{n,t}|^2\,\dt=\int\Ecurve\gamma\,\d\ppi_n(\gamma).
\end{equation}
Thanks to \eqref{eq:forconv}, Proposition \ref{prop:comp} is applicable, thus $(\ppi_n)$ is tight in $\prob{C([0,T],X)}$ and up to pass to a subsequence, not relabeled, we can assume that it \narrowly\  converges to some $\ppi_\infty$. The convergences in \eqref{eq:forconv},  the bound \eqref{eq:speedcontr},
the identity \eqref{eq:benen}, and the lower semicontinuity of the kinetic energy combined with Lemma \ref{le:lemmino},  give
\[
\int\Ecurve\gamma\,\d\ppi_\infty(\gamma)\le \limi_{n\to\infty}\int\Ecurve\gamma\,\d\ppi_n(\gamma)
<\infty.
\]
Put $\mu_t:=(\e_t)_\sharp\ppi_\infty$ and recalling \eqref{eq:nonopt},
we get that $\mu\in\AC2{[0,T]}{(\probt \Y ,W_2)}$ with
\begin{equation}
\label{eq:nonopt2}
\int_0^T|\dot\mu_t|^2\,\d t\leq \int\Ecurve\gamma\,\d\ppi_\infty(\gamma).
\end{equation}
The \narrow\ convergence of $(\ppi_n)$ to $\ppi_\infty$ ensures that $n\mapsto \mu_{n,t}$ \narrowly\  converges to $\mu_t$ for any $t\in[0,T]$.
%
Proposition \ref{prop:w2nar} yields 
\begin{equation}
\label{eq:kin}
\begin{split}
\underbrace{\int\Big(\frac12\Ecurve\gamma
-\Co\sfd^2(\gamma_T,\bar{x}_\infty)\Big)\,\d\ppi_\infty(\gamma)}_{=:{\rm a}}\leq\limi_{n\to\infty}
\underbrace{\int\Big(\frac12\Ecurve\gamma-\Co\sfd^2(\gamma_T,\bar x_n)\Big)\,\d\ppi_n(\gamma)}_{=:{\rm a}_n}.
\end{split}
\end{equation}
Applying Proposition  \ref{prop:glimkinent} to $\mu_n:=(\e_T)_\sharp\ppi_n$, $\mu_\infty:=(\e_T)_\sharp\ppi_\infty$ we obtain
\begin{equation}
\label{eq:entmom}
\underbrace{\Enti\big((\e_T)_\sharp\ppi_\infty\big)+\Co\int\sfd^2(\gamma_T,\bar{x}_\infty)\,\d\ppi_\infty(\gamma)}_{=:{\rm b}}\leq\limi_{n\to\infty}\bigg(\underbrace{\Entn\big((\e_T)_\sharp\ppi_n\big)+\Co\int\sfd^2(\gamma_T,\bar x _n)\,\d\ppi_n(\gamma)}_{=:{\rm b}_n}\bigg).
\end{equation}
Also, from Proposition \ref{prop:glimisl} and Fatou's lemma we get
\begin{equation}
\label{eq:slope}
\underbrace{\frac12\int_0^T|\Dm^-\Enti|^2(\mu_{t})\,\d t}_{=:{\rm c}}\leq\limi_{n\to\infty}\underbrace{\frac12\int_0^T|\Dm^-\Entn|^2(\mu_{n,t})\,\d t}_{=:{\rm c}_n}.
\end{equation}
Adding up \eqref{eq:kin}, \eqref{eq:entmom} and \eqref{eq:slope} and taking \eqref{eq:nonopt2} into account we get
\begin{equation}
\label{eq:quasi}
\begin{split}
\Enti\big(\mu_{T}\big)&+\frac12\int_0^T|\dot\mu_{t}|^2\,\d t+\frac12\int_0^T|\Dm^-\Enti|^2
(\mu_{t})
\,\d t\leq \limi_{n\to\infty}({\rm a}_n+{\rm b}_n+{\rm c}_n).
\end{split}
\end{equation}
From the  fact that $t\mapsto\mu_{n,t}$ is a gradient flow of $\Entn$, \eqref{eq:defgf}, 
\eqref{eq:forconv} and \eqref{eq:benen}, we obtain
\begin{equation}
\label{eq:quasi2}
{\rm a}_n+{\rm b}_n+{\rm c}_n=\Entn(\bar\mu_n)\to \Enti(\bar\mu_\infty).
\end{equation}
%
Therefore \eqref{eq:quasi} yields
\begin{equation}
\label{eq:finalmente}
\Enti\big(\mu_{T}\big)+\frac12\int_0^T|\dot\mu_{t}|^2\,\d t+\frac12\int_0^T|\Dm^-\Enti|^2
(\mu_{t})
\,\d t
\leq\Enti(\bar\mu_\infty),
\end{equation}
which means, according to \eqref{eq:defgf2}, that $[0,T]\ni t\mapsto \mu_t$ is the (restriction to $[0,T]$ of the) gradient flow of $\Enti$ starting from $\bar\mu_\infty$. Thus $\mu_t=\mu_{\infty,t}$ for any $t\in[0,T]$. 

By \eqref{eq:pergf}, the opposite inequality in \eqref{eq:finalmente} also holds, and since in proving  \eqref{eq:finalmente} we used \eqref{eq:quasi}, which in turn was proved using \eqref{eq:nonopt2}, we deduce that equality holds in \eqref{eq:nonopt2}, i.e.
\begin{equation}
\label{eq:bene}
\int_0^T|\dot\mu_t|^2\,\d t= \iint_0^T|\dot\gamma_t|^2\,\d t\,\d\ppi_\infty(\gamma).
\end{equation}
Notice also that  equality  in \eqref{eq:finalmente} reads as ${\rm a}+{\rm b}+{\rm c}=\Enti(\bar\mu_\infty)$, therefore from \eqref{eq:kin}, \eqref{eq:entmom}, \eqref{eq:slope}, \eqref{eq:quasi} and  \eqref{eq:quasi2}  we get
\[
\begin{split}
\Enti(\bar\mu_\infty)={\rm a}+{\rm b}+{\rm c}&\leq \limi_{n\to\infty}{\rm a}_n+\limi_{n\to\infty}{\rm b}_n+\limi_{n\to\infty}{\rm c}_n\leq \limi_{n\to\infty}({\rm a}_n+{\rm b}_n+{\rm c}_n)\\
&\leq \lims_{n\to\infty}({\rm a}_n+{\rm b}_n+{\rm c}_n)=\lims_{n\to\infty}\Entn(\bar\mu_n)= \Enti(\bar\mu_\infty),
\end{split}
\]
which forces ${\rm a}_n\to {\rm a}$, ${\rm b}_n\to {\rm b}$ and ${\rm c}_n\to {\rm c}$. Proposition \ref{prop:glimisl} and the convergence ${\rm c}_n\to {\rm c}$  give
\[
|\Dm^-\Enti|^2(\mu_t)\leq\limi_{n\to\infty}|\Dm^-\Entn|^2(\mu_{n,t}),\qquad a.e.\ t\in[0,T],\\
\]
and
\begin{equation}
\label{eq:mond}
\int_0^T|\Dm^-\Enti|^2(\mu_t)\,\d t=\lim_{n\to\infty}\int_0^T|\Dm^-\Entn|^2(\mu_{n,t})\,\d t,
\end{equation}
which implies
\[
|\Dm^-\Enti|^2(\mu_t)=\lim_{n\to\infty}|\Dm^-\Entn|^2(\mu_{n,t}),\qquad a.e.\ t\in[0,T].
\]
Taking the regularity property \eqref{eq:slopereg} into account we get
the 
convergence for all $t\in[0,T]$ with the possible exception of the countable
discontinuity set of $t\mapsto |\rmD^-\Enti|(\mu_{\infty,t})$.
From  \eqref{eq:kinugsl}  we also get $|\dot\mu_{\infty,t}|=\lim_{n}|\dot\mu_{n,t}|$ for a.e. $t\in[0,T]$ which together with \eqref{eq:mond} and the equalities \eqref{eq:bene}, \eqref{eq:benen} give that \eqref{eq:kinug} holds. Therefore from the second part of Proposition \ref{prop:w2nar} we get that $(\mu_{n,t})$ converges to $\mu_{\infty,t}$ w.r.t.~$W_2$ for any $t\in[0,T]$. Finally, the convergence $\Entn(\mu_{n,t})\to\Enti(\mu_t)$ follows from ${\rm b}_n\to{\rm b}$ and the convergence of the second moments (replace $T$ by $t\in[0,T]$ in the whole argument to get this latter convergence). 

Observe now that all these results do not really depend on the particular \narrowly\  converging sequence $(\ppi_n)$ chosen at the beginning, so that they are valid for the full original sequence.
Thus we proved all the stated convergence properties for $t\in[0,\frac1{8\Co}]$. In particular, we have that
\[
W_2(\mu_{n,1/(8\Co)},\mu_{\infty,1/(8\Co)})\to 0,\qquad\qquad\Entn(\mu_{n,1/(8\Co)})\to\Enti(\mu_{\infty,1/(8\Co)})<\infty.
\]
Hence the argument can be repeated with $\mu_{n,1/(8\Co)},\mu_{\infty,1/(8\Co)}$ in place of $\bar \mu_n,\bar\mu_\infty$. Iterating, we get the result on the whole $[0,\infty)$.
\end{proof}
\subsection{A consequence: Mosco-convergence of the slopes}
Now that we have at disposal the convergence of the heat flows, we can improve the result of Proposition \ref{prop:glimisl} and obtain the full Mosco-convergence of the slopes of the relative entropies:
\begin{corollary}[Mosco-convergence of the slopes]\label{cor:moscoslope}
 With the same notation and assumption as in Remark \ref{re:assnot}, assume furthermore that for some $K\in\R$ the spaces $\cX_n$ are all $\CD(K,\infty)$ spaces (so that by Theorem \ref{thm:stabilityCD} also the limit space $\cX_\infty$ is a $\CD(K,\infty)$ space). Then the following holds:

\begin{itemize}
\item\underline{weak $\Gamma-\liminf$}    For any sequence $n\mapsto\mu_n\in\probt{X_n}  $ \narrowly\  converging to some $\mu_\infty\in\probt {X_\infty}$ it holds
  \[
  |\Dm^-\Enti|(\mu_\infty)\leq \limi_{n\to\infty}|\Dm^-\Entn|(\mu_n).
  \]
\item\underline{strong $\Gamma-\limsup$} For any $\mu_\infty\in\probt{X_\infty}$ there exists a sequence   $n\mapsto\mu_n\in\probt{X_n}  $ $W_2$-converging to $\mu_\infty$ such that
  \[
  |\Dm^-\Enti|(\mu_\infty)\geq \lims_{n\to\infty}|\Dm^-\Entn|(\mu_n).
  \]
\end{itemize}
\end{corollary}
\begin{proof}
The first part is precisely the content of Proposition \ref{prop:glimisl}, so we pass to the second. Without loss of generality we assume $  |\Dm^-\Enti|(\mu_\infty)<\infty$ (otherwise there is nothing to prove) so that in particular $\Enti(\mu_\infty)<\infty$.    Invoking the second part of  Proposition \ref{prop:GammaConvergence} we can find
  a sequence $\mu_n$ $W_2$-converging to $\mu_\infty$ such that 
  $\Entn(\mu_n)\to\Enti(\mu_\infty)$. Setting
  $\mu_{n,t}:=\h_{n,t}(\mu_n)$,
  \eqref{eq:speedcontr} yields the uniform H\"older estimate
  \begin{displaymath}
    W_2(\mu_{n,t},\mu_n)\le C\sqrt t\quad t\in [0,T],
  \end{displaymath}
  for some $C,T>0$ independent of $n$.
  Let us choose a vanishing sequence of times $t_k\in
  (0,T)\setminus \mathcal S$, $\mathcal S$ being as in Theorem 
  \ref{thm:stabgf1}.
  Applying (\ref{subeq:conclusions}c), 
  we can find an increasing sequence 
  $k\mapsto n_k\in \N$ such that 
  for every $n\ge n_k$
  \begin{displaymath}
    |\rmD^-\Entn|(\mu_{n,t_k})\le |\rmD^-\Enti|(\mu_{\infty,t_k})
    +2^{-k}\le 
    \rme^{Kt_k} |\rmD^-\Enti|(\mu_{\infty})+2^{-k}.
  \end{displaymath}
  where in the last inequality we applied \eqref{eq:slopereg}.
  The sequence $\nu_n:=\mu_{n,t_k}$ whenever $n_k\le n<n_{k+1}$ 
  thus satisfies the desired inequality.
\end{proof}

\section{Mosco-convergence of the Cheeger energies.}
\label{sec:Mosco}
In this section we will prove a $\Gamma$-convergence result
for Cheeger energies
in the same spirit of Proposition 
\ref{prop:GammaConvergence}.

As usual, we adopt the assumptions and notation of Remark \ref{re:assnot}.
We denote by $\C_n$ the Cheeger energy associated to the measure
$\mm_n$ and by $L^p_{\rm loc}(\Y ,\mm_n)$ the Lebesgue space of 
$\mm_n$-measurable functions whose restriction to each bounded set
$B\subset \Y $ belongs to $L^p(B,\mm_n)$.

\subsection{Weak and strong $L^2$-convergence w.r.t.~varying
  measures.}
Let us first introduce a notion of weak/strong convergence 
for real valued functions defined in varying $L^p$-spaces
and strongly inspired by the theory of  Young measures,
see e.g.~\cite{Ambrosio-Fusco-Pallara00}. 
We refer to 
\cite[\S 5.4]{Ambrosio-Gigli-Savare08} for
similar definitions and the relevant proofs, that can be easily adapted to cover
the present case
of measures in $\Mloc \Y $ by the rescaling approach
we adopted in Section \ref{sec:pointedD}. 
The basic idea 
is to study the convergence of a sequence of functions $f_n\in L^1_{\rm
  loc}(\Y ,\mm_n)$ by the corresponding measures $\mmu_n=(\ii\times f_n)_\sharp \mm_n$ 
in $\Mloc{\Y \times\R}$, $\ii:X\to X$ being the identity map. 
Here we state everything in the equivalent terms of weak
convergence and test functions.
\begin{definition}
  \label{def:w-s}
  Let  $f_n\in L^1_{\rm loc}(\Y ,\mm_n)$, $n\in \bar \N$.
  We say that $(f_n)$ weakly converges to $f_\infty$ if
  \begin{equation}
    \label{eq:51}
    \int \varphi f_n\,\d\mm_n\longrightarrow
    \int\varphi f_\infty\,\d\mm_\infty
    \quad
    \forevery \varphi\in \Cc \Y .
  \end{equation}
 For $p\in(1,\infty)$ we say that $(f_n)$ $L^p$-weakly converges to $f_\infty$ provided \eqref{eq:51} holds and moreover
 \begin{equation}
\label{eq:perweak}
 \sup_{n\in\N}\int |f_n|^p\,\d\mm_n<\infty.
\end{equation}
 If $(f_n)$ is $L^p$-weakly converging to $f_\infty$ and furthermore
  \begin{equation}
    \label{eq:48}
    \lims_{n\to\infty} \int |f_n|^p\,\d\mm_n\le \int |f_\infty|^p\,\d\mm_\infty<\infty,
  \end{equation}
  then we say that it is $L^p$-strongly converging to $f_\infty$.
  
  If $f_n\in L^1_+(\Y ,\mm_n)$, we say that $(f_n)$ converge to $f_\infty$
  in
  the $W_2$-Entropy sense if \eqref{eq:51} holds and
  \begin{equation}
    \label{eq:55}
    \begin{split}
    \int \sfd^2(x,\bar x_n)f_n\,\d\mm_n&\longrightarrow
    \int \sfd^2(x,\bar x)f_\infty\,\d\mm_\infty<\infty\\
        \int f_n\log f_n\,\d\mm_n&\longrightarrow
    \int f_\infty\log f_\infty\,\d\mm_\infty<\infty.    
    \end{split}
  \end{equation}
\end{definition}
\begin{remark}
  \label{rem:KuShi}
  \upshape
  In the Hilbertian case $p=2$ a different approach,
  still leading to similar notions and results,
  can be found in 
  \cite[Section 2.2]{Kuwae-Shioya03}.
  It is based on the construction of a family of 
  linear maps $\Phi_n:\mathcal C\to L^2(X,\mm_n)$ 
  defined in a dense subspace $\mathcal C$ of $L^2(X,\mm_\infty)$
  and satisfying 
  $\lim_{n\to\infty}\|\Phi_n
  f\|_{L^2(X,\mm_n)}=\|f\|_{L^2(X,\mm_\infty)}$.
  In our situation the maps $\Phi_n$ could be defined 
  by a disintegration technique, as in the proof of 
  Proposition \ref{prop:GammaConvergence};
  however we prefer to adopt the more direct approach
  of  Definition \ref{def:w-s} above, 
  that takes advantage of the particular structure
  of the spaces $L^2(X,\mm_n)$ associated to the weakly convergent
  measures $\mm_n$. \fr
\end{remark}
It would not be difficult to show, arguing as for Proposition
\ref{prop:GammaConvergence}
(see also \cite[\S~5.2]{Ambrosio-Gigli-Savare08})
that along any weakly convergent sequence $(f_n)$
we always have
\begin{equation}
  \label{eq:52}
  \limi_{n\to\infty}\int |f_n|^p\,\d\mm_n\ge \int |f_\infty|^p\,\d\mm_\infty,
\end{equation}
so that \eqref{eq:48} is in fact a limit. Moreover the same argument
of \cite[Theorem 5.4.4]{Ambrosio-Gigli-Savare08} 
shows that strong convergence in the $L^p$ sense \eqref{eq:51}--\eqref{eq:48} implies
\begin{equation}
  \label{eq:56}
  \begin{gathered}
    \int \zeta(y,f_n(y))\,\d\mm_n(y)\longrightarrow
    \int\zeta(y,f_\infty(y))\,\d\mm_\infty(y) \\
    \forevery 
    \zeta\in \rmC(\Y \times\R)\ \text{with } 
    |\zeta(y,r)|\le \varphi(y)+C|r|^p,\quad
    \varphi\in \Cc \Y ,\ C\ge0.
  \end{gathered}
\end{equation}
A similar argument, based on the strict convexity of the function
$r\mapsto r\log r$ and on the decomposition \eqref{eq:chiave}
shows that convergence in the $W_2$-Entropy sense yields \eqref{eq:56}
with $p=1$.

We also remark that if $n\mapsto f_n\in L^p(X,\mm_n)$, $p\in(1,\infty)$, fulfills \eqref{eq:perweak}, then there is a subsequence $L^p$-weakly converging to some limit $f_\infty\in L^p(X,\mm_\infty)$.

In the following we shall need the convergence result
\begin{equation}
\label{eq:weakstrongscalar}
\left.\begin{array}{ll}
  f_n\stackrel{L^2}\longrightarrow &f_\infty,  \text{ strongly}\\
  g_n\stackrel{L^2}\longrightarrow &g_\infty, 
  \text{ weakly}
\end{array} \right\}
 \quad\Rightarrow\quad
  \int f_n\,g_n\d\,\mm_n\longrightarrow
  \int f_\infty\,g_\infty\,\d\mm_\infty,
\end{equation}
which can be proved by picking $\eps\in\R$, applying \eqref{eq:perweak} and \eqref{eq:52} to the $L^2$-weakly converging sequence $(f_n+\eps g_n)$ and \eqref{eq:48} to $(f_n)$ to pass to the limit in
\[
2\eps\int f_ng_n\,\d\mm_n=\int |f_n+\eps g_n|^2\,\d\mm_n-  \int |f_n|^2\,\d\mm_n-\eps^2\int |g_n|^2\,\d\mm_n
\]
and obtain
\[
\begin{split}
\limi_{n\to\infty}2\eps\int f_ng_n\,\d\mm_n&\geq \int |f_\infty+\eps g_\infty|^2\,\d\mm_\infty-  \int |f_\infty|^2\,\d\mm_\infty-\eps^2S\\
&\geq 2\eps\int f_\infty g_\infty\,\d\mm_\infty-\eps^2S,
\end{split}
\]
where $S:=\sup_{n\in\N}\int |g_n|^2\,\d\mm_n<\infty$. Then the claim follows by dividing by $\eps>0$ (resp. $\eps<0$) and letting $\eps\downarrow0$ (resp. $\eps\uparrow0$).

Finally, as for Proposition \ref{prop:GammaConvergence},
every $f_\infty\in L^p(\Y ,\mm_\infty)$ can be 
approximated
in the strong $L^p$-sense by a sequence $f_n\in L^p(\Y ,\mm_n)$.

The next result provides a useful criterium to improve
weak convergence, whenever one knows a bound on the 
Fisher information or on the Cheeger energy.
\begin{theorem}
  \label{thm:weaktostrong} Let $\cX_n$, $n\in\bar\N$, and $(X,\sfd)$ as in Remark \ref{re:assnot}, assume furthermore that for some $K\in\R$ the space $\cX_n$ is $\CD(K,\infty)$ for every $n\in\bar\N$ and  let $ f_n\in L^1_{\rm loc}(\Y ,\mm_n)$, $n\in\bar \N$. Then the following are true.
  \begin{enumerate}[i)]
  \item Let $\mu_n=f_n\mm_n\in\probt \Y$, $n\in\N$, be such that
\begin{equation}
\label{eq:aereo}
\sup_{n\in\N}W_2(\mu_n,\delta_{\bar x_n})+|\rmD^-\Entn|(\mu_n)<\infty.
\end{equation}
Then there exists a subsequence weakly converging to some $\mu_\infty=f_\infty\mm_\infty\in\probt\Y$.
   
If furthermore we assume that
    \begin{equation}
      \label{eq:57}
      W_2(\mu_n,\mu_\infty)\longrightarrow0,
      \end{equation}
as $n\to\infty$,    then $(f_n)$ converges to $f_\infty$ in the $W_2$-Entropy sense
    (thus \eqref{eq:55} and \eqref{eq:56} hold with
    $p=1$).
  \item Let  $(f_n)$ be converging to $f_\infty$ in the $L^2$-weak sense.
    Then
    \begin{equation}
      \label{eq:59}
\C_\infty(f_\infty)\leq      \limi_{n\to\infty}\C_n(f_n).
    \end{equation}
If furthermore we assume that
    \begin{equation}
      \label{eq:142}
\sup_{n\in\N}\C_n(f_n)<\infty\qquad\text{ and }\qquad\lim_{R\to\infty} \sup_{n} \int_{\Y \setminus B_{R}(\bar x_n)}
      f_n^2\,\d\mm_n=0,
    \end{equation}
    then $(f_n)$ converge to $f_\infty$ in the $L^2$-strong sense.
  \end{enumerate}
\end{theorem}
\begin{proof}$\ $\\
  \textbf{i)} Let us first recall that \eqref{eq:repslope}
  yields
  for every $\nu_n\in D(\Entn)$
\begin{equation}
\label{eq:rer2}
    \Entn(\mu_n)\le \Entn(\nu_n)+
    |\rmD^-\Entn|(\mu_n)W_2(\mu_n,\nu_n)+\frac{K_-}2W^2_2(\mu_n,\nu_n).
\end{equation}
Choose now $\nu_n$ such that 
\[
\sup_{n\in\N}\Entn(\nu_n)+W_2(\nu_n,\delta_{\bar x_n})<\infty
\]
(it is easy to see that this choice is possible) and recall \eqref{eq:aereo} to deduce that $\sup_n\Entn(\mu_n)<\infty$. Recalling \eqref{eq:aereo} again and using the identity
\begin{equation}
\label{eq:rer}
\Entn(\mu_n)=\Enttn(\mu_n)-\Co\int \sfd^2(\cdot,\bar x_n)\,\d\mu_n-\log \z_n,
\end{equation}
we obtain that $\sup_n\Enttn(\mu_n)<\infty$. Proposition \ref{prop:enttight}  then ensures that $(\mu_n)$ is tight and hence by Theorem \ref{thm:prok} that it has a subsequence weakly converging to some $\mu_\infty$ with (by \eqref{eq:joint}) $\Entti(\mu_\infty)<\infty$. In particular $\mu_\infty\ll\tilde\mm_\infty\ll\mm_\infty$. Moreover, the uniform bound on the second moments of the $\mu_n$'s granted by  assumption \eqref{eq:aereo} ensures that $\mu_\infty\in\probt{X_\infty}$ and therefore by \eqref{eq:rer} with $n=\infty$  we deduce $\Enti(\mu_\infty)<\infty$.

If we further assume \eqref{eq:57}, so that in particular the full sequence $(\mu_n)$ $W_2$-converges to $\mu_\infty$, then passing to the limit in \eqref{eq:rer} gives $\Enti(\mu_\infty)\leq\limi_{n\to\infty}\Entn(\mu_n)$.  Choosing now in \eqref{eq:rer2} $\nu_n$ $W_2$-converging to $\mu_\infty$ 
  with $\lims_{n\to\infty}\Entn(\nu_n)\le \Enti(\mu_\infty)$ 
  as in \eqref{eq:glims}, since the triangle inequality yields 
  $W_2(\nu_n,\mu_n)\longrightarrow0$, we obtain from \eqref{eq:rer2} that
  \begin{displaymath}
    \lims_{n\to\infty}\Entn(\mu_n)\le 
    \lims_{n\to\infty}\Entn(\nu_n)\le \Enti(\mu_\infty).
  \end{displaymath}
  \noindent \textbf{ii)} In order to prove \eqref{eq:59}, without loss of generality we can assume that $\sup_n\C_n(f_n)<\infty$. Initially we will also suppose that 
  \begin{equation}
   \text{$f_n(x)\equiv 0$ for $\mm_n$-a.e. $x\in X_n$ with
  $\sfd(x,\bar x_n)>R$,}
\label{eq:143}
\end{equation}
for some $R>0$ and that $f_n\geq 0$ $\mm_n$-a.e. for every $n\in\bar\N$. 

Now observe that up to extract a subsequence we can assume that  $Z_n:=\int  f_n^2\,\d\mm_n\to Z_\infty$. If $Z_\infty=0$ then $f_\infty=0$ and there is nothing to prove. Thus we can assume $Z_\infty>0$. Put $\mu_n:=Z_n^{-1} f_n^2\mm_n$, recall  that the identity  \eqref{eq:slopefish} gives $|\rmD^-\Entn|^2(\mu_n)=8Z_n^{-1}\C_n(f_n)$, the uniform bound on $\C_n(f_n)$ and  notice that \eqref{eq:143} grants 2-uniform integrability of  $(\mu_n)$. Hence by  point (i) we deduce that there is a subsequence of  $(\mu_n)$, not relabeled, converging to some $\mu_\infty=Z_\infty^{-1}g_\infty\mm_\infty$ in the  $W_2$-Entropy sense. Choosing
  $\zeta(y,r)=\varphi(y)\sqrt {r^+}$ in \eqref{eq:56} 
  we obtain that $g_\infty= f_\infty^2$ and the convergence of $Z_n$ to
  $Z_\infty$ 
  yields \eqref{eq:48} with $p=2$, i.e. the $L^2$-strong convergence of $(f_n)$ to $f_\infty$.

Inequality \eqref{eq:59} then follows  by Proposition \ref{prop:glimisl} and identity \eqref{eq:slopefish} again:
\[
8Z_\infty^{-1}\C_\infty(f_\infty)=|\rmD^-\Enti|^2(\mu_\infty)\le \limi_{n\to\infty}|\rmD^-\Entn|^2(\mu_n)=
    \limi_{n\to\infty}8Z_n^{-1}\C_n(f_n).
\]
To drop the assumption on the positivity of the $f_n$'s, apply what we just proved to $(f_n^+)$ and $(f_n^-)$ to deduce that up to subsequences they $L^2$-strongly converge to some $f^+$, $f^-$ respectively and that
\begin{equation}
\label{eq:chaperta}
\C_\infty(f^\pm)\leq\limi_{n\to\infty}\C_n(f_n^\pm).
\end{equation}
Since we know that $(f_n)$ is $L^2$-weakly converging to $f_\infty$, we deduce that $f_\infty^\pm=f^\pm$ and the claim follows adding up the two inequalities in \eqref{eq:chaperta} and observing that 
\[
\C_n(f_n)=\C_n(f_n^+)+\C_n(f_n^-),\qquad\forall n\in\bar\N.
\]

  We consider now the general case,
  by showing how to remove 
  the auxiliary assumption \eqref{eq:143}. Let $\nchi:[0,\infty)\to[0,1]$ be a 1-Lipschitz function with compact support identically 1 on $[0,1]$  and for $R>0$ define
  the truncated functions
\[
f_{R,n}:=\nchi(\sfd_n(\cdot,\bar x_n)/R) f_n,\qquad\forall n\in\bar \N.
\]
  The locality property of the minimal weak upper gradient
  and the Leibniz  rule \eqref{eq:calcrule} yield
  \begin{displaymath}
    |\rmD f_{R,n}|_w\leq |\rmD f_n|_w+
    \frac 1R |f_n|\qquad \mm_n
    \text{-a.e..}
    \end{displaymath}
Squaring and integrating we deduce
\[
\C_n(f_{R,n})\leq\C_n(f_n)\Big(1+\frac1R\Big)+\frac S2\Big(\frac1R+\frac1{R^2}\Big),\qquad\forall n\in\N,
\]
where $S:=\sup_{n\in\N}\int f_n^2\,\d\mm_n<\infty$. Taking first the $\limi$ as $n\to\infty$ and then as $R\uparrow\infty$ we deduce
\[
\limi_{n\to\infty}\C_n(f_n)\geq\limi_{R\uparrow\infty}\limi_{n\to\infty}\C_n(f_{R,n}).
\]
Now observe that  
  the $f_{R,n}$'s are $L^2$-weakly converging to $f_{R,\infty}$ as $n\to\infty$ and clearly satisfy \eqref{eq:143}, so that by what we previously proved we get $\limi_{n\to\infty}\C_n(f_{R,n})\geq \C_\infty(f_{R,\infty})$ for every $R>0$. Eventually noticing that $f_{R,\infty}\to f_\infty$ in $L^2(X,\mm_\infty)$ and taking into account the $L^2(X,\mm_\infty)$-lower semicontinuity of $\C_\infty$ we obtain \eqref{eq:59}.

It remains to prove $L^2$-strong convergence under the additional assumptions \eqref{eq:142}. To this aim notice that for given $R>0$ the functions $(f_{R,n})$ satisfy the assumption \eqref{eq:143}, so by what we previously proved they $L^2$-strongly converge to $f_{R,\infty}$ as $n\to\infty$. Moreover, as already noticed,  $f_{R,\infty}\to f_\infty$ as $R\to\infty$ in $L^2(X,\mm_\infty)$ while the second in \eqref{eq:142} grants that 
\[
\lim_{R\uparrow\infty}\sup_{n\in\N}\int|f_n-f_{R,n}|^2\,\d\mm_n=0.
\] 
The conclusion follows.
\end{proof}
In the particular case of a single measure 
$\mm_n\equiv\mm$, $n\in \bar \N$, and a 
$\CD(K,\infty)$ space $(\Y ,\sfd)$,  
the previous Theorem provides a simple criterium
for the compact imbedding of 
the Sobolev space $W^{1,2}(\Y ,\sfd,\mm)$
in $L^2(\Y ,\mm)$. 

In order to make the connection more evident we  recall the useful notion of 
Logarithmic-Sobolev-Talagrand inequality:
\begin{definition}
  \label{def:newLSTI}
  We say that a pointed metric measure space
  $(X,\sfd,\mm,\bar x)$  
  satisfies 
  a weak 
  Logarithmic-Sobolev-Talagrand inequality 
  $\mathrm{wLSTI}(A,B)$ with constants $
  A,B\ge0$ if   we have
   \begin{equation}
    \label{eq:newdef}
    \Big(\int_X \sfd^2(x,\bar x)f^2(x)\,\d\mm\Big)^{1/2}
    \le A
    \|f\|_{L^2(X,\mm)}+ B\sqrt{\C(f)},\qquad\forall f\in L^2(X,\mm).
  \end{equation}
\end{definition}
We are calling \eqref{eq:newdef} a weak  Logarithmic-Sobolev-Talagrand inequality because it can be seen as a consequence of a combination of the log-Sobolev inequality and of the Talagrand inequality, as made precise by the following simple proposition:
\begin{proposition}
Let  $(X,\sfd,\mm)$ be a $\CD(K,\infty)$ space  with finite mass, put   $\mm^*:=\mm(X)^{-1}\mm$ and assume that for some $a,b\geq 0$ we have
  \begin{equation}
    \label{eq:146b}
    W_2(\mu,\mm^\star)\le 
    a+b |\rmD^-\Ent_{\mm^\star}|(\mu)\quad
    \forevery \mu\in \probt X,
    \end{equation}
Then for every $\bar x\in\supp(\mm)$ the space $(X,\sfd,\mm,\bar x)$ satisfies a   $\mathrm{wLSTI}(A,B)$  with 
\[
    A:=
    2a+2b
    \Big(\frac{\mm(B_{2}(\bar x))}{\mm(B_1(\bar x))}\Big)^{1/2}+2,
    \quad
    B:=\sqrt 8\,b.
\]
\end{proposition}
\begin{proof} 
We start proving that $\mm^*\in\probt X$.  Indeed,
choosing $\mu=Z_R^{-1}h_R^2(x)\mm^\star$ with $h_R(x):=
\big((2-\sfd(x,\bar x)/R)\land 1\big)\lor 0$ and
$Z_R=\int_X h_R^2\,\d\mm^*$ we
get
\begin{displaymath}
  |\rmD^-\Ent_{\mm^\star}|(\mu)=
  \frac2{\sqrt{Z_R}}\Big(\int_X \weakgrad {h_R}^2\,\d\mm^\star\Big)^{1/2}
  \le \frac 2{R\sqrt {Z_R}}
  \Big(\mm^\star(B_{2R}(\bar x))\Big)^{1/2},
\end{displaymath}
and therefore
\begin{displaymath}
  \Big(\int_X \sfd^2(x,\bar x)\,\d\mm^\star\Big)^{1/2}\le 
  W_2(\mm^\star,\mu)+W_2(\mu,\delta_{\bar x})\le 
  a+\frac {2b}{R\sqrt {Z_R}}
  \Big(\mm^\star(B_{2R}(\bar x))\Big)^{1/2}+2R.
\end{displaymath}
Since $Z_R\ge \mm^*(B_R(\bar x))$ we eventually get
\begin{equation}
  \label{eq:149}
  \Big(\int_X \sfd^2(x,\bar x)\,\d\mm^\star\Big)^{1/2}\le 
  a+\frac {2b}{R}
  \Big(\frac{\mm(B_{2R}(\bar x))}{\mm(B_R(\bar x))}\Big)^{1/2}+2R,
  \quad\forevery R>0.
\end{equation}
Now let $M:=\mm(X)$ and $\mm^*=M^{-1}\mm$, pick
  $f\in W^{1,2}(X,\sfd,\mm)$ 
 and let  $\int f^2\,\d\mm=M\,F^2$. Then
  \eqref{eq:slopefish} and
  $\mathrm{wLSTI}(a,b)$ yield
\[
    W_2((f/F)^2\,\mm^\star,\mm^\star)\le a+\frac{b}{\sqrt M}
    \sqrt{8\C(f/F)}
\]
  so that the triangle inequality for $W_2$ yields
\[
    \Big(\int_X \sfd^2(x,\bar x)(f/F)^2(x)\,\d\mm^\star\Big)^{1/2}
    \le 
    a+\frac{b}{\sqrt M} \sqrt{8\C(f/F)}+
    \Big(\int_X \sfd^2(x,\bar x)\,\d\mm^\star\Big)^{1/2}.
\]
Taking into account \eqref{eq:149} with $R=1$,  \eqref{eq:newdef} follows. 
\end{proof}
\begin{remark}\label{re:forwlsti}{\rm
Otto and Villani showed in 
\cite{Otto-Villani00} that in the smooth setting 
a $\CD(K,\infty)$ bound yields a Logarithmic-Sobolev inequality of constant
$\rho>0$ which in turn yields \eqref{eq:146b} with $a=0$ and $b=\rho^{-1}$ and hence 
the $\mathrm{wLSTI}(0,\sqrt8\rho^{-1})$. These implications have been later generalized to the non-smooth setting: see for instance the calculus tools developed in \cite{Ambrosio-Gigli-Savare-preprint11a} to see that $\CD(K,\infty)$ yields the log-Sobolev and  \cite{Gigli-Ledoux13} and references therein for the implication from log-Sobolev to the \eqref{eq:146b}.

In particular if $(X,\sfd,\mm)$ 
is a $\mathrm{CD}(K,\infty)$ 
space of finite mass with $K>0$ then 
it satisfies $\mathrm{wLSTI}(0,K^{-1})$.
It is also obvious that 
whenever $\mm$ has bounded support
of diameter $D$, then 
$\mathrm{wLSTI}(D,0)$ holds.
}\fr\end{remark}

The relevance of the $\mathrm{wLSTI}(A,B)$ in our discussion is due to the following fact:
\begin{proposition}[Compact embedding of $W^{1,2}$ in $L^2$]
  \label{cor:newcompact}
  Let $(X,\sfd,\mm,\bar x)$ be a $\mathrm{CD}(K,\infty)$ 
 p.m.m.~space satisfying 
   $\mathrm{wLSTI}(A,B)$ for some $A,B\geq0$
  (e.g.~when $K>0$ and $\mm$ is finite
  or $\mm$ has bounded support).

Then   the imbedding
  of 
  $W^{1,2}(\Y ,\sfd,\mm)$ in $L^2(\Y ,\mm)$
  is compact.
\end{proposition}
\begin{proof}
The trivial inequality $\int_{\Y\setminus B_R(\bar x)}f^2\,\d\mm\leq\frac1{R^2}\int_\Y\sfd^2(\cdot,\bar x)f^2\,\d\mm$ and \eqref{eq:newdef} ensure that
whenever a sequence $f_n$ is bounded in $W^{1,2}(X,\sfd,\mm)$ it holds
\[
  \lim_{R\to\infty} \sup_{n} \int_{\Y \setminus B_{R}(\bar x)}
  f_n^2\,\d\mm=0.
\]
Hence  the conclusion comes from point (ii) of Theorem \ref{thm:weaktostrong}.
\end{proof}

\subsection{Mosco-convergence of Cheeger energies}

We apply the previous results to study the variational convergence of
the Cheeger energies.

 Even if the functionals are not imbedded in a
common Hilbert space, it is natural to call the property below
``Mosco''-convergence (see \cite[\S~3.3]{Attouch84}).
It would also be possible to imbed all the domains of the functionals
in a common topological space, as
e.g.~\cite{Peletier-Savare-Veneroni12},
and state the results in terms of $\Gamma$-convergence,
or to adopt the general approach of \cite[Section 2.5]{Kuwae-Shioya03}.
Nevertheless, we think that 
the more direct formulation in terms of 
functions in varying $L^p$-spaces
would be simpler and still sufficient for the applications.

\begin{theorem}[Mosco-convergence of Cheeger energies]
\label{thm:Mosco} Let $\cX_n$, $n\in\N$, be a sequence of $\CD(K,\infty)$ spaces converging to a limit space $\cX_\infty$ in the pmG-sense. Then with the notation of Remark \ref{re:assnot} the following holds:
     \begin{itemize}
  \item\underline{Weak $\Gamma-\liminf$} For every sequence $n\mapsto f_n\in L^2(\Y ,\mm_n)$ $L^2$-weakly converging to  some   $f_\infty\in L^2(X,\mm_\infty)$ we have
    \begin{equation}
      \label{eq:53}
      \limi_{n\to\infty}\C_n(f_n)\ge \C_\infty(f_\infty).
    \end{equation}
  \item\underline{Strong $\Gamma-\limsup$}  For every $f_\infty\in L^2(\Y ,\mm_\infty)$ there exists a  sequence
    $n\mapsto f_n\in L^2(\Y ,\mm_n)$ $L^2$-strongly convergent to $f_\infty$ such that
    \begin{equation}
      \label{eq:54}
      \lim_{n\to\infty}\C_n(f_n)= \C_\infty(f_\infty).
    \end{equation}
  \end{itemize}
\end{theorem}
\begin{proof} Inequality  \eqref{eq:53} has already beed proved in point (ii) of 
 Theorem \ref{thm:weaktostrong}, so we consider \eqref{eq:54}. 
 
 By a simple diagonalization argument we see that it is
  sufficient to approximate functions $f_\infty$ in a set dense in
  energy. Furthermore, given that we are going to build an $L^2$-strongly converging sequence, by using the identity $\C(f)=\C(f^+)+\C(f^-)$ we can also reduce to the case of non-negative $f_\infty$.

Hence after a  truncation,
  localization and normalization arguments we can then assume that
  $f_\infty$ is essentially 
  bounded with bounded support and, setting $g_\infty:=f^2_\infty$, that
 $\mu_\infty:=g_\infty\mm_\infty$ is in $\probt \Y $.
  By \eqref{eq:slopefish} and the same argument of the proof
  or Theorem \ref{thm:weaktostrong} it is then sufficient to find a sequence
  \[
    \text{$n\mapsto \nu_n=g_n\mm_n$ $W_2$-converging to $\mu_\infty$,
  with $|\rmD^-\Entn|(\nu_n)\longrightarrow
  |\rmD^-\Enti|(\mu_\infty)$.}
  \]
  This is precisely the content of the second part of Corollary \ref{cor:moscoslope}, so the thesis is achieved.
\end{proof}
\begin{remark}{\rm By the stability of the $\CD(K,\infty)$ condition granted by Theorem \ref{thm:stabilityCD} and the Mosco-convergence of the Cheeger energies just proved  it is easy to deduce that the class of $\CD(K,\infty)$ spaces satisfying  $\mathrm{wLSTI}(A,B)$ is closed w.r.t.\ pmG-convergence.
}\fr\end{remark}

We provide two useful corollaries to Theorem \ref{thm:Mosco},
concerning the convergence of the resolvents and of the $L^2$
gradient flows associated to the Cheeger energies.
Both the results are well known in the case of functionals
defined in the same Hilbert space (see e.g.~\cite{Attouch84}),
so we only give a brief sketch of the proofs.

For every $\tau>0,n\in \bar \N$ we define the resolvent map
$J_{n,\tau}:L^2(\Y ,\mm_n)\to D(\C_n)$ as
\begin{equation}
  \label{eq:66}
  J_{n,\tau}(f):=(I+\tau\partial \C_n)^{-1}(f)=
  \argmin_{L^2(\Y ,\mm_n)}\Phi_{n,\tau}(\cdot;f)
\end{equation}
where 
\[
  \Phi_{n,\tau}(g;f):=\frac 1{2\tau}\int |g-f|^2\,\d\mm_n+
  \C_n(g)\quad\forevery g\in L^2(\Y ,\mm_n).
\]
\begin{corollary}
  Under the same assumptions of Theorem \ref{thm:Mosco},
  for every sequence $(f_n)$ $L^2$-strongly converging   to $f_\infty\in L^2(\Y,\mm_\infty)$  we have
  \begin{equation}
    \label{eq:67}
    J_{n,\tau}(f_n)\stackrel{L^2}\longrightarrow
    J_{\infty,\tau}(f_\infty),\qquad\text{and}\qquad
    \C_n(J_{n,\tau}(f_n))\longrightarrow
    \C_\infty(J_{\infty,\tau}(f_\infty)),   
  \end{equation}
for every $\tau>0$.
\end{corollary}
\begin{proof}
  Let $g_n=J_{n,\tau}(f_n)$, for every $n \in \bar{\N}$. Choosing 0 as competitor in the definition of $g_n$ we easily see that
  $\int g_n^2\,\d\mm_n\le 4 \int f_n^2\,\d\mm_n$. Hence  
  we can extract a subsequence (still denoted by $g_n$) $L^2$-weakly converging to some $g\in L^2(\Y ,\mm_\infty)$.  Applying the second part  of Theorem 
  \ref{thm:Mosco} 
  we can also find 
  a sequence $\tilde g_n\in D(\C_n)$ such that 
$    \Phi_{n,\tau}(\tilde g_n;f_n)\longrightarrow \Phi_{\infty,\tau}(g_\infty;f_\infty).
$
Passing to the limit in the inequalities
$\Phi_{n,\tau}(g_n,f_n)\le \Phi_{n,\tau}(\tilde g_n,f_n)$ 
thanks to the first part  of Theorem \ref{thm:Mosco} we find
that $g$ is a minimizer of $\Phi_{\infty,\tau}(\cdot;f_\infty)$,
so that $g=g_\infty$. Since the limit is unique, we conclude that
the whole sequence $g_n$ is $L^2$-weakly converging to $g_\infty$
and $\Phi_{n,\tau}(g_n,f_n)\to \Phi_{\infty,\tau}(g_\infty,f_\infty)$.
A further application of the lower semicontinuity results 
\eqref{eq:52} and \eqref{eq:53} provides the second in \eqref{eq:67} and 
\[
\int|g_n-f_n|^2\,\d\mm_n\to\int |g_\infty-f_\infty|^2\,\d\mm_\infty.
\]
Expanding the squares, using the $L^2$-strong convergence of the $f_n$'s and recalling \eqref{eq:weakstrongscalar} we deduce that  $\int|g_n|^2\,\d\mm_n\to\int |g_\infty|^2\,\d\mm_\infty$, i.e.\ the first in \eqref{eq:67}.
\end{proof}
\begin{theorem}[$L^2$-convergence of the Heat flows]
  \label{thm:L2conv}  Let $\cX_n$, $n\in\N$, be a sequence of $\CD(K,\infty)$ spaces converging to a limit space $\cX_\infty$ in the pmG-sense and adopt  the notation of Remark \ref{re:assnot}. Furthermore, 
  let $(\sfH^n)_{t\ge0}$ be the $L^2$-gradient flows of the 
  corresponding Cheeger energies as in \S~\ref{subsub:L2g}.
  
 Then for every sequence $n\mapsto\bar f_n\in L^2(\Y ,\mm_n)$  $L^2$-strongly converging to $\bar f_\infty\in L^2(X,\mm_\infty)$  we have
  \begin{equation}
    \label{eq:63}
    \sfH^n_t\bar f_n\stackrel{L^2}\longrightarrow\sfH^\infty_t\bar
    f_\infty
    \text{\quad strongly }
    \forevery t\ge0.
  \end{equation}
\end{theorem}
\begin{proof}
  Let us first suppose that $\C_n(\bar f_n)\le C<\infty$ for every
  $n\in \bar \N$. We denote by $J^k_{n,\tau}$ the
  iterated resolvent $(J_{n,\tau})^k$;
  uniform convergence estimates (see
  e.g.~\cite[Thm.~4.0.4]{Ambrosio-Gigli-Savare08})
  show that 
\[
    \int\Big|\sfH^n_t(\bar f_n)-J^k_{n,t/k}(\bar f_n)\Big|^2\,\d\mm_n\le 
    C\frac tk
    \quad\forevery n\in \bar\N, \ k\in \N.
\]
  Since the estimate is uniform w.r.t.~$n$ and 
  $J^k_{n,t/k}\bar f_n\stackrel{L^2}\longrightarrow
  J^k_{\infty,t/k}\bar f_\infty$ as $n\to\infty$ thanks to the
  previous Corollary, we easily get \eqref{eq:63}.
  
  We then use the $L^2$-contraction property of $(\sfH^n_t)_{t\ge0}$ and 
  the $\Gamma$-$\limsup$ estimate \eqref{eq:54} to extend
  the result to the general case (see a similar argument in the proof
  of
  Theorem \ref{thm:stabrgf} below).
\end{proof}
\begin{remark}
  \upshape
  In the case of $\RCD(K,\infty)$ spaces, when $(\sfH_t)_{t\ge0}$ are
  linear operators, Theorem \ref{thm:L2conv}
  can be directly deduced by the corresponding Wasserstein result, Theorem
  \ref{thm:stabgf1}, see \cite[\S 5.2]{Ambrosio-Gigli-Savare-preprint12}.
\end{remark}

\section{Stability, convergence and spectral properties for
  $\RCD(K,\infty)$ spaces.}

\subsection{The $\RCD(K,\infty)$ condition and its stability.}
In \cite{Ambrosio-Gigli-Savare-preprint11b} a study of $\CD(K,\infty)$ with linear heat flow has been initiated, the definition being the following:
\begin{definition}[$\RCD(K,\infty)$ spaces]
   $(X,\sfd,\mm,\bar x)$ is a $\RCD(K,\infty)$  \pmm~space if it 
  satisfies the $\CD(K,\infty)$ condition and the $W_2$-heat flow $(\h_t)$ (see Theorem \ref{thm:gf})
    is linear, i.e.~
\begin{equation}
\label{eq:rcdent}
\h_t(\alpha\mu+\beta\nu)=\alpha\h_t\mu+\beta\h_t\nu\quad
\forevery t\ge0,\ \mu,\nu\in D(\entv),\ \alpha,\beta\in [0,1],\ \alpha+\beta=1,
\end{equation}
or, equivalently, if the Cheeger energy $\C$ is a quadratic form in
$L^2(X,\mm)$, i.e. 
\begin{equation}
\label{eq:rcdch}
\C(f+g)+\C(f-g)=2\C(f)+2\C(g) \qquad\text{for every
$f,g\in L^2(X,\mm).$}
\end{equation}
\end{definition}
The acronym $\RCD$ stands for Riemannian Curvature Dimension, indeed it is well known that Finsler geometries are included in the class of $\CD(K,\infty)$ spaces, and that the heat flow on a smooth Finsler manifold is linear if and only if the manifold is Riemannian. Hence the idea behind the definition of the subclass $\RCD(K,\infty)$ of $\CD(K,\infty)$ spaces is somehow to isolate those spaces which have a `Riemannian' behavior, see \cite{Ambrosio-Gigli-Savare-preprint11b} and \cite{AGMRS12} for results in this direction. We remark that in \cite{Ambrosio-Gigli-Savare-preprint11b} the $\RCD(K,\infty)$ was required both by asking the linearity of the heat flow as above, and enforcing a bit the $\CD(K,\infty)$ with the requirement that the entropy was $K$-geodesically convex along a family of geodesics larger than the one appearing in Definition \ref{def:cd}. It has been later understood in \cite{AGMRS12} that this enforcement  of the $\CD(K,\infty)$ assumption was actually unnecessary.

As for the $\CD(K,\infty)$ condition, the $\RCD(K,\infty)$ one is  invariant with respect to 
  isomorphism of \pmm~spaces. It is also stable w.r.t.~$\pGw$ convergence, as we show now. Such stability property can be achieved both passing to the limit in \eqref{eq:rcdch} thanks to the Mosco-convergence of the Cheeger energies that we just proved, or passing to the limit in \eqref{eq:rcdent} with an argument based on the properties of the relative entropy. We think that both approaches are interesting and thus we are going to propose both proofs.
  \begin{theorem}[Stability of $\RCD(K,\infty)$ under  $\pGw$-convergence]
\label{thm:stabrcd}   Let $\cX_n$, $n\in\N$, be a sequence of p.m.m.\ spaces  converging to a limit space $\cX_\infty$ in the pmG-sense. Assume that $\cX_n$ is an $\RCD(K,\infty)$ space for every $n\in\N$. Then $\cX_\infty$ is $\RCD(K,\infty)$ as well.
\end{theorem}
\begin{proof} We shall adopt the notation of Remark \ref{re:assnot}.\\
\noindent\textbf{Proof via the use of the Cheeger energy.} By assumption we know that for every $n\in\N$ it holds
\begin{equation}
\label{eq:rcdch2}
\C_n(f_n+g_n)+\C_n(f_n-g_n)=2\C_n(f_n)+2\C_n(g_n) \qquad\text{for every
$f_n,g_n\in L^2(X,\mm_n).$}
\end{equation}
Now pick $f_\infty,g_\infty\in L^2(X,\mm_\infty)$ and use the second part  of Theorem \ref{thm:Mosco} to find sequences $(f_n)$, $(g_n)$ $L^2$-strongly converging to $f_\infty,g_\infty$ respectively such that 
\[
\lim_{n\to\infty}\C_n(f_n)=\C_\infty(f_\infty),\qquad\qquad\lim_{n\to\infty}\C_n(g_n)=\C_\infty(g_\infty).
\]
Notice that $n\mapsto f_n\pm g_n$ $L^2$-strongly converges to $f_\infty\pm g_\infty$, use the first part of Theorem \ref{thm:Mosco} and pass to the limit in \eqref{eq:rcdch2} to get
\[
\C_\infty(f_\infty+g_\infty)+\C_\infty(f_\infty-g_\infty)\leq 2\C_\infty(f_\infty)+2\C_\infty(g_\infty).
\]
Repeat the argument with $f_\infty':=f_\infty+g_\infty$ and $g_\infty':=f_\infty-g_\infty$ in place of $f_\infty,g_\infty$ and recall that the Cheeger energy is 2-homogeneous to get the other inequality and the conclusion.

\noindent\textbf{Proof via the use of the relative entropy.} We denote by $\h_n$, $n\in \bar\N$, 
the $W_2$-gradient flow of $\Entn $ in $\probt \Y$.
Let $\mu^i_{\infty}\in D(\Enti)$, $i=0,1$, and $\mu^\alpha_\infty:=
(1-\alpha)\mu^0_\infty+\alpha\mu^1_\infty$, $\alpha\in(0,1)$.

The $\Glims$ inequality in Proposition \ref{prop:GammaConvergence} 
provides sequences $n\mapsto \mu^i_n\in D(\Entn)$ such that
\[
\lim_{n\to\infty}\Entn(\mu^i_n)=\Enti(\mu^i_\infty),\qquad
\mu^i_n\stackrel{W_2}\longrightarrow \mu^i_\infty
\qquad
i=0,1.
\]
%
%


By the general convergence result
Theorem \ref{thm:stabgf1} we know that 
\[
  \h_{n,t}(\mu^\alpha_n)=(1-\alpha)\h_{n,t}(\mu^0_{n})+\alpha\h_{n,t}(\mu^1_{n})
  \stackrel{W_2}\longrightarrow
  (1-\alpha)\h_{\infty,t}(\mu^0_{\infty})+\alpha\h_{\infty,t}(\mu^1_{\infty}).
\]
Hence our thesis follows if we show that
\begin{equation}
  \label{eq:50}
  \lim_{n\to\infty}\Entn(\mu^\alpha_n)=\Enti(\mu^\alpha_\infty),\qquad
  \mu^\alpha_n\stackrel{W_2}\longrightarrow \mu^\alpha_\infty,
\end{equation}
so that Theorem \ref{thm:stabgf1} yields $\h_{n,t}(\mu^\alpha_t)\stackrel{W_2}\longrightarrow
\h_{\infty,t}(\mu^\alpha_\infty)$.
In order to prove \eqref{eq:50}
let us first observe that the 
$W_2$-convergence of $\mu^\alpha_n$ to $\mu^\alpha_\infty$ 
is a direct consequence of the convexity of the squared Wasserstein distance
$W_2^2$.
%

Since $\mu^\alpha_n\in D(\Entn )$ and
$\mu^0_n,\mu^1_n\ll\mu^\alpha_n$ with bounded density if $\alpha\in (0,1)$,
for every $n\in \bar\N$ 
we have
\[
\begin{split}
\Ent_{\mu^\alpha_n}(\mu^i_n)=\int\log\left(\frac{\d\mu^i_n}{\d\mu^\alpha_n}\right)\,\d\mu^i_n&=\int\log\left(\frac{\d\mu^i_n}{\d\mm_n}\right)-\log\left(\frac{\d\mu^\alpha_n}{\d\mm_n}\right)\,\d\mu^i_n\\
&=\Entn (\mu^i_n)-\int\log\left(\frac{\d\mu^\alpha_n}{\d\mm_n}\right)\,\d\mu^i_n.
\end{split}
\]
Taking appropriate convex combinations  we get
\[
(1-\alpha)\Ent_{\mu^\alpha_n}(\mu^0_n)+\alpha \Ent_{\mu^\alpha_n}(\mu^1_n)=
(1-\alpha) \Entn(\mu^0_n)+\alpha\Entn(\mu^1_n)-\Entn(\mu^\alpha_n).
\]
Therefore from the $W_2$-convergence 
of $(\mu^i_{n}),( \mu^\alpha_n)$ to $\mu^i_\infty, \mu^\alpha_\infty$ respectively and the $\Glimi$ inequality in Proposition \ref{prop:GammaConvergence} we get
\[
\begin{split}
\limi_{n\to\infty}(1-\alpha)\Entn(\mu^0_{n})&+\alpha\Entn(\mu^1_{n})-\Entn(\mu^\alpha_n)=
\limi_{n\to\infty}(1-\alpha)\Ent_{\mu^\alpha_n}(\mu^0_{n})+\alpha \Ent_{\mu^\alpha_n}(\mu^1_{n})\\
&\geq (1-\alpha)\Ent_{\mu^\alpha_\infty}(\mu^0_\infty)+\alpha\Ent_{\mu^\alpha_\infty}(\mu^1_\infty)\\
&=(1-\alpha)\Enti(\mu^0_\infty)+\alpha\Enti(\mu^1_\infty)-\Enti(\mu^\alpha_\infty).
\end{split}
\]
This fact together with the assumption $\lim_n\Entn(\mu^i_n)=\Enti(\mu^i_\infty)$, $i=0,1$,
give
\[
\limi_{n\to\infty}-\Entn(\mu^\alpha_n)\geq -\Enti(\mu^\alpha_\infty).
\]
Since the other inequality in ensured by the $\Glimi$ part in Proposition \ref{prop:GammaConvergence}, the thesis is achieved.
\end{proof}
\subsection{Refined estimates on the convergence of the Heat flow.}
One of the main contributions of \cite{Ambrosio-Gigli-Savare-preprint11b} (see also \cite{AGMRS12}) is the proof that the linearity condition on the heat flow grants additional regularity properties for the flow itself. A crucial one is the following contractivity statement.
\begin{theorem}[$W_2$-contraction]\label{thm:kcontr}
Let $(X,\sfd,\mm,\bar x)$ be a $\RCD(K,\infty)$ space.
For every $\bar\mu,\bar\nu\in D(\entv)$ it holds
\begin{equation}
\label{eq:contr}
W_2(\h_t\bar\mu,\h_t\bar\nu)\leq \rme^{-Kt}W_2(\bar\mu,\bar\nu)\qquad\forevery t\geq 0.
\end{equation}
\end{theorem}
Notice that in a smooth setting, \eqref{eq:contr} is specific of Riemannian geometry, because Ohta-Sturm proved in \cite{Ohta-Sturm12} that no exponential contraction holds in $(\R^d,\|\cdot\|,\mathcal L^d)$ if the norm $\|\cdot\|$ does not come from a scalar product. 

A direct consequence of \eqref{eq:contr} is that the $W_2$-gradient flow $(\h_t)_{t\ge0}$
of the entropy can be extended from $D(\entv)$ to its $W_2$-closure, 
which is $\probt{\supp(\mm)}$ (i.e.~the subset of $\probt X$ made of measures $\mu$ such that $\supp(\mu)\subset\supp(\mm)$). In other words, a unique one parameter family of maps $\h_t:\probt{\supp(\mm)}\to\probt{\supp(\mm)}$ is defined by the following two properties:
\begin{equation}
\label{eq:contrh}
W_2(\h_t(\mu),\h_t(\nu))\leq \rme^{-Kt}W_2(\mu,\nu)\quad\forevery t\geq 0,\ \mu,\nu\in\probt{\supp(\mm)},
\end{equation}
and
\begin{quote}
$\forall\mu\in D(\entv)$
\textrm{ the curve $t\mapsto \h_t(\mu)$ is the gradient flow }\textrm{ of the entropy starting from $\mu$ in the sense of Definition \ref{def:gf}}.
\end{quote}
As for the heat flow and the $\CD(K,\infty)$, $\RCD(K,\infty)$ conditions, these maps are 
invariant w.r.t.~isomorphisms of \pmm~spaces.

A more explicit characterization  is provided by the following theorem, 
showing in particular that  $\h_t(\mu)\in D(\entv)$ for $t>0$ and every $\mu\in \probt{\supp(\mm)}$. 
The proof can be found, for instance, in the preliminary section 
of \cite{Ambrosio-Gigli-Savare-preprint11b}. \EEE
%
%

In the statement of the result and thereafter, the function ${\rm I}_K:\R^+\to\R^+$ is defined as 
\[
{\rm I}_K(t):=\int_0^t\rme^{Ks}\,\d s=
\begin{cases}
  \frac{\rme^{Kt}-1}{K}&\text{if }K\neq0\\
  t&\text{if }K=0.
\end{cases}
\] 
\begin{theorem}[A priori estimates]\label{thm:apriori}
Let $(X,\sfd,\mm)$ be a $\RCD(K,\infty)$ space. Then for every 
$\mu,\nu\in\probt{\supp(\mm)}$ and any $t>0$ it holds
\begin{equation}
\label{eq:apriori}
{\rm I}_K(t)\entv(\h_t(\mu))+\frac{{\rm I}_K(t)^2}{2}|\Dm^-\entv|^2(\h_t(\mu))\leq {\rm I}_K(t)\entv(\nu)+\frac12 W_2^2(\nu,\mu).
\end{equation}
\end{theorem}
As a direct consequence of \eqref{eq:apriori}, we have the following  a priori control on the entropy and its slope along the flow, which we state and prove only for $t$ close to 0, which is the regime we will need later.
\begin{corollary}\label{cor:apriori}
Let $(X,\sfd,\mm,\bar x)$ be a $\RCD(K,\infty)$ \pmm~space,
 $\Co>1+K_-$ and $\z,\tilde\mm$ as in \eqref{eq:expcontr}.
For every $\mu\in\probt{\supp(\mm)}$ and $t\in[0,\frac{1}{8\Co}]$ it holds
\begin{align}
\label{eq:apriorient}
{\rm I}_K(t)\entv(\h_t(\mu))&\leq -{\rm I}_K(t)\log \z+\big(1-\Co\, {\rm I}_K(t)\big)\int \sfd^2(\cdot,\bar x)\,\d\tilde\mm+\int \sfd^2(\cdot,\bar x)\,\d\mu\\
\label{eq:apriorislope}
\frac{{\rm I}_K(t)^2}{2}|\Dm^-\entv|^2(\h_t(\mu))&\leq\big(1+3\Co\, {\rm I}_K(t) \rme^{-2Kt}\big) \int \sfd^2(\cdot,\bar x)\,\d\mu + {\sf c}(K,t)\int \sfd^2(\cdot,\bar x)\,\d\tilde\mm,
\end{align}
where ${\sf c}(K,t)$ is given by ${\sf c}(K,t):=1-\Co\, {\rm I}_K(t)+3\Co {\rm I}_K(t)  (4\Co t+(1+\rme^{-Kt})^2)$.
\end{corollary}
\begin{proof}
The bound \eqref{eq:apriorient} follows plugging $\nu:=\tilde\mm$ in \eqref{eq:apriori}, neglecting the (non negative) term in the squared slope and using the identity
\begin{equation}
\label{eq:entt}
\entv(\tilde\mm)=\int \log(\z^{-1}\rme^{-\Co\sfd^2(\cdot,\bar x)})\,\d\tilde\mm=-\log \z-\Co W_2^2(\tilde\mm,\delta_{\bar x}).
\end{equation}
For \eqref{eq:apriorislope} we argue as follows. Let $(\nu_t)$ be the gradient flow of $\entv$ starting from $\tilde\mm$ and notice that
\[
\begin{split}
W_2(\h_t(\mu),\delta_{\bar x})&\leq W_2(\h_t(\mu),\nu_t)+W_2(\nu_t,\tilde\mm)+W_2(\tilde\mm,\delta_{\bar x})\\
&\leq \rme^{-Kt}W_2(\mu,\tilde\mm)+\int_0^t|\dot\nu_s|\,\d s+W_2(\tilde\mm,\delta_{\bar x})\\
&\leq \rme^{-Kt}W_2(\mu,\delta_{\bar x})+\sqrt{t\int_0^t|\dot\nu_s|^2\,\d s}+W_2(\tilde\mm,\delta_{\bar x})(1+\rme^{-Kt}).
\end{split}
\]
Squaring, using the trivial inequality $(a+b+c)^2\leq 3a^2+3b^2+3c^2$ and the a priori estimate \eqref{eq:speedcontr} (here we use that $t\in[0,\frac1{8\Co}]$) we get
\[
\begin{split}
\frac13W_2^2(\h_t(\mu),\delta_{\bar x})\leq \rme^{-2Kt}W_2^2(\mu,\delta_{\bar x})+ 4t\entv(\tilde\mm)+W_2^2(\tilde\mm,\delta_{\bar x})\big(8\Co t+(1+\rme^{-Kt})^2\big)+4 t\log \z.
\end{split}
\] 
Taking into account the identity \eqref{eq:entt}, this bound reduces to
\[
\frac13W_2^2(\h_t(\mu),\delta_{\bar x})\leq \rme^{-2Kt}W_2^2(\mu,\delta_{\bar x})+W_2^2(\tilde\mm,\delta_{\bar x})\big(4\Co t+(1+\rme^{-Kt})^2\big).
\]
Plugging this bound into \eqref{eq:chiave} and using the fact that $\Entt\geq 0$ we obtain
\[
\entv(\h_t(\mu))\geq -\log \z-3\Co \rme^{-2Kt}W_2^2(\mu,\delta_{\bar x})-3\Co W_2^2(\tilde\mm,\delta_{\bar x})\big(4\Co t+(1+\rme^{-Kt})^2\big).
\]
The conclusion follows using this inequality in \eqref{eq:apriori}, choosing $\nu:=\tilde\mm$ and using again \eqref{eq:entt}.
\end{proof}
\begin{remark}\label{re:stillgf}{\rm
Corollary \ref{cor:apriori} grants that for every
$\mu\in\probt{\supp(\mm)}$, the curve $t\mapsto \h_t(\mu)$ is
continuous and satisfies $\h_t(\mu)\in D(\entv)$ for any $t>0$. In particular, for any $t_0>0$ the curve $[0,\infty)\ni t\mapsto \h_{t+t_0}(\mu)$ is the gradient flow of $\entv$ starting from $\h_{t_0}(\mu)$ in the sense of Definition \ref{def:gf}.
For these reasons, we still call the curve $[0,\infty)\ni t\mapsto \h_t(\mu)$ the gradient flow of $\entv$ starting from $\mu$ even for $\mu\in \probt{\supp(\mm)}\setminus D(\entv)$.
}\fr\end{remark}
The a priori estimates in Corollary \ref{cor:apriori} and the contraction property \eqref{eq:contrh} allow to refine, in the case of $\RCD(K,\infty)$ spaces, the convergence result on the heat flow given in Theorem \ref{thm:stabgf1}. Notice indeed that in the statement below only $W_2$-convergence of the initial data is required, without any assumption on the behavior of the entropies.

\begin{theorem}[Improved stability of the heat flow]\label{thm:stabrgf} Let $\cX_n$, $n\in\N$, be a sequence of $\RCD(K,\infty)$ spaces converging to a limit space $\cX_\infty$ in the pmG-sense and adopt  the notation of Remark \ref{re:assnot}.
    
    If $(\bar\mu_n)\subset \probt X$ satisfies
    %
\[
      \bar\mu_n\stackrel{W_2}\longrightarrow \bar\mu_\infty\in 
      \probt{\supp(\mm_\infty)},
\]
  then
  the solutions $\mu_{n,t}=\h_{n,t}(\bar\mu_n)$, $t\ge0$,
  of the $W_2$-gradient flows of $\Entn$ satisfy
    all the properties {\upshape (\ref{subeq:conclusions}a,b,c,d)}.

\end{theorem}
\begin{proof} 
Let $k\mapsto\mu^k_\infty\subset D(\Enti)$ be a sequence such that
\[
\lim_{k\to\infty}\Enti(\mu^k_\infty)=\Enti(\bar\mu_\infty),\qquad\qquad\lim_{k\to\infty} W_2(\mu^k_\infty,\bar\mu_\infty)=0.
\]
For every $k\in\N$, use the $\Glims$ part of Proposition \ref{prop:GammaConvergence} to find a sequence $n\mapsto \mu_{n}^k\in\probt \Y $ such that
\begin{equation}
\label{eq:k}
\lim_{n\to\infty}\Entn(\mu^k_n)=\Enti(\mu^k_\infty),\qquad\qquad\lim_{n\to\infty} W_2(\mu^k_n,\mu^k_\infty)=0.
\end{equation}
Taking \eqref{eq:contrh} into account we get
\begin{equation}
\label{eq:aperto}
\begin{split}
W_2&\big(\mu_{\infty,t},\mu_{n,t}\big)\\
&\leq W_2\big(\mu_{\infty,t},\h_{\infty,t}(\mu^k_\infty)\big)+W_2\big(\h_{\infty,t}(\mu^k_\infty),\h_{n,t}(\mu_n^k)\big)+W_2\big(\h_{n,t}(\mu^k_n),\mu_{n,t})\big)\\
&\leq \rme^{-Kt}W_2(\bar\mu_\infty,\mu^k_\infty)+W_2\big(\h_{\infty,t}(\mu^k_\infty),\h_{n,t}(\mu_n^k)\big)+\rme^{-Kt}W_2(\mu^k_n,\bar\mu_n)\\
&\leq \rme^{-Kt}\big(2W_2(\bar\mu_\infty,\mu^k_\infty)+W_2(\bar\mu_\infty,\bar\mu_n)+W_2(\mu^k_\infty,\mu^k_n)\big)+W_2\big(\h_{\infty,t}(\mu^k_\infty),\h_{n,t}(\mu_n^k)\big).
\end{split}
\end{equation}
The choices \eqref{eq:k} and the assumption $\mu^k_\infty\in D(\Enti)$ ensure that the hypotheses of Theorem \ref{thm:stabgf1} are fulfilled with $\mu^k_n$ in place of $\bar\mu_n$ and $\mu^k_\infty$ in place of $\bar\mu_\infty$. Hence it holds $\lim_{n\to\infty}W_2\big(\h_{n,t}(\mu_n^k),\h_{\infty,t}(\mu^k_\infty)\big)=0$, so that passing to the limit first as $n\to\infty$ then as $k\to\infty$ in \eqref{eq:aperto} we get
\[
\lim_{n\to\infty}W_2\big(\mu_{\infty,t},\mu_{n,t}\big)=0\qquad\forevery t\geq 0.
\]

Now fix $\eps\in(0,\frac1{8\Co}]$ and notice that from $W_2\big(\mu_{\infty,\eps},\mu_{n,\eps}\big)\to 0$ and the $\Glimi$ part of Proposition \ref{prop:GammaConvergence} we get $\limi_{n\to\infty}\Entn(\mu_{n,\eps})\geq \Enti(\mu_{\infty,\eps})$. We claim that $\lims_{n\to\infty}\Entn(\mu_{n,\eps})\leq \Enti(\mu_{\infty,\eps})<\infty$ as well. The fact that $\mu_{\infty,\eps}\in D(\Enti)$ follows from \eqref{eq:apriorient}. Use the $\Glims$ part of Proposition \ref{prop:GammaConvergence} to find a sequence $(\nu_n)\subset\probt \Y $ such that 
\[
\lim_{n\to\infty}\Entn(\nu_n)=\Enti(\mu_{\infty,\eps}),\qquad\qquad\lim_{n\to\infty} W_2(\nu_n,\mu_{\infty,\eps})=0.
\]
From \eqref{eq:apriorislope} we get  $\sup_n|\Dm^-\Entn|(\mu_{n,\eps})=:S<\infty$, so that  formula \eqref{eq:repslope} gives
\[
\Entn(\mu_{n,\eps})\leq \Entn(\nu_n)+SW_2(\nu_n,\mu_{n,\eps})-\frac K2W_2^2(\nu_n,\mu_{n,\eps}).
\]
Letting $n\to\infty$ and using the fact that $W_2(\nu_n,\mu_{n,\eps})\to 0$ we get $\lims_{n\to\infty}\Entn(\mu_{n,\eps})\leq \Enti(\mu_{\infty,\eps})<\infty$, as claimed.

In summary, we proved that it holds
\[
\lim_{n\to\infty}\Entn(\mu_{n,\eps})=\Enti(\mu_{\infty,\eps}),\qquad\qquad\lim_{n\to\infty} W_2(\mu_{n,\eps},\mu_{\infty,\eps})=0,
\]
therefore Theorem \ref{thm:stabgf1} is applicable with $\mu_{n,\eps}$ in place of $\bar\mu_n$ and $\mu_{\infty,\eps}$ in place of $\bar\mu_\infty$. The conclusion comes from the arbitrariness of $\eps\in[0,\frac1{8\Co}]$.
\end{proof}

\subsection{Spectral convergence}

Since in a $\RCD(K,\infty)$ space $(X,\sfd,\mm)$ 
the Laplacian
$\Delta_{\sfd,\mm}$ (see \S~\ref{subsub:L2g}) 
is a selfadjoint linear closed operator in $L^2(X,\mm)$,
it would be interesting to study the behavior of its
spectrum w.r.t.~$\pGw$-convergence. 
We consider here the case 
of a sequence 
$\pmmXclassa n\in \X$, $n\in \bar\N$, 
of $\RCD(K,\infty)$ spaces
satisfying a uniform
weak Logarithmic-Sobolev-Talagrand inequality
$\mathrm{wLSTI(A,B)}$, so that in particular
\[
  W^{1,2}(X_n,\sfd_n,\mm_n) \quad
  \text{is compactly imbedded
in\quad $L^2(X_n,\mm_n)$\quad for every $n\in \bar \N$,}
\]
by Proposition 
 \ref{cor:newcompact}.
Let $\Delta_n:=\Delta_{\sfd_n,\mm_n}$ be
the corresponding Laplace operators;
since the resolvent $(I+\tau\Delta_n)^{-1}$ as in \eqref{eq:66} 
are compact operators, we can find an enumeration
\newcommand{\egnv}[2]{\lambda_{#1}(\Delta_{#2})}
\newcommand{\egnvJ}[2]{\lambda_{#1}(J_{#2})}
\[
  0\le \egnv 1n\le \egnv 2n\le \cdots\le \egnv kn\le\cdots,\qquad
  k\in \N,
  \]
of the eigenvalues of $\Delta_n$, 
each repeated according to its own multiplicity. 
The eigenvalues are invariant under isomorphisms of \pmm~spaces,
so $\egnv kn$ only depends on the equivalence class $\cX_n$ of the space $\pmmXa n$.
\begin{theorem}[Convergence of the eigenvalues of the Laplace operator]
  \label{thm:Lap-stab}
Let $\cX_n$ be a sequence of $\RCD(K,\infty)$ spaces converging to a limit $\cX_\infty$ space in the pmG-sense. Assume that  for some $a,b\geq 0$ $\cX_n$ satisfies the 
  $\mathrm{wLSTI}(A,B)$
  according to Definition \ref{def:newLSTI} for every $n\in\N$
  (in particular if $K>0$ and the masses are finite or
  $\mathrm{diam}\big(\supp(\mm_n)\big)$ 
  is uniformly bounded as in Remark \ref{re:forwlsti}).
  
  Then 
\[
    \lim_{n\to\infty}\egnv kn=\egnv k\infty\quad\forevery k\in \N.
\]
\end{theorem}
\begin{proof}
  The proof is a typical  application of the Mosco-convergence result
  proven in Theorem \ref{thm:Mosco}: the only difference with respect
  to the canonical setting concerns the fact that the domains of the 
  energies are not imbedded in a common Hilbert space. We shall adopt the notation of Remark \ref{re:assnot}.

  If $V$ is a subspace of $L^2(\Y ,\mm_n)$, we denote by
  $B(V)$ the closed unit ball in $V$ w.r.t.\ the $L^2(\Y ,\mm_n)$-norm.
  We call $\cV_{k,n}$ the collection of all the 
  $k$-dimensional subspaces 
  contained in $D(\C_n)$, 
  and we recall the minimum-maximum principle characterizing the 
  eigenvalues
  \begin{equation}
    \label{eq:71}
    \egnv kn=\min_{V\subset \cV_{k,n}} \max_{f\in B(V)}
    \cE_n(f,f),
  \end{equation}
  where $\cE_n:D(\C_n)\times D(\C_n)\to\R$ 
  is the symmetric bilinear form associated to $2\C_n$.
  Let us fix $k\in \N$ and first prove that 
    \begin{equation}
      \label{eq:73}
    \lims_{n\to\infty}\egnv kn\le \egnv k\infty.
  \end{equation}
  Let $V_\infty\in \cV_{k,\infty}$ be 
  a vector space realizing the minimum in \eqref{eq:71} and 
  let $(f_{i,\infty})\in \cV_{k,\infty}$, $i=1,\ldots,k$, be an { $L^2(X,\mm_n)$-}orthonormal system of $V_\infty$, so that
\[
    \egnv k\infty\ge \cE_\infty(f_{i,\infty},f_{i,\infty})\quad\forall i=1,\ldots,k \quad \& \quad { \cE_\infty(f_{i,\infty},f_{j,\infty})=0\quad\forall i\neq j=1,\ldots,k}.
\]
  By Theorem \ref{thm:Mosco} we find 
  $k$ sequences $n\mapsto f_{i,n}\in D(\C_n)$, $i=1,\ldots,k$,
  strongly converging 
  in the $L^2$-sense and in Cheeger energy  to $f_{i,\infty}$. It is obviously not restrictive to assume that $   \int f_{i,n}^2\,\d\mm_n=1$ for every $n\in\N$ and $i=1,\ldots,k$, furthermore  by \eqref{eq:71} we thus know that for every constant 
  $\eps>0$ there exists $n_\eps\in \N$ sufficiently big so that for every $n>n_\eps$ we have
  \begin{displaymath}
    \Big|\int f_{i,n}f_{j,n}\,\d\mm_n\Big| { + \big|\cE_n(f_{i,n},f_{j,n})\big|}\le \eps,\quad\forall i\neq j,\qquad
    \cE_n(f_{i,n},f_{i,n})\le \egnv k\infty+\eps,
    \quad \forall i.
    \end{displaymath}
  Setting $V_{k,n}=\mathrm{span}(f_{1,n},\ldots f_{k,n})$ we easily
  find that for every $\aalpha=(\alpha_1,\ldots,\alpha_k)\in \R^k$
  \begin{displaymath}
    \int \Big(\sum_{i=1}^k\alpha_i f_{i,n}\Big)^2\,\d\mm_n=
    |\aalpha|^2+\sum_{i\neq j}\alpha_i\alpha_j \int
    f_{i,n}f_{j,n}\,\d\mm_n
    \ge (1-\eps)|\aalpha|^2
  \end{displaymath}
  so that $\mathrm{dim}(V_{k,n})=k$ and every
  $f\in B({ V_{k,n}})$ can be represented as $f=\sum_{i=1}^k \alpha_if_{i,n}$
  with $|\aalpha|^2\le (1-\eps)^{-1}$. 
  Using this linear decomposition, we get 
  \begin{displaymath}
    \cE_n(f,f)=\sum_{i,j=1}^k\alpha_i\alpha_j\cE_n(f_{i,n},f_{j,n})\le 
    (\egnv k\infty+2 \eps)|\aalpha|^2\le (1-\eps)^{-1}(\egnv
    k\infty+2 \eps)
  \end{displaymath}
  for every $f\in B({ V_{k,n}})$: thus $\egnv kn\le (1-\eps)^{-1}(\egnv
    k\infty+2\eps)$ if $n\ge n_\eps$. Since $\eps>0$ is arbitrary, 
    this proves \eqref{eq:73}.

    Let us now prove the $\liminf$ inequality
\[
    \limi_{n\to\infty}\egnv kn\ge \egnv k\infty.
\] 
  By \eqref{eq:71}  for every $n\in \N$ 
  we can find a space $V_{k,n}\subset D(\C_n)$ 
  generated by the { $L^2(X,\mm_n)$}-orthonormal system 
  $(f_{i,n})$, $i=1,\ldots,k$, such that
  for every $n\in \N$ 
  and for every linear combination
  \begin{equation}
    \label{eq:148}
    f_{\saalpha,n}:=\sum_{i=1}^k \alpha_i
      f_{i,n}\quad\text{with }
      \aalpha=(\alpha_i)_i\in \R^k,\quad
      |\aalpha|^2=\sum_{i=1}^k\alpha_i^2=1,
  \end{equation}
  we have
  \begin{equation}
    \label{eq:72bis}
    \int f_{i,n}f_{j,n}\,\d\mm_n=\delta_{ij},
    \quad
    \egnv kn\ge \cE_n(f_{\saalpha,n},f_{\saalpha,n}).
  \end{equation}
  The previous estimate shows in particular that
  $\sup_{\saalpha,n}\cE_n(f_{\saalpha,n},f_{\saalpha,n})<\infty$:
  we can then extract an increasing 
  subsequence $n'$  
  such that 
  for every $\aalpha\in \R^k$ with $|\aalpha|=1$
  \begin{equation}
    \label{eq:150}
    \lambda_k(\Delta_n)\longrightarrow \lambda,
    \quad
    f_{\saalpha,n'}\longrightarrow f_{\saalpha,\infty} =
    \sum_{i=1}^k\alpha_i f_{i,\infty}
    \quad\text{weakly,}    
    \quad
    \cE_\infty(f_{\saalpha,\infty},f_{\saalpha,\infty})\le \lambda,
  \end{equation}
  where for the last inequality we applied \eqref{eq:53}.
  
  Applying ii) of Theorem \ref{thm:weaktostrong}
  thanks to 
  the uniform estimate 
  provided by the  $\mathrm{wLSTI}$,   we can reinforce the weak convergence of $n\mapsto f_{i,n}$
  to a strong convergence in $L^2$-sense, so that we can pass to the
  limit
  in the identities of \eqref{eq:72bis} and prove that
  $(f_{i,\infty})$ is an orthonormal system generating a $k$-dimensional
  subspace $V_{k,\infty}$.
  Every $f\in B({V_{k,\infty}})$ can then be
  represented as $f_{\saalpha,\infty}$ 
  as in \eqref{eq:148} 
so that  
  \eqref{eq:150} 
  yields 
  $\cE_\infty(f,f)\le \lambda$ for every $f\in B({V_{k,\infty}})$. Using
  \eqref{eq:71} with $n=\infty$ we conclude $\egnv k\infty\le \lambda$, as desired.  
\end{proof}

\begin{remark} \label{rk:Shyoia}
  \upshape
  When the pmG-convergent sequence of $\RCD(K,\infty)$ spaces $\cX_n$ 
  does not satisfy a uniform $\mathrm{wLSTI}$ condition, 
  the Mosco-convergence result of Theorem \ref{thm:Mosco} 
  is still sufficient to
  prove
  that every point $\lambda$ in the spectrum $\sigma(\Delta_\infty)$ is the 
  limit of points $\lambda_n\in \sigma(\Delta_n)$: one can
  argue as in the proof of \cite[Proposition 2.5]{Kuwae-Shioya03}
  (see also Remark \ref{rem:KuShi}), where 
  it is shown that (a suitably adapted) 
  Mosco-convergence of Dirichlet forms always implies
  lower semicontinuity of the spectra.

  The uniform $\mathrm{wLSTI}$ condition provides a further
  asymptotical compactness analogous to \cite[Definition
  2.12]{Kuwae-Shioya03}:
  in this case 
  Theorem \ref{thm:Lap-stab} could also be proven
  as for Theorem 2.6 in \cite{Kuwae-Shioya03}.
\fr\end{remark}

\small
\bibliographystyle{siam}
\bibliography{bibliografia-stability}

\end{document}